\documentclass[11pt,a4paper]{article}
\usepackage[T1]{fontenc}
\usepackage[utf8]{inputenc}
\usepackage{lmodern}
\usepackage{geometry}
\usepackage{amsmath,amssymb,amsfonts,latexsym}
\usepackage{amsthm}
\usepackage{mathrsfs}
\usepackage{booktabs}
\usepackage{float}

\usepackage[colorlinks=true,linkcolor=blue,citecolor=blue,urlcolor=blue]{hyperref}

\numberwithin{equation}{section}
\theoremstyle{plain}
\newtheorem{theorem}{Theorem}[section]
\newtheorem{lemma}[theorem]{Lemma}
\newtheorem{proposition}[theorem]{Proposition}
\newtheorem{corollary}[theorem]{Corollary}
\theoremstyle{definition}

\newtheorem{definition}[theorem]{Definition}

\theoremstyle{remark}
\newtheorem{remark}[theorem]{Remark}

\newcommand{\R}{\mathbb{R}}

\newcommand{\Hh}{\mathbb{H}}
\newcommand{\B}{\mathbb{B}}

\newcommand{\gradH}{\nabla_{\Hh}}

\title{Topological Effects on Bubbling in the Critical Dirichlet Problem on Hyperbolic Domains: $3\leq N \leq5$}
\author{Tapendu Rana, Michael Ruzhansky, Zhipeng Yang\thanks{Corresponding author: yangzhipeng326@163.com.}}
\date{}

\AtEndDocument{%
	\par
	\bigskip
	\bigskip
	\noindent
	\textbf{Tapendu Rana:}\\[0.2em]
	\textsc{Department of Mathematics: Analysis, Logic and Discrete Mathematics, Ghent University, Belgium}\\[0.3em]
	\textit{E-mail address}: \href{mailto:tapendurana@gmail.com}{\nolinkurl{tapendurana@gmail.com}}\\[1.5em]
	\noindent
	\textbf{Michael Ruzhansky:}\\[0.2em]
	\textsc{Department of Mathematics: Analysis, Logic and Discrete Mathematics, Ghent University, Belgium}\\
	\textsc{School of Mathematical Sciences, Queen Mary University of London, United Kingdom}\\[0.3em]
	\textit{E-mail address}: \href{mailto:michael.ruzhansky@ugent.be}{\nolinkurl{michael.ruzhansky@ugent.be}}\\[1.5em]
	\noindent
	\textbf{Zhipeng Yang:}\\[0.2em]
	\textsc{Department of Mathematics, Yunnan Normal University, Kunming, China}\\
	\textsc{Department of Mathematics: Analysis, Logic and Discrete Mathematics, Ghent University, Belgium}\\[0.3em]
	\textit{E-mail address}: \href{mailto:yangzhipeng326@163.com}{\nolinkurl{yangzhipeng326@163.com}}%
}

\begin{document}
	\maketitle
	
	\begin{abstract}
		Let \(\Omega\Subset\mathbb H^N\), \(N\in\{3,4,5\}\), be a bounded connected \(C^2\) domain. We prove that the pure critical Dirichlet problem
		\[
		-\Delta_{\mathbb H}u=u^{\frac{N+2}{N-2}}
		\quad\text{in }\Omega,
		\qquad
		u=0
		\quad\text{on }\partial\Omega
		\]
		admits a positive solution whenever \(H_d(\Omega;\mathbb F_2)\neq0\) for some \(1\le d\le N-1\). This gives a hyperbolic Bahri-Coron theorem for \(3\le N\le5\) under \(C^2\) boundary regularity. Under conformal reduction, the hyperbolic geometry produces a positive potential and leads to dimension-dependent bubbling mechanisms. In dimension three, the required energy drop follows from the balance between diagonal corrections and pair interactions at fixed large multiplicity. In dimensions four and five, it is obtained at the matched scale through a normalized defect estimate and an all-pairs source-transfer bound. A unified Thom-barycenter construction converts these analytic estimates into the topological contradiction. Thus nontrivial domain topology forces existence despite critical loss of compactness and the additional geometric potential.
	\end{abstract}
	
	

	\section{Introduction and main theorem}

	Let \(D\subset\mathbb R^N\), \(N\ge3\), be a smooth bounded domain and set
	\[
	2^*=\frac{2N}{N-2}.
	\]
	The critical Dirichlet equation
	\[
	-\Delta w=w^{2^*-1}\quad\text{in }D,\qquad
	w>0\quad\text{in }D,\qquad
	w=0\quad\text{on }\partial D
	\]
	lies at the endpoint of the Sobolev embedding \(H^1_0(D)\hookrightarrow L^{2^*}(D)\).  The embedding is continuous but not compact, and the best quotient on \(D\) equals the Euclidean Sobolev constant
	\[
	S=\inf_{0\ne z\in D^{1,2}(\mathbb R^N)}
	\frac{\displaystyle\int_{\mathbb R^N}|\nabla z|^2\,dx}
	{\displaystyle\left(\int_{\mathbb R^N}|z|^{2^*}\,dx\right)^{2/2^*}}.
	\]
	Aubin and Talenti identified the extremals as the translation-dilation family \cite{Aubin1976,Talenti1976}:
	\[
	U_{\lambda,\xi}(x)
	=c_N\left(\frac{\lambda}
	{1+\lambda^2|x-\xi|^2}\right)^{\frac{N-2}{2}},
	\qquad \lambda>0,\quad \xi\in\mathbb R^N.
	\]
	Cutoffs of these functions concentrate inside \(D\) and form minimizing sequences.  The bounded-domain infimum equals \(S\), but it is not attained. The concentration-compactness principle of Lions and the global compactness theorem of Struwe make this defect precise: bounded critical Palais-Smale sequences may split into a weak limit and finitely many rescaled entire profiles; for nonnegative sequences, the classification of Caffarelli, Gidas, and Spruck identifies these profiles as positive Talenti bubbles and gives quantized energy loss \cite{Lions84I,Lions84II,LionsLimit1985,Struwe1984,CGS1989}.
	
	Existence at the critical exponent is therefore controlled by mechanisms that prevent the min-max construction from disappearing into bubbles.  One such mechanism is a favorable lower-order term.  Brezis and Nirenberg studied
	\[
	-\Delta w=\mu w+w^{2^*-1}\quad\text{in }D,\qquad w=0\quad\text{on }\partial D,
	\]
	and proved, in particular, existence for \(N\ge4\) and \(0<\mu<\lambda_1(D)\); the three-dimensional case exhibits an additional threshold \cite{BN83}.  A different mechanism is the geometry of the domain. The Pohozaev identity rules out nontrivial solutions of the pure equation on star-shaped domains \cite{Poh65}, whereas a concentric annulus admits a positive radial solution \cite{KazdanWarner1975}.  Coron extended this observation to nonsymmetric domains with a sufficiently small hole \cite{Coron1984}.  Bahri and Coron then proved the general homological criterion: if
	\[
	H_d(D;\mathbb F_2)\ne0
	\qquad\text{for some }d>0,
	\]
	then the pure critical Dirichlet problem has a positive solution \cite{BahriCoron1988}.  Their finite-barycenter construction propagates a nonzero topological class through successive bubbling levels until a high-multiplicity energy estimate forces a contradiction.
	
	Bahri's critical points at infinity theory gives a Morse theoretic interpretation of bubbling orbits \cite{Bah89}; see also the contemporaneous survey \cite{Brezis1986}.  Green and Robin functions determine the leading self-action and interaction of concentrating profiles, and hence the topology contributed by multi-bubble configurations \cite{Rey1990,BLR95}.  Related barycenter methods were developed for critical equations on closed Riemannian manifolds \cite{BB96} and were subsequently adapted to other critical geometries, including the Heisenberg group and fractional Yamabe-type problems \cite{CU01,ACH18}.  The interaction estimate needed to recover topology is therefore operator- and geometry-dependent.
	
	The critical exponent is also central to the Yamabe problem, where the conformal Laplacian, rather than the Laplace-Beltrami operator alone, has the appropriate covariance under changes of metric \cite{Yam60,Aub76,Sch84,Aubin1998,Hebey1999}.  This distinction is decisive here: the pure hyperbolic Laplacian is not conformally covariant, and its Euclidean reduction retains a nonzero geometric potential.
	
	The hyperbolic Brezis-Nirenberg family exhibits this dependence.  On a bounded domain \(\Omega\Subset\mathbb H^N\), consider
	\[
	-\Delta_{\mathbb H}u=\mu u+|u|^{2^*-2}u,\qquad
	u\in H^1_0(\Omega;g_{\mathbb H}).
	\]
	In the Poincar\'e ball model, with \(\rho(x)=2/(1-|x|^2)\) and \(v=\rho^{(N-2)/2}u\), this equation becomes
	\[
	-\Delta v+
	\left(\frac{N(N-2)}4-\mu\right)\rho(x)^2v
	=|v|^{2^*-2}v.
	\]
	Thus the conformal value \(\mu_0=N(N-2)/4\) separates two analytically different regimes.  When \(\mu>\mu_0\), the reduced lower-order term has the favorable negative sign familiar from the Euclidean Brezis-Nirenberg problem.  For the pure Laplace-Beltrami equation, however, \(\mu=0\) and the reduced potential is strictly positive.
	
	The bounded domain theory with a spectral parameter was initiated by Stapelkamp \cite{Stapelkamp2002,Stapelkamp2003}.  For \(N\ge4\), existence holds in the shifted interval \(\mu_0<\mu<\lambda_1(-\Delta_{\mathbb H},\Omega)\), while a Pohozaev-type obstruction gives nonexistence on hyperbolically star-shaped domains when \(\mu\le\mu_0\).  In dimension three, Stapelkamp identified a geodesic-ball solution gap and uniqueness above its threshold; Benguria subsequently characterized the radial threshold for the continuous dimension range \(2<n<4\) \cite{Stapelkamp2002,Stapelkamp2003,Benguria2016}. Higher-order analogues for GJMS operators on bounded domains and on the complete hyperbolic space were established by Li, Lu, and Yang \cite{LiLuYang2022}.  More recently, Ghosh, Kumar, and Rana proved spectral multiplicity on smooth bounded hyperbolic domains for \(N\ge4\) and \(\mu>\mu_0\) \cite{GhoshKumarRana2026}.  These results develop the favorable spectral-perturbation regime, but they do not address topology-driven existence for the pure equation \(\mu=0\).
	
	A complementary literature concerns equations on the complete hyperbolic space.  Mancini and Sandeep classified positive finite-energy entire solutions and proved uniqueness modulo hyperbolic isometries in the critical existence range \cite{ManciniSandeep2008}.  Bhakta and Sandeep showed that critical Palais-Smale sequences may lose compactness in two different ways: through escape by hyperbolic isometries and through localized Euclidean Aubin-Talenti concentration \cite{BhaktaSandeep2012}.  Ganguly and Sandeep developed the corresponding sign-changing theory \cite{GangulySandeep2014}. For the forced variable coefficient whole space problem, Bhakta, Ganguly, Gupta, and Sahoo proved a global compactness theorem and obtained existence and multiplicity results \cite{BhaktaGangulyGuptaSahoo2025}.  For the corresponding homogeneous variable coefficient problem, they established positive-solution existence under energy-ordering and flatness hypotheses \cite{BhaktaGangulyGuptaSahoo2024}.  Sharp one-bubble and multi-bubble stability estimates for the Poincar\'e-Sobolev inequality were subsequently obtained in \cite{BhaktaGangulyKarmakarMazumdar2025a, BhaktaGangulyKarmakarMazumdar2025b}.  On a fixed \(\Omega\Subset\mathbb H^N\), the isometric escape channel is absent, but the positive geometric potential and Euclidean concentration remain.
	
	We study precisely this bounded, pure critical regime:
	\begin{equation}\label{eq1.1}
		\begin{cases}
			-\Delta_{\mathbb H}u=u^{2^*-1} &\text{in }\Omega,\\
			u>0 &\text{in }\Omega,\\
			u=0 &\text{on }\partial\Omega,
		\end{cases}
	\end{equation}
	where \(\Omega\) is bounded, connected, and of class \(C^2\).  The conformal reduction gives
	\begin{equation}\label{eq1.2}
		v=\rho^{\frac{N-2}{2}}u,\qquad
		-\Delta v+a(x)v=v^{2^*-1},\qquad
		a(x)=\frac{N(N-2)}{(1-|x|^2)^2}>0.
	\end{equation}
	The associated quotient
	\[
	\mathcal S_a(v)=
	\frac{\displaystyle\int_\Omega\bigl(|\nabla v|^2+a(x)v^2\bigr)\,dx}
	{\displaystyle\left(\int_\Omega|v|^{2^*}\,dx\right)^{2/2^*}}
	\]
	has infimum \(S\), but the infimum is not attained: positivity of \(a\) gives \(\mathcal S_a(v)>S\) for every nonzero \(v\), while concentrating Talenti bubbles approach \(S\).  Consequently the existence argument must go beyond direct minimization.  More importantly for the Bahri-Coron construction, each projected bubble carries a positive self-correction that must be overcome by its interaction with the other bubbles.
	
	The Euclidean positive potential literature contains two especially close comparisons.  For \(N\ge4\), \(x_0\in\Omega\), nonnegative \(\bar a\in L^{N/2}(\Omega)\), and nonzero nonnegative \(\alpha\in L^{N/2}(\mathbb R^N)\), Passaseo proved existence for potentials of the form
	\[
	a_\Lambda(x)=\bar a(x)
	+\Lambda^2\alpha\bigl(\Lambda(x-x_0)\bigr),
	\]
	when \(\Lambda\) is sufficiently large.  If \(\|\alpha\|_{L^{N/2}}<S(2^{2/N}-1)\), the cited theorem gives a second positive solution \cite{Passaseo1996}. These results do not cover the fixed geometric potential in \eqref{eq1.2}. In dimension three, Aldawood and Ndiaye proved existence on noncontractible bounded domains with connected smooth boundary for \(-\Delta+q\), where \(q\) is smooth and bounded \cite[Theorem~1.1]{AldawoodNdiaye2025}. The theorem below allows a \(C^2\) boundary of arbitrary connectivity in dimension three and establishes the corresponding result in dimensions four and five.  To the best of our knowledge, the present work provides the first unified topology-driven existence framework for the fixed geometric potential \eqref{eq1.2} in dimensions \(N=3,4,5\).
	
	The main result is the following.
	
	\begin{theorem}\label{Thm1.1}
		Let \(N\in\{3,4,5\}\), and let \(\Omega\Subset\mathbb H^N\) be a bounded connected \(C^2\) domain. If
		\[
		H_d(\Omega;\mathbb F_2)\ne0
		\qquad\text{for some }1\le d\le N-1,
		\]
		then \eqref{eq1.1} admits a positive weak solution \(u\in H^1_0(\Omega;g_{\mathbb H})\).
	\end{theorem}
	
	The hypothesis includes, for example, hyperbolic annular domains and domains with nontrivial handles.  The theorem shows that positive-degree topology can overcome both the critical loss of compactness and the unfavorable geometric potential.  Beyond the conformal reduction, the positive self-correction changes the high-multiplicity energy balance and requires an additional analytic mechanism.
	
	The dimension dependence enters through the competition between the positive single-bubble correction and the negative pair interactions on the balanced weak-interaction stratum.  The integers \(m\) and \(n\) both count bubbles, but they occur under different quantifiers.  In dimension three, \(m\) is one terminal multiplicity fixed before the concentration limit \(\lambda\to\infty\).  In dimensions four and five, \(n\) is the running number of bubbles in a uniform analytic estimate; after that estimate has been proved, one sufficiently large value is fixed and denoted by \(m\) in the topological argument.  Thus the four- and five-dimensional proof eventually specializes \(n=m\).
	
	The leading orders of one single-bubble correction and one fixed-separation pair interaction, before summation over the bubbles or pairs, are summarized in Table~\ref{tab:intro-dimension-comparison}.
	
	\begin{table}[H]
		\centering
		\small
		\begin{tabular}{@{}c c c p{6.1cm}@{}}
			\toprule
			Dimension & One-bubble term & One-pair term & Multiplicity--scale regime\\
			\midrule
			\(N=3\) & \(\lambda^{-1}\) & \(\lambda^{-1}\) &
			choose one fixed terminal multiplicity \(m\), then let
			\(\lambda\to\infty\)\\
			\(N=4\) & \(\lambda^{-2}\log\lambda\) & \(\lambda^{-2}\) &
			running multiplicity \(n\to\infty\), with \(\lambda_n=An\); then specialize
			\(n=m\)\\
			\(N=5\) & \(\lambda^{-2}\) & \(\lambda^{-3}\) &
			running multiplicity \(n\to\infty\), with
			\(\lambda_n=A n^{2/3}\); then specialize \(n=m\)\\
			\(N\ge6\) & \(\lambda^{-2}\) & \(\lambda^{-(N-2)}\) &
			\multicolumn{1}{l@{}}{nondecaying or growing transfer error}\\
			\bottomrule
		\end{tabular}
		\caption{Single-term scales and multiplicity--scale regimes.}
		\label{tab:intro-dimension-comparison}
	\end{table}
	
	For \(N=3\), the diagonal and pair terms have the same \(\lambda^{-1}\) order.  After normalization at balanced weights, the positive diagonal coefficient is bounded independently of \(m\), whereas the total negative Green-interaction coefficient grows like \(m-1\). One therefore first fixes \(m\) sufficiently large that the interaction dominates the diagonal contribution, and only then lets \(\lambda\to\infty\).
	
	For \(N\in\{4,5\}\), the proof instead uses a uniform estimate in the running multiplicity \(n\).  On the balanced full-support weak-interaction stratum there are \(n(n-1)\) ordered-pair contributions.  This leads to the matched scale
	\[
	\lambda_n=A n^{\frac{2}{N-2}},
	\]
	for which
	\[
	n(n-1)\lambda_n^{-(N-2)}
	=A^{-(N-2)}\left(1-\frac1n\right).
	\]
	Thus the aggregate interaction remains of order one.  Configurations with zero weights, unbalanced coefficients, or strong collisions are controlled separately by fixed defect estimates.
	
	After exact normalization, the projected-source estimate bounds the total transfer remainder by \(Cn\varepsilon_N(\lambda)\), where
	\[
	\varepsilon_4(\lambda)
	=\lambda^{-2}(1+\log\lambda)^2,
	\qquad
	\varepsilon_5(\lambda)
	=\lambda^{-2}(1+\log\lambda).
	\]
	At the matched scale,
	\[
	n\varepsilon_4(\lambda_n)
	=O_A\!\left(n^{-1}(1+\log n)^2\right),
	\qquad
	n\varepsilon_5(\lambda_n)
	=O_A\!\left(n^{-1/3}(1+\log n)\right),
	\]
	and both quantities tend to zero.  The transfer remainder can therefore be absorbed by the order-one defect, giving a strict drop below the \(n\)-bubble level.  The induced relative map is then transported to the larger topology scale by the scale homotopy, after which one fixes a large admissible value and writes \(n=m\).
	
	The same comparison identifies the endpoint of the present method.  For \(N\ge5\), the projection estimates used here have the base scale \(\lambda^{-2}\).  Even if all logarithmic losses are ignored, at the matched scale the accumulated bound has size
	\[
	n\lambda_n^{-2}
	=A^{-2}n^{\frac{N-6}{N-2}}.
	\]
	This tends to zero for \(N=5\), remains of order one for \(N=6\), and grows for \(N>6\).  Thus the present bounds yield a nondecaying transfer remainder in dimensions \(N\ge6\), even at the optimistic logarithm-free scale. Equivalently, retaining an order-one all-pairs defect requires \(\lambda_n\lesssim n^{2/(N-2)}\), whereas making the aggregate projection bound negligible would require \(\lambda_n\gg n^{1/2}\).  These requirements are compatible for \(N=4,5\), meet at the nondecaying borderline \(N=6\), and are reversed for \(N>6\).  Consequently, the error-absorption and scale-bridge argument developed here does not close for \(N\ge6\).
	
	The proof has a common analytic part and a topological part.  The conformal identity \eqref{eq1.2} reduces the hyperbolic variational problem to the Euclidean functional
	\[
	J(v)=\frac12\int_\Omega\big(|\nabla v|^2+a(x)v^2\big)\,dx
	-
	\frac1{2^*}\int_\Omega |v|^{2^*}\,dx.
	\]
	A global compactness theorem shows that noncompact nonnegative Palais-Smale sequences split into a solution and finitely many interior Talenti bubbles; a half-space Liouville theorem excludes boundary bubbles. The compactness decomposition uses ordinary Dirichlet projections.  For the barycenter test family, we instead use projections associated with \(L_a\), which encode the mixed potential terms through its positive Green kernel. Under the contradiction hypothesis that no positive solution exists, quotient sublevel pairs deform to finite-dimensional bubble tubes on which the preceding energy estimates apply.
	
	The topological module begins with a Thom realization \(h:V\to K\Subset \Omega\) of a nonzero class in \(H_d(\Omega;\mathbb F_2)\).  Working over \(\mathbb F_2\) avoids orientation restrictions.  The construction allows arbitrary smooth maps \(h\), including non-embeddings.  The resolved barycenter model treats zero-weight and intrinsic-collision faces, while separate analytic collars handle off-diagonal self-collisions \(z_r\ne z_s\) with \(h(z_r)=h(z_s)\).  After these strata are controlled, a cap-boundary identity propagates a nonzero relative homology class through every multiplicity.  At the top level, topology therefore makes the test map nonzero, whereas the strict energy drop makes the same map zero.  This is the final contradiction.
	
	The organization follows these dependencies.  Sections~2-4 establish the common conformal framework, compactness and deformation theory, and projected-bubble estimates.  Section~5 proves the fixed-multiplicity energy drop in dimension three, the matched-scale energy drop in dimensions four and five, and the scale bridge needed to compare the analytic and topological families.  Section~6 develops the Thom realization and resolved labelled-barycenter module, handles the three kinds of degeneracy just described, and closes the contradiction proving Theorem~\ref{Thm1.1}.

	\section{Conformal formulation, Green kernels, and projected bubbles}\label{sec:geom}

	\subsection{Conformal formulation and Green kernels}
	
	We use the Poincar\'e ball model
	\[
	\B^N=\{x\in\R^N:\ |x|<1\},\qquad
	g_{\Hh}=\rho(x)^2\delta,\qquad
	\rho(x)=\frac{2}{1-|x|^2}.
	\]
	Thus
	\[
	dV_{\Hh}=\rho^N\,dx,\qquad
	|\gradH u|^2=\rho^{-2}|\nabla u|^2,
	\]
	and
	\[
	\Delta_{\Hh}
	=
	\rho^{-2}\Delta+(N-2)\rho^{-2}\frac{2x}{1-|x|^2}\cdot\nabla
	=
	\Big(\frac{1-|x|^2}{2}\Big)^2\Delta
	+(N-2)\Big(\frac{1-|x|^2}{2}\Big)x\cdot\nabla.
	\]
	Since the domain is relatively compact in $\Hh^N$, we identify $\Omega$ with its image in the Poincar\'e ball and keep the same notation.  Then
	\[
	\overline\Omega\subset\B^N,\qquad \rho,\rho^{-1}\in C^\infty(\overline\Omega).
	\]
	
	Set
	\[
	\phi(x)=\rho(x)^{\frac{N-2}{2}},
	\qquad
	v=\phi u.
	\]
	Define
	\[
	a(x)=\frac{N(N-2)}{4}\rho(x)^2
	=\frac{N(N-2)}{(1-|x|^2)^2}
	\]
	and
	\[
	J(v)=
	\frac12\int_\Omega\big(|\nabla v|^2+a(x)v^2\big)\,dx
	-\frac1{2^*}\int_\Omega |v|^{2^*}\,dx.
	\]
	The next lemma records the exact form of the problem in the Euclidean coordinates.
	
	\begin{lemma}\label{Lem2.1}
		For every $u\in C_c^\infty(\Omega)$ and $v=\phi u$ one has
		\begin{equation}\label{eq2.1}
			\int_\Omega |\gradH u|^2\,dV_{\Hh}
			=
			\int_\Omega |\nabla v|^2\,dx
			+
			\int_\Omega a(x)v^2\,dx,
		\end{equation}
		Moreover,
		\begin{equation}\label{eq2.2}
			\int_\Omega |u|^{2^*}\,dV_{\Hh}
			=
			\int_\Omega |v|^{2^*}\,dx,
			\qquad 2^*=\frac{2N}{N-2}.
		\end{equation}
		Consequently the hyperbolic equation
		\begin{equation}\label{eq2.3}
			-\Delta_{\Hh}u=|u|^{2^*-2}u,\qquad u\in H^1_0(\Omega),
		\end{equation}
		is equivalent, under $v=\phi u$, to
		\begin{equation}\label{eq2.4}
			-\Delta v+a(x)v=|v|^{2^*-2}v,\qquad v\in H^1_0(\Omega).
		\end{equation}
		Critical points of the original hyperbolic functional are in one-to-one correspondence with critical points of $J$.
	\end{lemma}
	
	\begin{proof}
		The scalar curvature of $g_{\Hh}$ is $-N(N-1)$. The conformal Laplacian identity for $g_{\Hh}=\rho^2\delta$ gives
		\[
		-\Delta_{\Hh}u-\frac{N(N-2)}4u
		=
		\rho^{-\frac{N+2}{2}}(-\Delta v),
		\qquad v=\rho^{\frac{N-2}{2}}u.
		\]
		Multiplying by $\rho^{\frac{N+2}{2}}$ yields
		\[
		-\Delta v+\frac{N(N-2)}4\rho^2v
		=
		\rho^{\frac{N+2}{2}}(-\Delta_{\Hh}u).
		\]
		Since
		\[
		\rho^{\frac{N+2}{2}}|u|^{2^*-2}u
		=
		|v|^{2^*-2}v,
		\]
		\eqref{eq2.3} is equivalent to \eqref{eq2.4}. The identity \eqref{eq2.2} follows from
		\[
		|u|^{2^*}dV_{\Hh}
		=
		\rho^{-\frac{N-2}{2}2^*}|v|^{2^*}\rho^N\,dx
		=
		|v|^{2^*}\,dx.
		\]
		Finally, testing the conformal Laplacian identity against $u$ and using the vanishing trace gives
		\[
		\int_\Omega |\gradH u|^2\,dV_{\Hh}
		-\frac{N(N-2)}4\int_\Omega u^2\,dV_{\Hh}
		=
		\int_\Omega |\nabla v|^2\,dx.
		\]
		Since
		\[
		\int_\Omega u^2\,dV_{\Hh}
		=
		\int_\Omega \rho^{-\,(N-2)}v^2\rho^N\,dx
		=
		\int_\Omega \rho^2v^2\,dx,
		\]
		we obtain \eqref{eq2.1}. Density extends the identities to $H^1_0(\Omega)$.
	\end{proof}
	
	We use the norm
	\[
	\|v\|_a^2=\int_\Omega\big(|\nabla v|^2+a(x)v^2\big)\,dx.
	\]
	Since $a\in C^\infty(\overline\Omega)$ and $a>0$, this norm is equivalent to \(\|\cdot\|_{H^1_0(\Omega)}\).  The Sobolev inequality gives
	\[
	S\left(\int_\Omega |v|^{2^*}\,dx\right)^{2/2^*}
	\le
	\int_\Omega |\nabla v|^2\,dx
	\le
	\|v\|_a^2,
	\]
	where $S$ is the sharp Sobolev constant in $\R^N$.
	
	The closed coercive symmetric form
	\[
	\mathfrak a_a[u,w]
	=
	\int_\Omega\bigl(\nabla u\cdot\nabla w+a(x)uw\bigr)\,dx,
	\qquad
	D(\mathfrak a_a)=H^1_0(\Omega),
	\]
	will also be denoted by \(\langle u,w\rangle_a=\mathfrak a_a[u,w]\); its induced norm is the norm \(\|\cdot\|_a\) above.  This form defines a positive self-adjoint Friedrichs realization \(L_a=-\Delta+a(x)\) in \(L^2(\Omega)\).  Since \(\partial\Omega\) is \(C^2\) and \(a\in C^\infty(\overline\Omega)\), elliptic regularity gives
	\[
	D(L_a)
	=
	\{u\in H^1_0(\Omega):-\Delta u+a(x)u\in L^2(\Omega)\}
	=
	H^2(\Omega)\cap H^1_0(\Omega).
	\]
	The same symbol \(L_a\) denotes the isomorphism \(H^1_0(\Omega)\to H^{-1}(\Omega)\) induced by \(\mathfrak a_a\); its inverse \(L_a^{-1}:H^{-1}(\Omega)\to H^1_0(\Omega)\) is bounded and preserves positivity for nonnegative sources.  Accordingly \(L_a^{-1}\) below is also used in this Lax-Milgram sense.  We use the Green kernel of this realization in the standard potential-theoretic sense \cite{ArmitageGardiner2001}. It is denoted by \(G_a\):
	\[
	L_{a,x}G_a(x,y)=\delta_y\quad\hbox{in }\Omega,
	\qquad
	G_a(x,y)=0\quad\hbox{for }x\in\partial\Omega.
	\]
	By the maximum principle and self-adjointness, \(G_a(x,y)=G_a(y,x)>0\) for \(x\ne y\).  Its diagonal singularity is the same as that of the Laplacian Green kernel.  The comparison with \(G_\Omega\) used below is recorded in \eqref{eq2.6}.  The fixed-multiplicity interaction estimate in Section~\ref{sec:matched} is formulated with \(G_a\).
	
	Let $G_\Omega$ be the Green function of $-\Delta$ in $\Omega$ with zero boundary value:
	\[
	-\Delta_xG_\Omega(x,y)=\delta_y\quad\hbox{in }\Omega,\qquad
	G_\Omega(x,y)=0\quad\hbox{for }x\in\partial\Omega.
	\]
	We write
	\begin{equation}\label{eq2.5}
		G_\Omega(x,y)=\Gamma(x-y)-H_\Omega(x,y),
		\qquad
		\Gamma(z)=\frac1{(N-2)\omega_{N-1}}|z|^{2-N}.
	\end{equation}
	The Robin function is
	\[
	R_\Omega(x)=H_\Omega(x,x),\qquad x\in\Omega.
	\]
	
	\begin{lemma}\label{Lem2.2}
		Let $\Omega\subset\R^N$ be bounded with $C^2$ boundary. Then:
		\begin{enumerate}
			\item $G_\Omega(x,y)=G_\Omega(y,x)>0$ for $x\ne y$.
			\item $H_\Omega\in C^\infty(\Omega\times\Omega)$ and $R_\Omega\in C^\infty(\Omega)$.
			\item If $\Omega_1\subset\Omega_2$, then
			\[
			G_{\Omega_1}(x,y)\le G_{\Omega_2}(x,y),
			\qquad x,y\in\Omega_1,\ x\ne y,
			\]
			and
			\[
			R_{\Omega_1}(x)\ge R_{\Omega_2}(x),\qquad x\in\Omega_1.
			\]
			\item There are constants $c,C>0$, depending only on $\Omega$, such that
			\[
			c\,\operatorname{dist}(x,\partial\Omega)^{2-N}
			\le
			R_\Omega(x)
			\le
			C\,\operatorname{dist}(x,\partial\Omega)^{2-N},
			\qquad x\in\Omega.
			\]
			In particular $R_\Omega(x)\to+\infty$ as $x\to\partial\Omega$. More precisely, if \(r_{\rm e}>0\) is a uniform exterior-sphere radius, then, whenever \(d(x):=\operatorname{dist}(x,\partial\Omega)\le r_{\rm e}\),
			\begin{equation}\label{eq2.robin-explicit}
				\frac{1}{(N-2)\omega_{N-1}3^{N-2}}\,
				d(x)^{2-N}
				\le R_\Omega(x)
				\le
				\frac{1}{(N-2)\omega_{N-1}}\,d(x)^{2-N}.
			\end{equation}
			In dimension three the lower bound is
			\begin{equation}\label{eq2.robin-three}
				R_\Omega(x)\ge\frac1{12\pi d(x)}
				\qquad(0<d(x)\le r_{\rm e}).
			\end{equation}
		\end{enumerate}
	\end{lemma}
	
	\begin{proof}
		Existence, symmetry, and positivity of $G_\Omega$ follow from the self-adjointness of the Dirichlet Laplacian and the maximum principle.  Since $\Gamma(\cdot-y)$ has the same singularity as $G_\Omega(\cdot,y)$ at $y$, the difference $H_\Omega(\cdot,y)=\Gamma(\cdot-y)-G_\Omega(\cdot,y)$ is harmonic near $y$ after removal of the singularity.  Interior elliptic regularity in both variables gives $H_\Omega\in C^\infty(\Omega\times\Omega)$.
		
		The domain monotonicity follows by applying the maximum principle to $G_{\Omega_2}(\cdot,y)-G_{\Omega_1}(\cdot,y)$ in $\Omega_1\setminus\{y\}$.  Since $\Gamma$ is domain independent, the monotonicity of $R_\Omega$ follows by taking the diagonal value of $H_\Omega=\Gamma-G_\Omega$.
		
		It remains to prove the boundary estimate.  Put
		\[
		c_N=\frac1{(N-2)\omega_{N-1}},
		\qquad d=d(x).
		\]
		The ball \(B(x,d)\) is contained in \(\Omega\).  Domain monotonicity and the Green function of a ball, evaluated at its centre, give
		\[
		R_\Omega(x)\le R_{B(x,d)}(x)=c_Nd^{2-N}.
		\]
		
		For the reverse inequality we use the uniform exterior-sphere condition, which is valid for a compact \(C^2\) boundary.  Thus there is \(r_{\rm e}>0\) such that, for every \(p\in\partial\Omega\), one can find \(q_p\notin\Omega\) for which
		\[
		\overline {B(q_p,r_{\rm e})}\cap\overline\Omega=\{p\},
		\qquad
		B(q_p,r_{\rm e})\subset\mathbb R^N\setminus\overline\Omega .
		\]
		Decrease \(r_{\rm e}\), if necessary, below the tubular-neighbourhood radius of \(\partial\Omega\).  If \(d\le r_{\rm e}\), the nearest point \(p\) is unique and
		\[
		|x-q_p|=r_{\rm e}+d.
		\]
		The exterior domain
		\[
		D_p=\mathbb R^N\setminus\overline {B(q_p,r_{\rm e})}
		\]
		contains \(\Omega\).  Its Green function, with zero boundary value and decay at infinity, is obtained by Kelvin reflection.  On the diagonal its regular part is
		\[
		R_{D_p}(x)
		=
		c_N
		\left(
		\frac{r_{\rm e}}
		{|x-q_p|^2-r_{\rm e}^2}
		\right)^{N-2}
		=
		c_N
		\left(
		\frac{r_{\rm e}}
		{d(2r_{\rm e}+d)}
		\right)^{N-2}.
		\]
		Since \(\Omega\subset D_p\), domain monotonicity yields
		\[
		R_\Omega(x)\ge R_{D_p}(x)
		\ge c_N3^{2-N}d^{2-N},
		\qquad 0<d\le r_{\rm e}.
		\]
		This proves \eqref{eq2.robin-explicit}, and \(\omega_2=4\pi\) gives \eqref{eq2.robin-three}.  On the compact set \(\{d\ge r_{\rm e}\}\), positivity and continuity of \(R_\Omega\) extend the lower estimate after decreasing its constant.  The lemma follows.
	\end{proof}
	
	For the \(L_a\)-Green kernel we use the following comparison with the Dirichlet Laplacian.  Since
	\[
	L_aG_\Omega(\cdot,y)=\delta_y+a(\cdot)G_\Omega(\cdot,y),
	\]
	the function
	\[
	W_y(x)=G_\Omega(x,y)-G_a(x,y)
	\]
	satisfies
	\[
	L_aW_y=a(x)G_\Omega(x,y)\quad\hbox{in }\Omega,
	\qquad
	W_y=0\quad\hbox{on }\partial\Omega.
	\]
	Hence, by the maximum principle and Green representation,
	\[
	0\le W_y(x)
	=
	\int_\Omega G_a(x,\zeta)a(\zeta)G_\Omega(\zeta,y)\,d\zeta
	\le
	\|a\|_{L^\infty(\Omega)}
	\int_\Omega G_\Omega(x,\zeta)G_\Omega(\zeta,y)\,d\zeta.
	\]
	Using \eqref{eq2.5}, the boundedness of \(H_\Omega\) on \(K\times K\), and the following elementary convolution bound,
	\[
	\int_\Omega |x-z|^{2-N}|z-y|^{2-N}\,dz
	\le C_K\Psi_N(|x-y|),
	\qquad x,y\in K,
	\]
	we obtain the following: for every compact set \(K\Subset\Omega\) and \(N\ge3\), there exists \(C_K>0\) such that
	\begin{equation}\label{eq2.6}
		|G_a(x,y)-\Gamma(x-y)|
		\le
		C_K\Psi_N(|x-y|),
		\qquad
		x,y\in K,\ x\ne y,
	\end{equation}
	where
	\[
	\Psi_N(r)=
	\begin{cases}
		1,&N=3,\\
		1+|\log r|,&N=4,\\
		r^{4-N},&N\ge5.
	\end{cases}
	\]
	The convolution estimate is obtained by separating \(B_{r/2}(x)\), \(B_{r/2}(y)\), and the remaining region, where \(r=|x-y|\).  Each near-pole piece is bounded by \(Cr^{2-N}\int_0^{r/2}s\,ds\), while the remaining region is bounded by \(C\int_{r/2}^{C_K}s^{3-N}\,ds\).  This gives a bounded contribution in \(N=3\), a logarithmic contribution in \(N=4\), and a contribution of order \(r^{4-N}\) for \(N\ge5\). In particular,
	\[
	G_a(x,y)=\Gamma(x-y)+o(\Gamma(x-y))
	\qquad\text{as }x\to y,
	\]
	uniformly for \(y\) in compact subsets of \(\Omega\).  Moreover, for every compact set \(K\Subset\Omega\),
	\begin{equation}\label{eq2.7}
		\inf\{G_a(x,y):x,y\in K,\ x\ne y\}>0.
	\end{equation}
	The estimate \eqref{eq2.6} makes \(G_a\) arbitrarily large near the diagonal, while \(G_a\) is continuous and positive on \(\{(x,y)\in K\times K:\ |x-y|\ge\varepsilon\}\).
	
	We use the following boundary flattening in the compactness argument.  The coordinates are Euclidean coordinates in the Poincar\'e ball.
	
	\begin{lemma}\label{Lem2.3}
		Let $x_0\in\partial\Omega$.  After a rigid change of coordinates, there are $r_0>0$ and a $C^2$ function $\psi:B'_{r_0}\subset\R^{N-1}\to\R$ such that
		\[
		\Omega\cap B_{r_0}(x_0)
		=
		\{(x',x_N):\ x_N>\psi(x')\}\cap B_{r_0}(x_0),
		\qquad
		\psi(0)=0,\quad \nabla\psi(0)=0.
		\]
		If $\lambda_k\to\infty$, $x_k\in\Omega$, and $x_k\to x_0$, then the rescaled domains
		\[
		\Omega_k=\lambda_k(\Omega-x_k)
		\]
		either locally exhaust $\R^N$, if
		\[
		\lambda_k\operatorname{dist}(x_k,\partial\Omega)\to+\infty,
		\]
		or converge locally in $C^1$ to a half-space, after passing to a subsequence, if
		\[
		\lambda_k\operatorname{dist}(x_k,\partial\Omega)
		\]
		is bounded.
	\end{lemma}
	
	\begin{proof}
		Translate \(x_0\) to the origin and rotate the coordinates so that the inward unit normal to \(\partial\Omega\) at \(x_0\) is \(e_N\).  The local graph representation, with \(\psi(0)=0\) and \(\nabla\psi(0)=0\), follows from the \(C^2\) implicit function theorem.
		
		Put
		\[
		d_k=\operatorname{dist}(x_k,\partial\Omega).
		\]
		If \(\lambda_kd_k\to\infty\), then
		\[
		B(0,\lambda_kd_k)\subset\lambda_k(\Omega-x_k)=\Omega_k.
		\]
		Hence every fixed compact subset of \(\mathbb R^N\) is eventually contained in \(\Omega_k\), which is the asserted local exhaustion of \(\mathbb R^N\).
		
		Suppose now that \((\lambda_kd_k)\) is bounded.  Let \(s\) be the signed distance to \(\partial\Omega\), chosen positive in \(\Omega\).  Since \(\partial\Omega\) is \(C^2\), the function \(s\) is \(C^2\) in a tubular neighbourhood of the boundary.  For large \(k\), let \(y_k\) be the unique nearest boundary point to \(x_k\).  Then
		\[
		s(x_k)=d_k,\qquad \nabla s(x_k)=\nu(y_k),
		\]
		where \(\nu\) is the inward unit normal.  Moreover, \(y_k\to x_0\), so \(\nu(y_k)\to e_N\).  After passing to a subsequence, write
		\[
		\lambda_kd_k\to\ell\in[0,\infty).
		\]
		For \(z\) in a fixed ball, define
		\[
		F_k(z)=\lambda_k s\left(x_k+\frac{z}{\lambda_k}\right).
		\]
		The uniform \(C^2\) bound for \(s\) gives, for every fixed \(R>0\),
		\[
		F_k(z)=\lambda_kd_k+\nu(y_k)\cdot z+O_R(\lambda_k^{-1}),
		\qquad |z|\le R,
		\]
		and
		\[
		\nabla F_k(z)=\nu(y_k)+O_R(\lambda_k^{-1}).
		\]
		Consequently,
		\[
		F_k\to \ell+z_N
		\quad\hbox{in }C^1(B_R)
		\]
		for every \(R>0\).  In the same neighbourhood,
		\[
		\Omega_k=\{z:F_k(z)>0\}.
		\]
		Thus \(\Omega_k\) converges locally in \(C^1\) to the half-space \(\{z_N>-\ell\}\).
	\end{proof}
	
	\begin{lemma}\label{Lem2.4}
		Let \(N\ge3\).  If
		\[
		w\in D^{1,2}_0(\mathbb R^N_+)
		\]
		is a weak solution of
		\[
		-\Delta w=|w|^{2^*-2}w
		\quad\hbox{in }\mathbb R^N_+,
		\qquad
		w=0\quad\hbox{on }\partial\mathbb R^N_+,
		\]
		then \(w\equiv0\).
	\end{lemma}
	
	\begin{proof}
		Let \(\mathcal C:B_1\to\mathbb R^N_+\) be the Cayley transform
		\[
		\mathcal C(x)=
		\left(
		\frac{2x'}{|x-e_N|^2},
		\frac{1-|x|^2}{|x-e_N|^2}
		\right),
		\qquad
		\kappa(x)=\frac{2}{|x-e_N|^2}.
		\]
		Thus \(D\mathcal C(x)^TD\mathcal C(x)=\kappa(x)^2I\).  Define
		\[
		\quad
		u(x)=\kappa(x)^{\frac{N-2}{2}}w(\mathcal C(x)).
		\]
		The critical Kelvin-Cayley transform is an isometry for the Dirichlet integral and the \(L^{2^*}\)-norm.  Approximation of \(w\) by \(C_c^\infty(\mathbb R^N_+)\) therefore gives
		\[
		u\in H^1_0(B_1)
		\]
		and
		\[
		-\Delta u=|u|^{2^*-2}u
		\quad\hbox{weakly in }B_1.
		\]
		Indeed, the conformal covariance formula is
		\[
		-\Delta\!\left(
		\kappa^{\frac{N-2}{2}}w\circ\mathcal C
		\right)
		=
		\kappa^{\frac{N+2}{2}}(-\Delta w)\circ\mathcal C.
		\]
		
		Writing the equation as
		\[
		-\Delta u=V(x)u,
		\qquad
		V=|u|^{2^*-2}\in L^{N/2}(B_1),
		\]
		the Br\'ezis-Kato argument first gives \(u\in L^s(B_1)\) for every finite \(s\).  Choose \(s>N(2^*-1)\) and put \(q=s/(2^*-1)>N\).  Then \(|u|^{2^*-2}u\in L^q(B_1)\), and the global Dirichlet \(W^{2,q}\)-estimate on the ball, followed by Sobolev embedding, gives
		\[
		u\in W^{2,q}(B_1)\cap C^{1,\alpha}(\overline{B_1})
		\]
		for some \(\alpha\in(0,1)\).  Interior bootstrapping gives \(u\in C^2(B_1)\); see \cite{BrezisKato1979,GilbargTrudinger2001}.  The usual approximation argument therefore justifies the Pohozaev identity \cite{Poh65}:
		\[
		\frac{N-2}{2}\int_{B_1}|\nabla u|^2\,dx
		-
		\frac{N}{2^*}\int_{B_1}|u|^{2^*}\,dx
		+
		\frac12\int_{\partial B_1}
		(x\cdot\nu)(\partial_\nu u)^2\,dS
		=0.
		\]
		Testing the equation with \(u\) and using \(N/2^*=(N-2)/2\), the volume terms cancel.  Since \(x\cdot\nu=1\) on \(\partial B_1\),
		\[
		\int_{\partial B_1}(\partial_\nu u)^2\,dS=0.
		\]
		Consequently \(u=\partial_\nu u=0\) on \(\partial B_1\).
		
		Let \(\overline u\) be the extension of \(u\) by zero to \(\mathbb R^N\).  Because both the Dirichlet and normal traces vanish, integration by parts gives
		\[
		-\Delta\overline u
		=
		|\overline u|^{2^*-2}\overline u
		\quad\hbox{in }\mathcal D'(\mathbb R^N).
		\]
		Equivalently,
		\[
		-\Delta\overline u=\overline V\,\overline u,
		\qquad
		\overline V=|\overline u|^{2^*-2}
		\in L^{N/2}_{\mathrm{loc}}(\mathbb R^N).
		\]
		The vanishing Dirichlet and normal traces make the zero extension a weak solution of the displayed equation across \(\partial B_1\).  Local elliptic regularity places \(\overline u\) in the solution class required below. The strong unique-continuation theorem of Jerison and Kenig \cite[Theorem~6.3]{JerisonKenig1985} applies for \(N\ge3\) and potentials in \(L^{N/2}_{\mathrm{loc}}\). Since \(\overline u\) vanishes on the nonempty open set \(\mathbb R^N\setminus\overline{B_1}\), it follows that \(\overline u\equiv0\).  Hence \(u\equiv0\), and conformal invertibility gives \(w\equiv0\).
	\end{proof}
	
	\subsection{Projected Talenti bubbles and the diagonal potential}
	
	Let
	\[
	\alpha_N=[N(N-2)]^{\frac{N-2}{4}},
	\qquad
	U_{\lambda,\xi}(x)
	=
	\alpha_N\left(\frac{\lambda}{1+\lambda^2|x-\xi|^2}\right)^{\frac{N-2}{2}},
	\quad \lambda>0,\ \xi\in\Omega.
	\]
	Then
	\[
	-\Delta U_{\lambda,\xi}=U_{\lambda,\xi}^{2^*-1}\quad\hbox{in }\R^N,
	\]
	and
	\[
	\int_{\R^N}|\nabla U_{\lambda,\xi}|^2\,dx
	=
	\int_{\R^N}U_{\lambda,\xi}^{2^*}\,dx
	=
	S^{N/2}.
	\]
	We set
	\[
	I_\infty=\frac1N S^{N/2}.
	\]
	
	\begin{definition}
		For \(\lambda>0\) and \(\xi\in\Omega\), let \(PU_{\lambda,\xi}\in H^1_0(\Omega)\) be the unique weak solution of
		\begin{equation*}
			-\Delta PU_{\lambda,\xi}=U_{\lambda,\xi}^{2^*-1}
			\quad\hbox{in }\Omega,\qquad
			PU_{\lambda,\xi}=0
			\quad\hbox{on }\partial\Omega.
		\end{equation*}
		Equivalently,
		\begin{equation}\label{eq2.8}
			PU_{\lambda,\xi}(x)
			=
			\int_\Omega G_\Omega(x,y)U_{\lambda,\xi}(y)^{2^*-1}\,dy.
		\end{equation}
	\end{definition}
	
	For the compactness theory it is convenient to keep the ordinary projected bubble \(PU_{\lambda,\xi}\), because after blow-up the potential term disappears and the extracted profile is the Euclidean Talenti bubble.  For the Bahri-Coron test family, however, we use the projection associated with \(L_a\).  If \(\delta_{\lambda,\xi}\) denotes the Sobolev optimizer normalized by
	\[
	\int_{\mathbb R^N}|\nabla\delta_{\lambda,\xi}|^2\,dx=1,
	\]
	then there is a dimensional constant
	\[
	\sigma_N=\left(\int_{\mathbb R^N}\delta_{1,0}^{2^*}\,dx\right)^{-1}
	=S^{\frac{N}{N-2}}
	\]
	such that
	\[
	-\Delta\delta_{\lambda,\xi}=\sigma_N\delta_{\lambda,\xi}^{2^*-1}
	\quad\hbox{in }\mathbb R^N.
	\]
	The two bubble normalizations and their ordinary Dirichlet projections are related by
	\[
	\delta_{\lambda,\xi}=S^{-N/4}U_{\lambda,\xi},
	\qquad
	P\delta_{\lambda,\xi}=S^{-N/4}PU_{\lambda,\xi}.
	\]
	Thus \(PU_{\lambda,\xi}\) is the ordinary projection of the unnormalized compactness bubble, \(P\delta_{\lambda,\xi}\) is the ordinary projection of the normalized bubble, and the following \(L_a\)-projection is the one used in the quotient test maps. For \(\lambda>0\) and \(\xi\in\Omega\), we define \(P_a\delta_{\lambda,\xi}\in H^1_0(\Omega)\) by
	\begin{equation}\label{eq2.9}
		L_aP_a\delta_{\lambda,\xi}
		=\sigma_N\delta_{\lambda,\xi}^{2^*-1}
		\quad\hbox{in }\Omega,
		\qquad
		P_a\delta_{\lambda,\xi}=0
		\quad\hbox{on }\partial\Omega.
	\end{equation}
	Equivalently,
	\[
	P_a\delta_{\lambda,\xi}(x)
	=\sigma_N\int_\Omega G_a(x,y)\delta_{\lambda,\xi}(y)^{2^*-1}\,dy.
	\]
	The difference between \(P_a\delta_{\lambda,\xi}\) and the ordinary projection is small in the blow-up scale. More precisely,
	\[
	\|P_a\delta_{\lambda,\xi}-P\delta_{\lambda,\xi}\|_a=o(1)
	\]
	uniformly for \(\xi\in K\Subset\Omega\).  Here \(P\delta_{\lambda,\xi}\) is the ordinary Dirichlet projection of the normalized bubble.  The difference \(e\) satisfies \(L_ae=-aP\delta_{\lambda,\xi}\), and the right-hand side tends to zero in \(H^{-1}(\Omega)\) by the vanishing of the \(L^2\)-mass of a concentrating bubble.  Thus the compactness decomposition may be stated with \(P\delta\), whereas the high barycenter test maps are stated with \(P_a\delta\).
	
	\begin{lemma}\label{Lem2.6}
		Let $K\Subset\Omega$. Uniformly for $\xi\in K$,
		\begin{equation}\label{eq2.10}
			PU_{\lambda,\xi}(x)
			=
			U_{\lambda,\xi}(x)
			-
			A_N\lambda^{-\frac{N-2}{2}}H_\Omega(x,\xi)
			+
			r_{\lambda,\xi}(x),
		\end{equation}
		where, for every compact $K'\Subset\Omega$,
		\[
		\|r_{\lambda,\xi}\|_{C^1(K')}
		=
		O\!\left(\lambda^{-\frac N2}\right).
		\]
		Moreover,
		\begin{equation*}
			\|PU_{\lambda,\xi}-U_{\lambda,\xi}\|_{L^{2^*}(\Omega)}
			=
			O\!\left(\lambda^{-\frac{N-2}{2}}\right),
		\end{equation*}
		and
		\begin{equation}\label{eq2.11}
			PU_{\lambda,\xi}(\xi+z/\lambda)
			=
			U_{\lambda,\xi}(\xi+z/\lambda)
			-
			A_N\lambda^{-\frac{N-2}{2}}R_\Omega(\xi)
			+
			O\!\left(\lambda^{-\frac N2}(1+|z|)\right)
		\end{equation}
		locally uniformly in $z\in\R^N$.
	\end{lemma}
	
	\begin{proof}
		From \eqref{eq2.8} and \eqref{eq2.5},
		\[
		PU_{\lambda,\xi}(x)
		=
		\int_\Omega \Gamma(x-y)U_{\lambda,\xi}(y)^{2^*-1}\,dy
		-
		\int_\Omega H_\Omega(x,y)U_{\lambda,\xi}(y)^{2^*-1}\,dy.
		\]
		Since $U_{\lambda,\xi}$ solves the equation in $\R^N$,
		\[
		U_{\lambda,\xi}(x)
		=
		\int_{\R^N}\Gamma(x-y)U_{\lambda,\xi}(y)^{2^*-1}\,dy.
		\]
		The part of the last integral over $\R^N\setminus\Omega$ is $O(\lambda^{-(N+2)/2})$ uniformly on compact subsets of $\Omega$, because $\xi$ stays a positive distance away from $\partial\Omega$ and $U_{\lambda,\xi}^{2^*-1}$ has mass of order $\lambda^{-(N+2)/2}$ away from $\xi$. Thus
		\[
		PU_{\lambda,\xi}(x)
		=
		U_{\lambda,\xi}(x)
		-
		\int_\Omega H_\Omega(x,y)U_{\lambda,\xi}(y)^{2^*-1}\,dy
		+
		O(\lambda^{-\frac{N+2}{2}})
		\]
		locally uniformly in $x$.
		
		We expand the $H_\Omega$ integral. Since \(H_\Omega\) is smooth on compact subsets of \(\Omega\times\Omega\) and $U_{\lambda,\xi}^{2^*-1}$ concentrates at $\xi$,
		\[
		\int_\Omega H_\Omega(x,y)U_{\lambda,\xi}(y)^{2^*-1}\,dy
		=
		H_\Omega(x,\xi)\int_{\R^N}U_{\lambda,\xi}^{2^*-1}\,dy
		+
		O(\lambda^{-\frac N2}).
		\]
		A direct scaling gives
		\[
		\int_{\R^N}U_{\lambda,\xi}^{2^*-1}\,dy
		=
		A_N\lambda^{-\frac{N-2}{2}},
		\qquad
		A_N:=\int_{\R^N}U_{1,0}^{2^*-1}\,dz>0.
		\]
		This proves \eqref{eq2.10}.  Differentiating the same representation gives the $C^1$ estimate.  For the global $L^{2^*}$ estimate, set $h_{\lambda,\xi}=U_{\lambda,\xi}-PU_{\lambda,\xi}$.  This function is harmonic in $\Omega$ and has boundary values $U_{\lambda,\xi}$.  If $d_0=\operatorname{dist}(K,\partial\Omega)>0$, then, uniformly for $\xi\in K$,
		\[
		0\le U_{\lambda,\xi}
		\le C\lambda^{-\frac{N-2}{2}}
		\quad\hbox{on }\partial\Omega.
		\]
		The maximum principle therefore gives
		\[
		0\le U_{\lambda,\xi}-PU_{\lambda,\xi}
		\le C\lambda^{-\frac{N-2}{2}}
		\quad\hbox{throughout }\Omega,
		\]
		and hence
		\[
		\|PU_{\lambda,\xi}-U_{\lambda,\xi}\|_{L^{2^*}(\Omega)}
		\le C|\Omega|^{1/2^*}\lambda^{-\frac{N-2}{2}}.
		\]
		Finally, \eqref{eq2.11} is obtained by taking $x=\xi+z/\lambda$ in \eqref{eq2.10} and using the smoothness of $H_\Omega$:
		\[
		H_\Omega(\xi+z/\lambda,\xi)=R_\Omega(\xi)+O(|z|/\lambda).
		\]
	\end{proof}
	
	Set
	\[
	\varpi_N(\lambda)=
	\begin{cases}
		\lambda^{-1},&N=3,\\
		\lambda^{-2}\log\lambda,&N=4,\\
		\lambda^{-2},&N\ge5,
	\end{cases}
	\]
	and define the diagonal potential coefficient by
	\[
	\mathfrak p_N(\xi)=
	\begin{cases}
		\displaystyle
		\frac12\int_\Omega a(x)
		\left(
		\alpha_3|x-\xi|^{-1}
		-
		A_3H_\Omega(x,\xi)
		\right)^2\,dx,
		&N=3,\\[1.2em]
		\displaystyle
		\frac12\alpha_4^2\omega_3\,a(\xi),
		&N=4,\\[0.8em]
		\displaystyle
		\frac12\alpha_N^2a(\xi)
		\int_{\mathbb R^N}(1+|z|^2)^{-(N-2)}\,dz,
		&N\ge5.
	\end{cases}
	\]
	
	\begin{lemma}\label{Lem2.7}
		Let \(K\Subset\Omega\).  Uniformly for \(\xi\in K\), as \(\lambda\to\infty\),
		\begin{equation}\label{eq2.12}
			\frac12\int_\Omega a(x)(PU_{\lambda,\xi})^2\,dx
			=
			\mathfrak p_N(\xi)\varpi_N(\lambda)
			+
			o\big(\varpi_N(\lambda)\big).
		\end{equation}
	\end{lemma}
	
	\begin{proof}
		We write
		\[
		E_{\rm pot}(\lambda,\xi)
		=
		\frac12\int_\Omega a(x)(PU_{\lambda,\xi})^2\,dx.
		\]
		All estimates below are uniform for \(\xi\in K\).
		
		We treat \(N\ge5\) first.  By the projection expansion,
		\[
		PU_{\lambda,\xi}=U_{\lambda,\xi}
		+
		O\!\left(\lambda^{-\frac{N-2}{2}}\right)
		\]
		in the sense needed for the \(L^2\)-estimate, and the contribution of the projection correction to \(E_{\rm pot}\) is \(o(\lambda^{-2})\). Hence
		\[
		E_{\rm pot}(\lambda,\xi)
		=
		\frac12\alpha_N^2\lambda^{-2}
		\int_{\lambda(\Omega-\xi)}
		a(\xi+z/\lambda)(1+|z|^2)^{-(N-2)}\,dz
		+
		o(\lambda^{-2}).
		\]
		Since \(N\ge5\), the function \((1+|z|^2)^{-(N-2)}\) is integrable on \(\mathbb R^N\). Dominated convergence gives
		\[
		E_{\rm pot}(\lambda,\xi)
		=
		\mathfrak p_N(\xi)\lambda^{-2}
		+
		o(\lambda^{-2}).
		\]
		This proves \eqref{eq2.12} for \(N\ge5\).
		
		For \(N=4\), the same change of variables gives
		\[
		E_{\rm pot}(\lambda,\xi)
		=
		\frac12\alpha_4^2\lambda^{-2}
		\int_{\lambda(\Omega-\xi)}
		a(\xi+z/\lambda)(1+|z|^2)^{-2}\,dz
		+
		O(\lambda^{-2}).
		\]
		The logarithmic divergence is
		\[
		\int_{|z|\le c\lambda}(1+|z|^2)^{-2}\,dz
		=
		\omega_3\log\lambda+O(1),
		\]
		whereas the contribution of \(a(\xi+z/\lambda)-a(\xi)\) is \(O(1)\) after integration. Thus
		\[
		E_{\rm pot}(\lambda,\xi)
		=
		\frac12\alpha_4^2\omega_3 a(\xi)\lambda^{-2}\log\lambda
		+
		O(\lambda^{-2}).
		\]
		Since \(O(\lambda^{-2})=o(\lambda^{-2}\log\lambda)\), \eqref{eq2.12} follows for \(N=4\).
		
		Consider \(N=3\), and set
		\[
		F_{\lambda,\xi}:=\lambda^{1/2}PU_{\lambda,\xi},
		\qquad
		Q_\xi:=\alpha_3|\,\cdot-\xi|^{-1}-A_3H_\Omega(\cdot,\xi).
		\]
		Since \(A_3=\int_{\mathbb R^3}U_{1,0}^5=4\pi\alpha_3\) and \(\Gamma(z)=(4\pi|z|)^{-1}\), we have
		\[
		Q_\xi=A_3G_\Omega(\cdot,\xi).
		\]
		In particular, \(0\le Q_\xi(x)\le\alpha_3|x-\xi|^{-1}\).  Also, by the maximum principle,
		\begin{equation}\label{eq2.N3-pointwise}
			0\le F_{\lambda,\xi}(x)
			\le\lambda^{1/2}U_{\lambda,\xi}(x)
			\le\alpha_3|x-\xi|^{-1}.
		\end{equation}
		
		We next make the convergence away from the pole uniform in the centre.  Let \(d_0=\operatorname{dist}(K,\partial\Omega)>0\), fix \(0<\delta<d_0/4\), and put \(D_{\delta,\xi}=\Omega\setminus B(\xi,\delta)\).  The Green representation gives
		\begin{equation}\label{eq2.N3-Green}
			F_{\lambda,\xi}(x)
			=\int_\Omega G_\Omega(x,y)\,d\mu_{\lambda,\xi}(y),
			\qquad
			d\mu_{\lambda,\xi}(y)
			=\lambda^{1/2}U_{\lambda,\xi}(y)^5\,dy.
		\end{equation}
		After the change of variables \(z=\lambda(y-\xi)\), uniformly for \(\xi\in K\),
		\begin{equation}\label{eq2.N3-mass}
			\mu_{\lambda,\xi}(\Omega)=A_3+O(\lambda^{-2}),
			\qquad
			\mu_{\lambda,\xi}(\Omega\setminus B(\xi,\rho))
			=O((\lambda\rho)^{-2})
		\end{equation}
		for each fixed \(0<\rho<\delta/2\).  On the part \(|y-\xi|<\rho\), the kernel \(G_\Omega(x,y)\) is uniformly continuous in \(y\) for \(x\in\overline\Omega\) with \(|x-\xi|\ge\delta\).  On the complement, using \(G_\Omega(x,y)\le C|x-y|^{-1}\),
		\[
		\lambda^{1/2}U_{\lambda,\xi}(y)^5
		\le C\lambda^{-2}\rho^{-5},
		\qquad |y-\xi|\ge\rho,
		\]
		and \(\sup_x\int_\Omega|x-y|^{-1}\,dy<\infty\).  It follows from \eqref{eq2.N3-Green}-\eqref{eq2.N3-mass}, first by letting \(\lambda\to\infty\) and then \(\rho\downarrow0\), that
		\begin{equation}\label{eq2.N3-far}
			\sup_{\xi\in K}\sup_{x\in D_{\delta,\xi}}
			|F_{\lambda,\xi}(x)-Q_\xi(x)|\to0.
		\end{equation}
		
		Let \(N_{\delta,\xi}=\Omega\cap B(\xi,\delta)\) and \(e_{\lambda,\xi}=F_{\lambda,\xi}-Q_\xi\).  By \eqref{eq2.N3-pointwise},
		\begin{equation}\label{eq2.N3-near}
			\sup_{\lambda\ge1}\sup_{\xi\in K}
			\int_{N_{\delta,\xi}}a(x)
			\bigl(F_{\lambda,\xi}^2+Q_\xi^2\bigr)\,dx
			\le C\int_0^\delta dr
			\le C\delta.
		\end{equation}
		On \(D_{\delta,\xi}\), \eqref{eq2.N3-far} gives
		\[
		\sup_{\xi\in K}\int_{D_{\delta,\xi}}
		a(x)e_{\lambda,\xi}^2\,dx=o_\delta(1).
		\]
		Letting first \(\lambda\to\infty\) and then \(\delta\downarrow0\), we obtain
		\begin{equation}\label{eq2.N3-L2}
			\sup_{\xi\in K}\int_\Omega
			a(x)e_{\lambda,\xi}^2\,dx\to0.
		\end{equation}
		The family \(Q_\xi\) is uniformly bounded in \(L^2(\Omega,a\,dx)\), so Cauchy-Schwarz also yields
		\begin{equation}\label{eq2.N3-cross}
			\sup_{\xi\in K}
			\left|\int_\Omega a(x)Q_\xi e_{\lambda,\xi}\,dx\right|
			\to0.
		\end{equation}
		Consequently, uniformly for \(\xi\in K\),
		\[
		\begin{aligned}
			\lambda E_{\rm pot}(\lambda,\xi)
			&=\frac12\int_\Omega a(x)(Q_\xi+e_{\lambda,\xi})^2\,dx\\
			&=\frac12\int_\Omega a(x)Q_\xi^2\,dx+o(1)
			=\mathfrak p_3(\xi)+o(1).
		\end{aligned}
		\]
		This proves \eqref{eq2.12} for \(N=3\) and completes the proof.
	\end{proof}

	\section{Palais-Smale decomposition and deformation to bubble tubes}\label{sec:PS}

	Throughout this section \(\Omega\Subset\mathbb B^N\) is the image of the given bounded \(C^2\) hyperbolic domain in the Poincar\'e ball model.  We use the notation of Section~\ref{sec:geom}.  In particular,
	\[
	2^*=\frac{2N}{N-2},\qquad
	\phi(x)=\rho(x)^{\frac{N-2}{2}},\qquad
	\rho(x)=\frac{2}{1-|x|^2}.
	\]
	The map
	\[
	T:H^1_0(\Omega;g_{\Hh})\to H^1_0(\Omega),
	\qquad Tu=\phi u,
	\]
	is an isomorphism. The hyperbolic functional
	\[
	I(u)=
	\frac12\int_\Omega |\nabla_{\Hh}u|^2\,dV_{\Hh}
	-\frac1{2^*}\int_\Omega |u|^{2^*}\,dV_{\Hh}
	\]
	is transformed into
	\begin{equation*}
		J(v)=
		\frac12\int_\Omega\big(|\nabla v|^2+a(x)v^2\big)\,dx
		-\frac1{2^*}\int_\Omega |v|^{2^*}\,dx,
		\qquad v=Tu,
	\end{equation*}
	where
	\[
	a(x)=\frac{N(N-2)}{(1-|x|^2)^2}\in C^\infty(\overline\Omega),
	\qquad a>0.
	\]
	We write
	\[
	\|v\|_a^2=
	\int_\Omega\big(|\nabla v|^2+a(x)v^2\big)\,dx.
	\]
	The norms \(\|u\|_{H^1_0(\Omega;g_{\Hh})}\) and \(\|Tu\|_a\) are equivalent. Thus \((u_k)\) is a \((PS)_c\) sequence for \(I\) if and only if \((v_k)=(Tu_k)\) is a \((PS)_c\) sequence for \(J\).
	
	Let
	\[
	S=\inf_{0\ne w\in C_c^\infty(\mathbb R^N)}
	\frac{\displaystyle\int_{\mathbb R^N}|\nabla w|^2\,dx}
	{\left(\displaystyle\int_{\mathbb R^N}|w|^{2^*}\,dx\right)^{2/2^*}},
	\]
	and let \(U\) be the Talenti bubble normalized by
	\[
	-\Delta U=U^{2^*-1}\quad\hbox{in }\mathbb R^N,
	\qquad
	\int_{\mathbb R^N}|\nabla U|^2\,dx
	=
	\int_{\mathbb R^N}U^{2^*}\,dx
	=S^{N/2}.
	\]
	Set
	\begin{equation*}
		I_\infty=\frac1N S^{N/2}.
	\end{equation*}
	For \(\lambda>0\) and \(\xi\in\Omega\), let
	\[
	U_{\lambda,\xi}(x)=\lambda^{\frac{N-2}{2}}U(\lambda(x-\xi)),
	\]
	and let \(PU_{\lambda,\xi}\in H^1_0(\Omega)\) be the Dirichlet projection
	\[
	-\Delta PU_{\lambda,\xi}=U_{\lambda,\xi}^{2^*-1}
	\quad\hbox{in }\Omega,
	\qquad
	PU_{\lambda,\xi}=0
	\quad\hbox{on }\partial\Omega.
	\]
	
	\subsection{Weak limits, splitting, and concentration}
	
	\begin{definition}
		A sequence \((v_k)\subset H^1_0(\Omega)\) is a \((PS)_c\) sequence for \(J\) if
		\[
		J(v_k)\to c,
		\qquad
		J'(v_k)\to0
		\quad\hbox{in }H^{-1}(\Omega).
		\]
	\end{definition}
	
	\begin{lemma}\label{Lem3.2}
		Every \((PS)_c\) sequence \((v_k)\) for \(J\) is bounded in \(H^1_0(\Omega)\). Up to a subsequence,
		\[
		v_k\rightharpoonup v\quad\hbox{weakly in }H^1_0(\Omega),
		\qquad
		v_k\to v\quad\hbox{strongly in }L^p(\Omega),\quad 1\le p<2^*,
		\]
		and \(v_k(x)\to v(x)\) almost everywhere in \(\Omega\). Moreover \(v\) is a weak solution of
		\begin{equation}\label{eq3.1}
			-\Delta v+a(x)v=|v|^{2^*-2}v
			\quad\hbox{in }\Omega,
			\qquad
			v=0
			\quad\hbox{on }\partial\Omega.
		\end{equation}
	\end{lemma}
	
	\begin{proof}
		For every \(v\in H^1_0(\Omega)\),
		\[
		\langle J'(v),v\rangle
		=
		\|v\|_a^2-\int_\Omega |v|^{2^*}\,dx.
		\]
		Therefore
		\[
		J(v)-\frac1{2^*}\langle J'(v),v\rangle
		=
		\left(\frac12-\frac1{2^*}\right)\|v\|_a^2
		=
		\frac1N\|v\|_a^2.
		\]
		If \((v_k)\) is a \((PS)_c\) sequence, then
		\[
		\frac1N\|v_k\|_a^2
		=
		J(v_k)-\frac1{2^*}\langle J'(v_k),v_k\rangle
		\le C+o(1)\|v_k\|_a.
		\]
		Since \(\|\cdot\|_a\) is equivalent to the usual \(H^1_0\)-norm, this gives boundedness in \(H^1_0(\Omega)\). The stated weak convergence and the compact \(L^p\)-convergence for \(p<2^*\) follow after passing to a subsequence.
		
		We pass to the limit in the equation. Let \(\varphi\in C_c^\infty(\Omega)\). The linear terms pass to the limit by weak convergence in \(H^1_0\) and strong convergence in \(L^2\). For the nonlinear term, \((v_k)\) is bounded in \(L^{2^*}\), \(v_k\to v\) almost everywhere, and hence
		\[
		|v_k|^{2^*-2}v_k
		\rightharpoonup
		|v|^{2^*-2}v
		\quad\hbox{weakly in }L^{(2^*)'}(\Omega).
		\]
		This follows from the boundedness in \(L^{(2^*)'}\), Vitali's theorem on sets of finite measure after truncation, and uniqueness of the almost-everywhere limit. Passing to the limit in \(\langle J'(v_k),\varphi\rangle=o(1)\) gives the weak formulation of \eqref{eq3.1}. Density gives the identity for all \(\varphi\in H^1_0(\Omega)\).
	\end{proof}
	
	\begin{lemma}\label{Lem3.3}
		Let \(v_k\rightharpoonup v\) in \(H^1_0(\Omega)\) and set
		\[
		w_k=v_k-v.
		\]
		Then, after passing to a subsequence,
		\begin{align}
			\|v_k\|_a^2
			&=
			\|v\|_a^2+\|w_k\|_a^2+o(1),\notag\\
			\int_\Omega |v_k|^{2^*}\,dx
			&=
			\int_\Omega |v|^{2^*}\,dx
			+
			\int_\Omega |w_k|^{2^*}\,dx
			+o(1),\label{eq3.2}\\
			J(v_k)
			&=
			J(v)+J(w_k)+o(1).\label{eq3.3}
		\end{align}
		If, in addition, \(J'(v_k)\to0\) and \(J'(v)=0\), then
		\begin{equation*}
			J'(w_k)\to0
			\quad\hbox{in }H^{-1}(\Omega).
		\end{equation*}
	\end{lemma}
	
	\begin{proof}
		Since \(w_k\rightharpoonup0\) in \(H^1_0(\Omega)\),
		\[
		\|v_k\|_a^2
		=
		\|v\|_a^2+\|w_k\|_a^2
		+
		2\int_\Omega(\nabla v\cdot\nabla w_k+a(x)vw_k)\,dx
		=
		\|v\|_a^2+\|w_k\|_a^2+o(1).
		\]
		The identity \eqref{eq3.2} is the Brezis-Lieb lemma on the finite measure space \((\Omega,dx)\) \cite{BrezisLieb1983}. Hence \eqref{eq3.3} follows.
		
		For the derivative splitting, write for \(\psi\in H^1_0(\Omega)\)
		\[
		\langle J'(v_k)-J'(v)-J'(w_k),\psi\rangle
		=
		-\int_\Omega R_k\psi\,dx,
		\]
		where
		\[
		R_k=
		|v_k|^{2^*-2}v_k
		-
		|v|^{2^*-2}v
		-
		|w_k|^{2^*-2}w_k.
		\]
		The following elementary nonlinear decoupling estimate gives
		\[
		R_k\to0
		\quad\hbox{in }L^{(2^*)'}(\Omega).
		\]
		Put \(p=2^*\) and \(p'=p/(p-1)\).  The pointwise inequality
		\[
		\big||s+t|^{p-2}(s+t)-|s|^{p-2}s-|t|^{p-2}t\big|
		\le
		C\big(|s|^{p-2}|t|+|t|^{p-2}|s|\big)
		\]
		and Young's inequality imply that, for every \(\eta>0\),
		\[
		\big||s+t|^{p-2}(s+t)-|s|^{p-2}s-|t|^{p-2}t\big|^{p'}
		\le \eta|t|^p+C_\eta|s|^p.
		\]
		Since \(w_k\to0\) almost everywhere, \(R_k\to0\) almost everywhere, and
		\[
		\bigl(|R_k|^{p'}-\eta|w_k|^p\bigr)_+
		\le C_\eta|v|^p.
		\]
		Dominated convergence and the boundedness of \(w_k\) in \(L^p\) therefore give
		\[
		\limsup_{k\to\infty}\|R_k\|_{p'}^{p'}
		\le \eta\sup_k\|w_k\|_p^p.
		\]
		Letting \(\eta\downarrow0\) proves the asserted strong convergence.  Thus
		\[
		\sup_{\|\psi\|_a\le1}
		\left|\int_\Omega R_k\psi\,dx\right|
		\le
		C\|R_k\|_{L^{(2^*)'}(\Omega)}
		\to0.
		\]
		Consequently
		\[
		J'(w_k)=J'(v_k)-J'(v)+o(1)=o(1)
		\quad\hbox{in }H^{-1}(\Omega).
		\]
	\end{proof}
	
	\begin{lemma}\label{Lem3.4}
		Let \((z_k)\subset H^1_0(\Omega)\) be bounded, \(z_k\rightharpoonup0\) in \(H^1_0(\Omega)\), and
		\[
		J'(z_k)\to0
		\quad\hbox{in }H^{-1}(\Omega).
		\]
		Extend \(z_k\) by zero outside \(\Omega\). After passing to a subsequence,
		\[
		|\nabla z_k|^2\,dx \rightharpoonup \mu,
		\qquad
		|z_k|^{2^*}\,dx \rightharpoonup \nu
		\]
		weakly as finite Radon measures on \(\mathbb R^N\). Then the defect measure \(\nu\) is purely atomic and supported in \(\overline\Omega\):
		\[
		\nu=\sum_{j\in J}\nu_j\delta_{x_j},
		\qquad x_j\in\overline\Omega,\quad \nu_j>0,
		\]
		with at most countably many atoms. Moreover, for every atom,
		\[
		\nu_j\ge S^{N/2}.
		\]
		Consequently, if \(K\subset\mathbb R^N\) is compact and there exist \(r_0>0\), \(\eta_0<S^{N/2}\) such that
		\[
		\limsup_{k\to\infty}
		\int_{B(x,r)\cap\Omega}|z_k|^{2^*}\,dx
		\le \eta_0
		\]
		for every \(x\in K\) and every \(0<r<r_0\) with \(\nu(\partial B(x,r))=0\), then
		\[
		z_k\to0
		\quad\hbox{strongly in }L^{2^*}(K\cap\Omega).
		\]
	\end{lemma}
	
	\begin{proof}
		The zero extensions are bounded in \(D^{1,2}(\mathbb R^N)\).  The limit-case concentration-compactness theorem \cite[Lemma~I.1, pp.~158-159]{LionsLimit1985} gives, after passing to a subsequence, finite Radon measures \(\mu,\nu\) and at most countably many atoms \(x_j\in\overline\Omega\) such that
		\[
		\nu=\sum_{j\in J}\nu_j\delta_{x_j},
		\qquad
		\mu\ge\sum_{j\in J}\mu_j\delta_{x_j},
		\qquad
		S\nu_j^{2/2^*}\le \mu_j.
		\]
		Here the weak limit is zero, and therefore there is no absolutely continuous part associated with a nonzero weak limit.
		
		We compare \(\mu_j\) and \(\nu_j\). Fix an atom \(x_j\) and choose \(\eta_\varepsilon\in C_c^\infty(\mathbb R^N)\) such that
		\[
		0\le\eta_\varepsilon\le1,\qquad
		\eta_\varepsilon\equiv1 \hbox{ on } B_\varepsilon(x_j),\qquad
		\operatorname{supp}\eta_\varepsilon\subset B_{2\varepsilon}(x_j),
		\qquad
		|\nabla\eta_\varepsilon|\le C\varepsilon^{-1}.
		\]
		Since \(z_k\in H^1_0(\Omega)\), the function \(\eta_\varepsilon z_k\) belongs to \(H^1_0(\Omega)\), even if \(x_j\in\partial\Omega\). Testing \(J'(z_k)=o(1)\) with \(\eta_\varepsilon z_k\) gives
		\[
		\int_\Omega \nabla z_k\cdot\nabla(\eta_\varepsilon z_k)\,dx
		+
		\int_\Omega a(x)\eta_\varepsilon z_k^2\,dx
		-
		\int_\Omega \eta_\varepsilon |z_k|^{2^*}\,dx
		=o(1).
		\]
		For fixed \(\varepsilon\), \(z_k\to0\) strongly in \(L^2(\Omega)\), and \((z_k)\) is bounded in \(H^1_0(\Omega)\). Hence
		\[
		\int_\Omega z_k\nabla z_k\cdot\nabla\eta_\varepsilon\,dx=o(1),
		\qquad
		\int_\Omega a(x)\eta_\varepsilon z_k^2\,dx=o(1).
		\]
		Thus
		\[
		\int_\Omega \eta_\varepsilon |\nabla z_k|^2\,dx
		=
		\int_\Omega \eta_\varepsilon |z_k|^{2^*}\,dx+o(1).
		\]
		Passing to the weak limits of measures and then letting \(\varepsilon\downarrow0\), we obtain
		\[
		\mu(\{x_j\})=\nu(\{x_j\}).
		\]
		Therefore \(\mu_j\le\nu_j\). Combining this with \(S\nu_j^{2/2^*}\le\mu_j\) yields
		\[
		S\nu_j^{2/2^*}\le\nu_j.
		\]
		If \(\nu_j>0\), then
		\[
		\nu_j^{2/N}\ge S,
		\]
		and hence
		\[
		\nu_j\ge S^{N/2}.
		\]
		
		The final assertion follows immediately. If strong convergence failed on \(K\cap\Omega\), then \(\nu\) would have positive mass in \(K\). Since \(\nu\) is purely atomic, some atom \(x_j\in K\cap\overline\Omega\) would be present, and its mass would be at least \(S^{N/2}\), contradicting the assumed small-mass bound.
	\end{proof}
	
	\begin{lemma}\label{Lem3.5}
		Let \(\Omega_k\subset\mathbb R^N\) converge locally in \(C^1\) either to \(\mathbb R^N\) or to a half-space \(\mathcal H\).  Let \(Z_k\in D^{1,2}_0(\Omega_k)\), and denote its zero extension to \(\mathbb R^N\) by \(\overline Z_k\).  Assume
		\[
		\sup_k\|\nabla\overline Z_k\|_{L^2(\mathbb R^N)}<\infty,
		\qquad
		\overline Z_k\rightharpoonup0
		\quad\hbox{in }D^{1,2}_{\mathrm{loc}}(\mathbb R^N).
		\]
		Assume also the following local dual residual condition: for every compact \(Q\subset\mathbb R^N\) and every sequence \(\psi_k\in D^{1,2}_0(\Omega_k)\) satisfying
		\[
		\operatorname{supp}\psi_k\subset Q,
		\qquad
		\sup_k\|\nabla\psi_k\|_{L^2(\Omega_k)}<\infty,
		\]
		one has
		\begin{equation}\label{eq3.local-residual}
			\int_{\Omega_k}\nabla Z_k\cdot\nabla\psi_k\,dz
			-
			\int_{\Omega_k}|Z_k|^{2^*-2}Z_k\psi_k\,dz
			=o(1).
		\end{equation}
		If, for some \(\rho>0\) and \(\eta<S^{N/2}\),
		\[
		\sup_{y\in\mathbb R^N}
		\int_{B(y,\rho)\cap\Omega_k}|Z_k|^{2^*}\,dz
		\le\eta
		\]
		for all sufficiently large \(k\), then
		\[
		\int_{L\cap\Omega_k}|Z_k|^{2^*}\,dz\to0
		\]
		for every compact \(L\subset\mathbb R^N\).  In the half-space case this includes compact sets meeting \(\partial\mathcal H\).
	\end{lemma}
	
	\begin{proof}
		Fix a compact \(L\subset\mathbb R^N\).  Choose a bounded smooth domain \(D\) and a cutoff \(\chi\in C_c^\infty(D)\) such that
		\[
		L\Subset\{\chi=1\}\Subset D.
		\]
		Set
		\[
		Y_k=\chi\,\overline Z_k.
		\]
		This localization is performed before the restriction to \(D\).  Since \(\overline Z_k\in D^{1,2}(\mathbb R^N)\) and \(\chi\) vanishes near \(\partial D\),
		\[
		Y_k\in H^1_0(D).
		\]
		Moreover, \((Y_k)\) is bounded in \(H^1_0(D)\), and the assumed weak convergence of the zero extensions gives
		\[
		Y_k\rightharpoonup0\quad\hbox{in }H^1_0(D).
		\]
		
		After passing to a subsequence,
		\[
		|\nabla Y_k|^2\,dz\rightharpoonup\mu,
		\qquad
		|Y_k|^{2^*}\,dz\rightharpoonup\nu
		\]
		as finite Radon measures on \(D\).  The limit-case lemma on concentration compactness \cite[Lemma~I.1, pp.~158-159]{LionsLimit1985} gives
		\[
		\nu=\sum_{\ell\in\mathcal J}\nu_\ell\delta_{y_\ell},
		\qquad
		\mu\ge\sum_{\ell\in\mathcal J}\mu_\ell\delta_{y_\ell},
		\qquad
		S\nu_\ell^{2/2^*}\le\mu_\ell.
		\]
		
		Consider an atom \(y_\ell\) lying in \(\{\chi=1\}\).  Choose \(\eta_\varepsilon\in C_c^\infty(\{\chi=1\})\) such that
		\[
		\eta_\varepsilon=1\ \hbox{on }B_\varepsilon(y_\ell),
		\qquad
		\operatorname{supp}\eta_\varepsilon
		\subset B_{2\varepsilon}(y_\ell),
		\qquad
		|\nabla\eta_\varepsilon|\le C\varepsilon^{-1}.
		\]
		The functions
		\[
		\psi_k=\eta_\varepsilon Z_k
		\]
		belong to \(D^{1,2}_0(\Omega_k)\), have support in one fixed compact set, and are bounded in \(D^{1,2}\).  They are therefore admissible in \eqref{eq3.local-residual}.  Since \(Y_k=Z_k\) on \(\operatorname{supp}\eta_\varepsilon\), and \(Y_k\to0\) strongly in \(L^2(D)\), we obtain
		\[
		\int_{\Omega_k}
		Z_k\nabla Z_k\cdot\nabla\eta_\varepsilon\,dz=o(1).
		\]
		Therefore
		\[
		\int_{\Omega_k}\eta_\varepsilon|\nabla Z_k|^2\,dz
		=
		\int_{\Omega_k}\eta_\varepsilon|Z_k|^{2^*}\,dz+o(1).
		\]
		Passing first to \(k\to\infty\) and then to \(\varepsilon\downarrow0\) yields
		\[
		\mu(\{y_\ell\})=\nu(\{y_\ell\}).
		\]
		Hence \(\mu_\ell\le\nu_\ell\), and therefore
		\[
		\nu_\ell=0
		\quad\hbox{or}\quad
		\nu_\ell\ge S^{N/2}.
		\]
		
		The local mass hypothesis rules out the second alternative.  Indeed, choose \(r<\rho\) with \(\nu(\partial B(y_\ell,r))=0\).  Since \(|Y_k|\le|\overline Z_k|\),
		\[
		\nu_\ell
		\le
		\lim_{k\to\infty}
		\int_{B(y_\ell,r)\cap\Omega_k}|Z_k|^{2^*}\,dz
		\le\eta<S^{N/2},
		\]
		so no atom lies in \(\{\chi=1\}\).  Therefore
		\[
		\int_{L\cap\Omega_k}|Z_k|^{2^*}\,dz
		=
		\int_L|Y_k|^{2^*}\,dz
		\to0.
		\]
	\end{proof}
	
	\begin{lemma}\label{Lem3.6}
		Let \((w_k)\) be bounded in \(H^1_0(\Omega)\), \(w_k\rightharpoonup0\) in \(H^1_0(\Omega)\), and \(J'(w_k)\to0\) in \(H^{-1}(\Omega)\). If \(w_k\not\to0\) in \(L^{2^*}(\Omega)\), then there exist points \(x_k\in\Omega\), radii \(r_k\downarrow0\), and a number
		\[
		0<\eta<S^{N/2}
		\]
		such that
		\begin{equation}\label{eq3.4}
			\int_{B(x_k,r_k)\cap\Omega}|w_k|^{2^*}\,dx\ge\eta,
			\qquad
			\sup_{y\in\mathbb R^N}
			\int_{B(y,r_k/4)\cap\Omega}|w_k|^{2^*}\,dx
			\le\eta.
		\end{equation}
	\end{lemma}
	
	\begin{proof}
		Since \(w_k\not\to0\) in \(L^{2^*}\), after passing to a subsequence there is \(\delta_0>0\) such that
		\[
		\int_\Omega |w_k|^{2^*}\,dx\ge2\delta_0.
		\]
		Choose
		\[
		0<\eta<\min\{\delta_0,S^{N/2}\}.
		\]
		For \(r\ge0\) define
		\[
		Q_k(r)=
		\sup_{x\in\overline\Omega}
		\int_{B(x,r)\cap\Omega}|w_k|^{2^*}\,dx.
		\]
		The function \(Q_k\) is continuous, \(Q_k(0)=0\), and \(Q_k(\operatorname{diam}\Omega)\ge2\delta_0\). Hence there is \(s_k>0\) such that \(Q_k(s_k)=\eta\). Choose \(\bar x_k\in\overline\Omega\) at which the supremum is attained.  Since \(\Omega\) is open and dense in \(\overline\Omega\), choose \(x_k\in\Omega\) with \(|x_k-\bar x_k|<s_k\), and put \(r_k=2s_k\).  Then \(B(\bar x_k,s_k)\subset B(x_k,r_k)\), so the first inequality in \eqref{eq3.4} holds, while
		\[
		\sup_{x\in\Omega}\int_{B(x,r_k/2)\cap\Omega}|w_k|^{2^*}\,dx
		=Q_k(s_k)=\eta .
		\]
		If \(B(y,r_k/4)\cap\Omega\ne\varnothing\), choose \(x\in B(y,r_k/4)\cap\Omega\).  Then \(B(y,r_k/4)\cap\Omega\subset B(x,r_k/2)\cap\Omega\), so the preceding \(\Omega\)-centred bound proves the second inequality in \eqref{eq3.4}; the empty-intersection case is immediate.
		
		We show \(r_k\to0\). Suppose not. Then, after passing to a subsequence, \(r_k\ge r_*>0\). Hence
		\[
		\sup_{x\in\overline\Omega}
		\int_{B(x,r_*/2)\cap\Omega}|w_k|^{2^*}\,dx
		\le\eta<S^{N/2}.
		\]
		By Lemma~\ref{Lem3.4}, no concentration atom can occur in \(\overline\Omega\). Therefore \(w_k\to0\) strongly in \(L^{2^*}(\Omega)\), contradicting the assumption. Thus \(r_k\downarrow0\), after passing to a subsequence.
	\end{proof}
	
	\subsection{Bubble extraction and boundary exclusion}
	
	\begin{proposition}\label{Prop3.7}
		Let \((w_k)\) be bounded in \(H^1_0(\Omega)\), \(w_k\rightharpoonup0\), and \(J'(w_k)\to0\) in \(H^{-1}(\Omega)\). Let \((x_k,r_k)\) be chosen as in Lemma~\ref{Lem3.6}, and set
		\[
		\lambda_k=r_k^{-1}.
		\]
		Then
		\begin{equation*}
			\lambda_k\operatorname{dist}(x_k,\partial\Omega)\to+\infty.
		\end{equation*}
	\end{proposition}
	
	\begin{proof}
		Assume by contradiction that
		\[
		\lambda_k d_k\le C,
		\qquad
		d_k=\operatorname{dist}(x_k,\partial\Omega).
		\]
		Since \(r_k\downarrow0\) and \(\lambda_k=r_k^{-1}\), we have \(d_k\to0\), and hence, after passing to a subsequence,
		\[
		x_k\to x_0\in\partial\Omega.
		\]
		Define
		\[
		\Omega_k=\lambda_k(\Omega-x_k),
		\qquad
		\widetilde w_k(z)=
		\lambda_k^{-\frac{N-2}{2}}w_k(x_k+z/\lambda_k),
		\qquad z\in\Omega_k,
		\]
		and extend \(\widetilde w_k\) by zero outside \(\Omega_k\). The sequence \((\widetilde w_k)\) is bounded in \(D^{1,2}_0(\Omega_k)\), because the Dirichlet norm is invariant under the critical scaling. By Lemma~\ref{Lem2.3}, after a rigid motion the domains \(\Omega_k\) converge in \(C^1_{\mathrm{loc}}\) to a half-space \(\mathcal H\).
		
		Let \(\varphi\in C_c^\infty(\mathcal H)\). For large \(k\),
		\[
		\varphi_k(x)=
		\lambda_k^{\frac{N-2}{2}}\varphi(\lambda_k(x-x_k))
		\]
		belongs to \(H^1_0(\Omega)\). Testing \(J'(w_k)=o(1)\) with \(\varphi_k\) and changing variables gives
		\[
		\int_{\Omega_k}\nabla\widetilde w_k\cdot\nabla\varphi\,dz
		+
		\lambda_k^{-2}\int_{\Omega_k}
		a(x_k+z/\lambda_k)\widetilde w_k\varphi\,dz
		-
		\int_{\Omega_k}
		|\widetilde w_k|^{2^*-2}\widetilde w_k\varphi\,dz
		=o(1).
		\]
		The potential term is \(o(1)\), since \(a\in L^\infty(\Omega)\) and \(\lambda_k^{-2}\to0\).  Extend \(\widetilde w_k\) by zero outside \(\Omega_k\).  These extensions are bounded in \(D^{1,2}(\mathbb R^N)\).  Passing to a subsequence,
		\[
		\widetilde w_k\rightharpoonup W
		\quad\text{weakly in }D^{1,2}_{\mathrm{loc}}(\mathcal H).
		\]
		Local compactness gives strong convergence in \(L^q_{\mathrm{loc}}(\mathcal H)\) for every \(q<2^*\) and almost-everywhere convergence.  Hence
		\[
		|\widetilde w_k|^{2^*-2}\widetilde w_k
		\rightharpoonup |W|^{2^*-2}W
		\quad\text{in }L^{(2^*)'}_{\mathrm{loc}}(\mathcal H).
		\]
		The weak limit of the zero extensions vanishes on the complementary half-space.  By the zero-extension characterization of the homogeneous Dirichlet space, lower semicontinuity and exhaustion give
		\[
		W\in D^{1,2}_0(\mathcal H).
		\]
		The preceding weak formulation gives
		\[
		-\Delta W=|W|^{2^*-2}W
		\quad\text{in }\mathcal H,
		\qquad
		W=0
		\quad\text{on }\partial\mathcal H.
		\]
		
		More generally, let \(Q\subset\mathbb R^N\) be compact and let \(\psi_k\in D^{1,2}_0(\Omega_k)\) be supported in \(Q\), with \(\sup_k\|\nabla\psi_k\|_2<\infty\).  The rescaled functions
		\[
		\Psi_k(x)=
		\lambda_k^{\frac{N-2}{2}}\psi_k(\lambda_k(x-x_k))
		\]
		are bounded in \(H^1_0(\Omega)\).  Testing \(J'(w_k)\) with \(\Psi_k\) therefore proves \eqref{eq3.local-residual}; its rescaled potential term is \(O_Q(\lambda_k^{-2})\).  If \(W=0\), every weak limit of the zero extensions is supported in \(\overline{\mathcal H}\), vanishes in \(\mathcal H\), and is zero almost everywhere outside \(\mathcal H\).  Hence those extensions converge weakly to zero in \(D^{1,2}_{\mathrm{loc}}(\mathbb R^N)\).
		
		We prove \(W\not\equiv0\). The normalization from Lemma~\ref{Lem3.6} becomes
		\[
		\int_{B(0,1)\cap\Omega_k}|\widetilde w_k|^{2^*}\,dz\ge\eta,
		\qquad
		\sup_{y\in\mathbb R^N}
		\int_{B(y,1/4)\cap\Omega_k}|\widetilde w_k|^{2^*}\,dz
		\le \eta<S^{N/2}.
		\]
		If \(W=0\), all hypotheses of Lemma~\ref{Lem3.5} are satisfied on compact subsets of the half-space, with \(\rho=1/4\). Therefore
		\[
		\widetilde w_k\to0
		\quad\text{strongly in }L^{2^*}(B(0,1)\cap\Omega_k),
		\]
		contradicting the first lower bound. Hence \(W\not\equiv0\).
		
		This contradicts the half-space Liouville theorem Lemma~\ref{Lem2.4}, which excludes every finite-energy Dirichlet solution in a half-space, including sign-changing ones. Therefore
		\[
		\lambda_k\operatorname{dist}(x_k,\partial\Omega)\to+\infty.
		\]
	\end{proof}
	
	\begin{lemma}\label{Lem3.8}
		Let
		\[
		B_k=PU_{\mu_k,\eta_k},
		\qquad
		\mu_k\operatorname{dist}(\eta_k,\partial\Omega)\to\infty,
		\]
		and let \((\lambda_k,\xi_k)\) be another scale-interior parameter sequence. If
		\[
		\frac{\lambda_k}{\mu_k}
		+
		\frac{\mu_k}{\lambda_k}
		+
		\lambda_k\mu_k|\xi_k-\eta_k|^2
		\to\infty,
		\]
		then, after zero extension,
		\[
		\lambda_k^{-\frac{N-2}{2}}
		B_k(\xi_k+\cdot/\lambda_k)
		\rightharpoonup0
		\quad\hbox{in }D^{1,2}_{\mathrm{loc}}(\mathbb R^N)
		\]
		and weakly in \(L^{2^*}_{\mathrm{loc}}(\mathbb R^N)\).
	\end{lemma}
	
	\begin{proof}
		The scale-interior condition and the Dirichlet projection property give
		\[
		\left\|
		\overline{PU_{\mu_k,\eta_k}}-U_{\mu_k,\eta_k}
		\right\|_{D^{1,2}(\mathbb R^N)}
		\to0.
		\]
		Indeed, \(PU_{\mu_k,\eta_k}\) is the Dirichlet-orthogonal projection of \(U_{\mu_k,\eta_k}|_\Omega\), and a cutoff which is one on \(B(\eta_k,\tfrac12\operatorname{dist}(\eta_k,\partial\Omega))\) approximates \(U_{\mu_k,\eta_k}\) in \(D^{1,2}\).
		
		Put
		\[
		q_k=\frac{\mu_k}{\lambda_k},
		\qquad
		\mathbf y_k=\lambda_k(\eta_k-\xi_k).
		\]
		In the \((\lambda_k,\xi_k)\)-variables the unprojected bubble is
		\[
		q_k^{\frac{N-2}{2}}U(q_k(z-\mathbf y_k)),
		\]
		and orthogonality becomes
		\[
		q_k+q_k^{-1}+q_k|\mathbf y_k|^2\to\infty.
		\]
		After passing to a subsequence, either \(q_k\to0\), or \(q_k\to\infty\), or \(q_k\) stays bounded above and below while \(|\mathbf y_k|\to\infty\). In the first case the dilates vanish locally, in the second they concentrate and converge weakly to zero, and in the third they escape every compact set. This proves both weak convergences.
	\end{proof}
	
	Let \((w_k)\) be as in Proposition~\ref{Prop3.7} and assume \(w_k\not\to0\) in \(L^{2^*}(\Omega)\). Let \((x_k,r_k)\) be chosen by Lemma~\ref{Lem3.6}, set
	\[
	\lambda_k=r_k^{-1},
	\qquad
	\xi_k=x_k,
	\]
	and define
	\[
	\widetilde w_k(z)=
	\lambda_k^{-\frac{N-2}{2}}w_k(\xi_k+z/\lambda_k),
	\qquad
	\Omega_k=\lambda_k(\Omega-\xi_k).
	\]
	
	\begin{lemma}\label{Lem3.9}
		With the preceding notation,
		\[
		\lambda_k\operatorname{dist}(\xi_k,\partial\Omega)\to+\infty,
		\]
		\(\Omega_k\) exhausts \(\mathbb R^N\), and, after passing to a subsequence,
		\[
		\widetilde w_k\rightharpoonup V
		\quad\hbox{weakly in }D^{1,2}_{\mathrm{loc}}(\mathbb R^N),
		\]
		where \(V\in D^{1,2}(\mathbb R^N)\), \(V\not\equiv0\), and
		\begin{equation*}
			-\Delta V=|V|^{2^*-2}V
			\quad\hbox{in }\mathbb R^N.
		\end{equation*}
		Moreover, if this profile is extracted in the inductive procedure from an original nonnegative sequence \((v_k)\) and is orthogonal to all previously extracted bubbles, then \(V\ge0\). In that case \(V\) is a positive Talenti bubble:
		\[
		V(z)=\mu^{\frac{N-2}{2}}U(\mu(z-z_0))
		\]
		for some \(\mu>0\) and \(z_0\in\mathbb R^N\).
	\end{lemma}
	
	\begin{proof}
		The boundary separation follows from Proposition~\ref{Prop3.7}. Thus
		\[
		\Omega_k=\lambda_k(\Omega-\xi_k)
		\]
		exhausts \(\mathbb R^N\). After zero extension, the rescaled sequence is bounded in \(D^{1,2}(\mathbb R^N)\). Hence, after passing to a subsequence,
		\[
		\widetilde w_k\rightharpoonup V
		\quad\text{weakly in }D^{1,2}_{\mathrm{loc}}(\mathbb R^N).
		\]
		For every \(R>0\),
		\[
		\int_{B_R}|\nabla V|^2\,dz
		\le
		\liminf_{k\to\infty}
		\int_{B_R\cap\Omega_k}|\nabla\widetilde w_k|^2\,dz
		\le C.
		\]
		Letting \(R\to\infty\) gives
		\[
		V\in D^{1,2}(\mathbb R^N).
		\]
		
		Testing \(J'(w_k)=o(1)\) against
		\[
		\varphi_k(x)=
		\lambda_k^{\frac{N-2}{2}}\varphi(\lambda_k(x-\xi_k)),
		\qquad
		\varphi\in C_c^\infty(\mathbb R^N),
		\]
		and changing variables gives
		\[
		\int_{\Omega_k}\nabla\widetilde w_k\cdot\nabla\varphi\,dz
		+
		\lambda_k^{-2}\int_{\Omega_k}
		a(\xi_k+z/\lambda_k)\widetilde w_k\varphi\,dz
		-
		\int_{\Omega_k}
		|\widetilde w_k|^{2^*-2}\widetilde w_k\varphi\,dz
		=o(1).
		\]
		The potential term tends to zero. Passing to the limit gives
		\[
		-\Delta V=|V|^{2^*-2}V
		\quad\text{in }\mathbb R^N.
		\]
		
		The same computation is uniform for every sequence \(\psi_k\in D^{1,2}_0(\Omega_k)\) supported in a fixed compact set and bounded in \(D^{1,2}\): after scaling \(\psi_k\) back to \(\Omega\), it is an admissible bounded test sequence for \(J'(w_k)=o(1)\), while the potential term is \(O(\lambda_k^{-2})\).  Thus \eqref{eq3.local-residual} holds.  If \(V=0\), the zero extensions converge weakly to zero in \(D^{1,2}_{\mathrm{loc}}(\mathbb R^N)\).
		
		The profile is nontrivial:
		\[
		\int_{B(0,1)\cap\Omega_k}|\widetilde w_k|^{2^*}\,dz\ge\eta,
		\qquad
		\sup_{y\in\mathbb R^N}
		\int_{B(y,1/4)\cap\Omega_k}|\widetilde w_k|^{2^*}\,dz
		\le\eta<S^{N/2}.
		\]
		If \(V=0\), Lemma~\ref{Lem3.5}, with limiting domain \(\mathbb R^N\), gives strong convergence to zero in \(L^{2^*}(B(0,1))\), contradicting the normalization. Hence \(V\not\equiv0\).
		
		Assume now that this profile is extracted during the induction from an original nonnegative sequence \(v_k\ge0\), and that the new parameters are orthogonal to all previously extracted bubbles.  In the new scale,
		\[
		\lambda_k^{-\frac{N-2}{2}}v(\xi_k+z/\lambda_k)
		\to0
		\quad\hbox{strongly in }D^{1,2}_{\mathrm{loc}}(\mathbb R^N)
		\cap L^{2^*}_{\mathrm{loc}}(\mathbb R^N),
		\]
		by absolute continuity of the integrals of \(|\nabla v|^2\) and \(|v|^{2^*}\).  By Lemma~\ref{Lem3.8}, every previously extracted projected bubble converges weakly to zero in the new variables. Consequently the weak local limit of the rescaled remainder is the same as the weak local limit of the rescaled nonnegative functions \(v_k\). The nonnegative cone is weakly closed in \(L^{2^*}_{\mathrm{loc}}\); therefore \(V\ge0\).
		
		Since \(V\ge0\), \(V\not\equiv0\), and \(V\in D^{1,2}(\mathbb R^N)\), the strong maximum principle gives \(V>0\). By the Caffarelli-Gidas-Spruck classification theorem \cite{CGS1989}, \(V\) is a Talenti bubble. Absorbing its dilation and translation into the parameters \((\lambda_k,\xi_k)\), we normalize the profile to be \(U\).
	\end{proof}
	
	\begin{lemma}\label{Lem3.10}
		For every Talenti bubble
		\[
		V(z)=\mu^{\frac{N-2}{2}}U(\mu(z-z_0))
		\]
		one has
		\[
		\frac12\int_{\mathbb R^N}|\nabla V|^2\,dz
		-\frac1{2^*}\int_{\mathbb R^N}|V|^{2^*}\,dz
		=I_\infty.
		\]
	\end{lemma}
	
	\begin{proof}
		Multiplying \(-\Delta V=V^{2^*-1}\) by \(V\) gives
		\[
		\int_{\mathbb R^N}|\nabla V|^2\,dz
		=
		\int_{\mathbb R^N}V^{2^*}\,dz.
		\]
		Talenti bubbles attain the best Sobolev constant, so both integrals are \(S^{N/2}\). Hence
		\[
		\frac12\int |\nabla V|^2
		-\frac1{2^*}\int V^{2^*}
		=
		\left(\frac12-\frac1{2^*}\right)S^{N/2}
		=
		\frac1N S^{N/2}
		=I_\infty.
		\]
	\end{proof}
	
	\begin{lemma}\label{Lem3.11}
		Let \(\lambda_k\to+\infty\), \(\xi_k\in\Omega\), and
		\[
		\lambda_k\operatorname{dist}(\xi_k,\partial\Omega)\to+\infty.
		\]
		Set
		\[
		B_k=PU_{\lambda_k,\xi_k}.
		\]
		Then
		\begin{align}
			\int_\Omega|\nabla B_k|^2\,dx&=S^{N/2}+o(1),\notag\\
			\int_\Omega B_k^{2^*}\,dx&=S^{N/2}+o(1),\notag\\
			\int_\Omega a(x)B_k^2\,dx&=o(1),\notag\\
			\|B_k\|_a^2&=S^{N/2}+o(1),\label{eq3.5}\\
			J(B_k)&=I_\infty+o(1),\label{eq3.6}\\
			J'(B_k)&\to0
			\quad\hbox{in }H^{-1}(\Omega).\notag
			\notag\end{align}
		Moreover, in the variables \(z=\lambda_k(x-\xi_k)\),
		\[
		\lambda_k^{-\frac{N-2}{2}}B_k(\xi_k+z/\lambda_k)\to U(z)
		\]
		strongly in \(D^{1,2}_{\mathrm{loc}}(\mathbb R^N)\) and in \(L^{2^*}_{\mathrm{loc}}(\mathbb R^N)\).
	\end{lemma}
	
	\begin{proof}
		Put \(d_k=\operatorname{dist}(\xi_k,\partial\Omega)\) and \(R_k=\lambda_kd_k\).  Choose \(\chi_k\in C_c^\infty(\Omega)\) such that
		\[
		0\le\chi_k\le1,\qquad
		\chi_k=1\ \hbox{on }B(\xi_k,d_k/2),\qquad
		\operatorname{supp}\chi_k\subset B(\xi_k,3d_k/4),
		\qquad
		|\nabla\chi_k|\le C d_k^{-1}.
		\]
		The ordinary Dirichlet projection \(B_k=PU_{\lambda_k,\xi_k}\) is the \(D^{1,2}\)-orthogonal projection of \(U_{\lambda_k,\xi_k}|_\Omega\) onto \(H^1_0(\Omega)\).  Hence, after extending \(B_k\) and \(\chi_kU_{\lambda_k,\xi_k}\) by zero outside \(\Omega\),
		\[
		\begin{aligned}
			\|\overline B_k-U_{\lambda_k,\xi_k}\|_{D^{1,2}(\mathbb R^N)}^2
			&\le
			\|U_{\lambda_k,\xi_k}
			-\overline{\chi_kU_{\lambda_k,\xi_k}}\|_{D^{1,2}(\mathbb R^N)}^2\\
			&\le C\int_{|z|\ge R_k/2}|\nabla U(z)|^2\,dz
			+\frac{C}{R_k^2}
			\int_{R_k/2\le|z|\le3R_k/4}U(z)^2\,dz
			=o(1).
		\end{aligned}
		\]
		Here the last limit follows from \(R_k\to\infty\) and the Talenti decay. The Sobolev inequality therefore yields
		\begin{equation}\label{eq3.projected-scale-interior}
			\|\overline B_k-U_{\lambda_k,\xi_k}\|_{L^{2^*}(\mathbb R^N)}=o(1).
		\end{equation}
		After rescaling, these two estimates give the asserted strong local convergences to \(U\).
		
		Using
		\[
		-\Delta B_k=U_{\lambda_k,\xi_k}^{2^*-1}
		\quad\hbox{in }\Omega,
		\]
		we obtain
		\[
		\int_\Omega|\nabla B_k|^2\,dx
		=
		\int_\Omega U_{\lambda_k,\xi_k}^{2^*-1}B_k\,dx
		=
		S^{N/2}+o(1).
		\]
		Similarly, by \eqref{eq3.projected-scale-interior},
		\[
		\int_\Omega B_k^{2^*}\,dx
		=
		S^{N/2}+o(1).
		\]
		The maximum principle gives
		\[
		0\le B_k\le U_{\lambda_k,\xi_k}
		\quad\hbox{in }\Omega.
		\]
		Consequently,
		\[
		\int_\Omega a(x)B_k^2\,dx=o(1),
		\]
		because \(a\in L^\infty(\Omega)\) and the \(L^2\)-mass of a concentrating bubble in a bounded domain tends to zero.  The corresponding orders are \(O(\lambda_k^{-1})\) for \(N=3\), \(O(\lambda_k^{-2}\log\lambda_k)\) for \(N=4\), and \(O(\lambda_k^{-2})\) for \(N\ge5\).  Thus \eqref{eq3.5} and, using Lemma~\ref{Lem3.10}, \eqref{eq3.6} follow.
		
		Finally, for \(\psi\in H^1_0(\Omega)\),
		\[
		\langle J'(B_k),\psi\rangle
		=
		\int_\Omega a(x)B_k\psi\,dx
		+
		\int_\Omega
		\big(U_{\lambda_k,\xi_k}^{2^*-1}-B_k^{2^*-1}\big)\psi\,dx.
		\]
		The first term is \(o(1)\|\psi\|_a\) because \(B_k\to0\) in \(L^{(2^*)'}(\Omega)\).  For the nonlinear term, \eqref{eq3.projected-scale-interior} gives
		\[
		\|U_{\lambda_k,\xi_k}-B_k\|_{L^{2^*}(\Omega)}\to0,
		\]
		and hence, with \(p=2^*\),
		\[
		\begin{aligned}
			\|U_{\lambda_k,\xi_k}^{p-1}-B_k^{p-1}\|_{L^{p'}(\Omega)}
			&\le
			C\bigl(\|U_{\lambda_k,\xi_k}\|_{L^p}^{p-2}
			+\|B_k\|_{L^p}^{p-2}\bigr)
			\|U_{\lambda_k,\xi_k}-B_k\|_{L^p}  \\
			&=o(1).
		\end{aligned}
		\]
		By Sobolev embedding, the second term in the display for \(\langle J'(B_k),\psi\rangle\) is therefore \(o(1)\|\psi\|_a\).  Thus \(J'(B_k)\to0\) in \(H^{-1}(\Omega)\).
	\end{proof}
	
	\begin{lemma}\label{Lem3.12}
		Let \(w_k\rightharpoonup0\) in \(H^1_0(\Omega)\), and suppose that, for some \(\lambda_k\to+\infty\) and \(\xi_k\in\Omega\) with
		\[
		\lambda_k\operatorname{dist}(\xi_k,\partial\Omega)\to+\infty,
		\]
		one has
		\[
		\lambda_k^{-\frac{N-2}{2}}w_k(\xi_k+z/\lambda_k)
		\rightharpoonup U
		\quad\hbox{in }D^{1,2}_{\mathrm{loc}}(\mathbb R^N).
		\]
		Let \(B_k=PU_{\lambda_k,\xi_k}\) and set
		\[
		r_k=w_k-B_k.
		\]
		Then
		\begin{align}
			\int_\Omega\nabla B_k\cdot\nabla r_k\,dx&=o(1),\label{eq3.7}\\
			\int_\Omega a(x)B_kr_k\,dx&=o(1),\label{eq3.8}\\
			\int_\Omega |w_k|^{2^*}\,dx
			&=
			\int_\Omega B_k^{2^*}\,dx
			+
			\int_\Omega |r_k|^{2^*}\,dx
			+o(1),\label{eq3.9}\\
			|w_k|^{2^*-2}w_k
			-
			B_k^{2^*-1}
			-
			|r_k|^{2^*-2}r_k
			&\to0
			\quad\hbox{in }H^{-1}(\Omega).\notag
			\notag\end{align}
		Moreover, the rescaling of \(r_k\) around \((\lambda_k,\xi_k)\) converges weakly to zero in \(D^{1,2}_{\mathrm{loc}}(\mathbb R^N)\).
	\end{lemma}
	
	\begin{proof}
		Let
		\[
		\widehat w_k(z)=
		\lambda_k^{-\frac{N-2}{2}}w_k(\xi_k+z/\lambda_k),
		\qquad
		\widehat B_k(z)=
		\lambda_k^{-\frac{N-2}{2}}B_k(\xi_k+z/\lambda_k).
		\]
		Then
		\[
		\widehat w_k\rightharpoonup U
		\quad\text{in }D^{1,2}_{\mathrm{loc}}(\mathbb R^N),
		\qquad
		\widehat B_k\to U
		\quad\text{strongly in }D^{1,2}_{\mathrm{loc}}(\mathbb R^N)
		\]
		and in \(L^{2^*}_{\mathrm{loc}}\). Therefore
		\[
		\widehat r_k:=\widehat w_k-\widehat B_k
		\rightharpoonup0
		\quad\text{in }D^{1,2}_{\mathrm{loc}}(\mathbb R^N).
		\]
		
		We prove the gradient orthogonality first. Since
		\[
		-\Delta B_k=U_{\lambda_k,\xi_k}^{2^*-1},
		\]
		\[
		\int_\Omega\nabla B_k\cdot\nabla w_k\,dx
		=
		\int_\Omega U_{\lambda_k,\xi_k}^{2^*-1}w_k\,dx.
		\]
		After the change of variables \(z=\lambda_k(x-\xi_k)\), this becomes
		\[
		\int_{\Omega_k}U^{2^*-1}(z)\widehat w_k(z)\,dz.
		\]
		For fixed \(R\), the integral over \(B_R\) converges to
		\[
		\int_{B_R}U^{2^*}\,dz.
		\]
		The tail outside \(B_R\) is uniformly small because \(U^{2^*-1}\in L^{(2^*)'}(\mathbb R^N)\) and \((\widehat w_k)\) is bounded in \(L^{2^*}\). Hence
		\[
		\int_\Omega\nabla B_k\cdot\nabla w_k\,dx
		\to
		\int_{\mathbb R^N}U^{2^*}\,dz
		=
		S^{N/2}.
		\]
		Since
		\[
		\int_\Omega|\nabla B_k|^2\,dx=S^{N/2}+o(1)
		\]
		by Lemma~\ref{Lem3.11}, we obtain
		\[
		\int_\Omega\nabla B_k\cdot\nabla r_k\,dx=o(1).
		\]
		
		Since \(a\in L^\infty(\Omega)\), \(\|B_k\|_2\to0\), and \(r_k\) is bounded in \(L^2(\Omega)\), the potential term is orthogonal:
		\[
		\left|\int_\Omega a(x)B_kr_k\,dx\right|
		\le
		\|a\|_\infty\|B_k\|_2\|r_k\|_2=o(1).
		\]
		
		We prove the \(L^{2^*}\)-splitting. Fix \(R>1\). In the core \(B(\xi_k,R/\lambda_k)\), the rescaled functions satisfy
		\[
		\widehat w_k=\widehat B_k+\widehat r_k,
		\qquad
		\widehat B_k\to U\text{ strongly in }L^{2^*}(B_R),
		\qquad
		\widehat r_k\rightharpoonup0.
		\]
		After passing to a subsequence, local Rellich compactness gives \(\widehat r_k\to0\) almost everywhere on \(B_R\). After a further subsequence, \(\widehat B_k\to U\) almost everywhere, and hence \(\widehat w_k\to U\) almost everywhere on \(B_R\).  The scalar Brezis-Lieb lemma \cite{BrezisLieb1983}, applied with the fixed limit \(U\), gives
		\[
		\int_{B_R}\bigl(
		|\widehat w_k|^{2^*}
		-|\widehat w_k-U|^{2^*}
		-|U|^{2^*}\bigr)\,dz\to0.
		\]
		Since \(\widehat B_k-U\to0\) in \(L^{2^*}(B_R)\) and \(\widehat r_k\) is bounded there, the inequality
		\[
		\big||a+b|^{2^*}-|a|^{2^*}\big|
		\le C\bigl(|a|^{2^*-1}|b|+|b|^{2^*}\bigr)
		\]
		shows that \(|\widehat w_k-U|^{2^*}\) may be replaced by \(|\widehat r_k|^{2^*}\), while strong convergence similarly replaces \(|U|^{2^*}\) by \(|\widehat B_k|^{2^*}\).  Rescaling therefore gives
		\[
		\int_{B(\xi_k,R/\lambda_k)}
		\left(
		|w_k|^{2^*}
		-
		B_k^{2^*}
		-
		|r_k|^{2^*}
		\right)\,dx
		\to0
		\quad(k\to\infty).
		\]
		On the complement of \(B(\xi_k,R/\lambda_k)\), use
		\[
		\big||\alpha+\beta|^{2^*}-|\beta|^{2^*}\big|
		\le
		C\left(|\alpha|^{2^*}+|\alpha||\beta|^{2^*-1}\right)
		\]
		with \(\alpha=B_k\), \(\beta=r_k\). The \(L^{2^*}\)-norm of \(B_k\) outside the core is bounded by the tail of \(U\), which tends to zero as \(R\to\infty\), uniformly in \(k\). Hence, after first letting \(k\to\infty\) and then \(R\to\infty\), we obtain
		\[
		\int_\Omega |w_k|^{2^*}\,dx
		=
		\int_\Omega B_k^{2^*}\,dx
		+
		\int_\Omega |r_k|^{2^*}\,dx
		+o(1).
		\]
		
		The derivative splitting is proved by the same core-tail argument.  On the core, first replace \(\widehat B_k\) by \(U\) using
		\[
		\bigl\||f|^{2^*-2}f-|g|^{2^*-2}g\bigr\|_{(2^*)'}
		\le C\bigl(\|f\|_{2^*}+\|g\|_{2^*}\bigr)^{2^*-2}
		\|f-g\|_{2^*},
		\]
		and then apply the explicit \(\eta\)-decoupling estimate from Lemma~\ref{Lem3.3} to the fixed function \(U\) and the bounded sequence \(\widehat r_k\).  This gives convergence to zero in \(L^{(2^*)'}(B_R)\).  On the exterior region, use
		\[
		\big||\alpha+\beta|^{2^*-2}(\alpha+\beta)
		-
		|\alpha|^{2^*-2}\alpha
		-
		|\beta|^{2^*-2}\beta\big|
		\le
		C\left(
		|\alpha|^{2^*-1}
		+
		|\alpha||\beta|^{2^*-2}
		\right),
		\]
		again with \(\alpha=B_k\), \(\beta=r_k\), and control the exterior contribution by the \(L^{2^*}\)-tail of \(B_k\). Therefore
		\[
		|w_k|^{2^*-2}w_k
		-
		B_k^{2^*-1}
		-
		|r_k|^{2^*-2}r_k
		\to0
		\quad\text{in }L^{(2^*)'}(\Omega),
		\]
		and hence in \(H^{-1}(\Omega)\).
	\end{proof}
	
	\begin{lemma}\label{Lem3.13}
		Let \((w_k)\) be bounded in \(H^1_0(\Omega)\), \(w_k\rightharpoonup0\), and \(J'(w_k)\to0\) in \(H^{-1}(\Omega)\). Suppose that a positive Talenti profile \(U\) is extracted at parameters \((\lambda_k,\xi_k)\), with
		\[
		\lambda_k\operatorname{dist}(\xi_k,\partial\Omega)\to+\infty,
		\]
		and let
		\[
		B_k=PU_{\lambda_k,\xi_k},
		\qquad
		r_k=w_k-B_k.
		\]
		Then
		\begin{align}
			J(w_k)&=I_\infty+J(r_k)+o(1),\label{eq3.10}\\
			\|w_k\|_a^2&=S^{N/2}+\|r_k\|_a^2+o(1),\label{eq3.11}\\
			J'(r_k)&\to0
			\quad\hbox{in }H^{-1}(\Omega).\notag
			\notag\end{align}
		Moreover, after rescaling about the extracted bubble, the remainder satisfies
		\[
		\widehat r_k\rightharpoonup0
		\quad\hbox{in }D^{1,2}_{\mathrm{loc}}(\mathbb R^N).
		\]
	\end{lemma}
	
	\begin{proof}
		The norm decomposition \eqref{eq3.11} follows from Lemma~\ref{Lem3.11} and \eqref{eq3.7}-\eqref{eq3.8}. The nonlinear splitting \eqref{eq3.9}, together with
		\[
		J(B_k)=I_\infty+o(1),
		\]
		gives the energy splitting \eqref{eq3.10}.
		
		For the derivative, write
		\[
		J'(r_k)=J'(w_k)-J'(B_k)+\mathcal R_k,
		\]
		where
		\[
		\langle\mathcal R_k,\psi\rangle
		=
		\int_\Omega
		\Big(
		|w_k|^{2^*-2}w_k
		-
		B_k^{2^*-1}
		-
		|r_k|^{2^*-2}r_k
		\Big)\psi\,dx.
		\]
		By Lemma~\ref{Lem3.12}, \(\mathcal R_k\to0\) in \(H^{-1}\). Since \(J'(w_k)\to0\) by assumption and \(J'(B_k)\to0\) by Lemma~\ref{Lem3.11}, we obtain \(J'(r_k)\to0\). The final weak convergence statement is also part of Lemma~\ref{Lem3.12}.
	\end{proof}
	
	Let \((v_k)\) be a nonnegative \((PS)_c\) sequence for \(J\), let \(v\) be its weak limit, and define
	\[
	B_{i,k}=PU_{\lambda_{i,k},\xi_{i,k}},
	\qquad 1\le i\le j,
	\qquad
	r_{j,k}=v_k-v-\sum_{i=1}^jB_{i,k}.
	\]
	
	\begin{lemma}\label{Lem3.14}
		Suppose that the bubbles \(B_{i,k}\) have been extracted so that
		\[
		\lambda_{i,k}\operatorname{dist}(\xi_{i,k},\partial\Omega)\to+\infty
		\]
		and, for \(i\ne h\),
		\begin{equation}\label{eq3.12}
			\frac{\lambda_{i,k}}{\lambda_{h,k}}
			+
			\frac{\lambda_{h,k}}{\lambda_{i,k}}
			+
			\lambda_{i,k}\lambda_{h,k}|\xi_{i,k}-\xi_{h,k}|^2
			\to+\infty.
		\end{equation}
		Assume also that \(J'(r_{j,k})\to0\), \(r_{j,k}\rightharpoonup0\), and, around each previously extracted bubble, the rescaling of \(r_{j,k}\) converges weakly to zero in \(D^{1,2}_{\mathrm{loc}}(\mathbb R^N)\). If a further nontrivial profile is extracted from \(r_{j,k}\), then its parameters are orthogonal to all previous bubbles in the sense of \eqref{eq3.12}, and the new profile is nonnegative. Consequently it is a positive Talenti bubble.
	\end{lemma}
	
	\begin{proof}
		Let the new profile be extracted at \((\lambda_k,\xi_k)\). Suppose first that the new parameters are not orthogonal to some old bubble, say \(B_{i,k}\). Then after passing to a subsequence,
		\[
		\frac{\lambda_k}{\lambda_{i,k}}
		+
		\frac{\lambda_{i,k}}{\lambda_k}
		+
		\lambda_k\lambda_{i,k}|\xi_k-\xi_{i,k}|^2
		\]
		remains bounded. Therefore the rescaling around \((\lambda_k,\xi_k)\) differs from the rescaling around \((\lambda_{i,k},\xi_{i,k})\) only by a bounded dilation and translation. A nonzero weak limit in the new variables would then produce a nonzero weak local limit of the rescaled remainder around the old bubble.  This contradicts the induction hypothesis that the rescaling of \(r_{j,k}\) around every old bubble converges weakly to zero. Hence the new parameters are orthogonal to all previous ones.
		
		Now rescale the identity
		\[
		r_{j,k}=v_k-v-\sum_{i=1}^jB_{i,k}
		\]
		around the new parameters. Since \(\lambda_k\to\infty\),
		\[
		\lambda_k^{-\frac{N-2}{2}}
		v(\xi_k+z/\lambda_k)
		\to0
		\quad\hbox{strongly in }
		D^{1,2}_{\mathrm{loc}}(\mathbb R^N)
		\cap L^{2^*}_{\mathrm{loc}}(\mathbb R^N),
		\]
		by absolute continuity of the integrals of \(|\nabla v|^2\) and \(|v|^{2^*}\).  Lemma~\ref{Lem3.8} shows that every old projected bubble converges weakly, though not necessarily strongly, to zero in the new variables.  Hence the weak local limit of the rescaled remainder equals the weak local limit of the rescaled nonnegative functions \(v_k\). The nonnegative cone is weakly closed in \(L^{2^*}_{\mathrm{loc}}(\mathbb R^N)\), so the new profile is nonnegative. By Lemma~\ref{Lem3.9} and the Caffarelli-Gidas-Spruck classification theorem \cite{CGS1989}, it is a positive Talenti bubble.
	\end{proof}
	
	\subsection{Iteration, separation, quantization, and quotient deformation}
	
	\begin{lemma}\label{Lem3.15}
		Any two distinct bubbles produced by the extraction procedure satisfy the orthogonality condition \eqref{eq3.12}. Moreover each bubble satisfies
		\[
		\lambda_{i,k}\operatorname{dist}(\xi_{i,k},\partial\Omega)\to+\infty.
		\]
	\end{lemma}
	
	\begin{proof}
		The boundary condition follows from Proposition~\ref{Prop3.7}. The orthogonality is precisely the first conclusion of Lemma~\ref{Lem3.14}, applied inductively at the moment when the later of the two bubbles is extracted.
	\end{proof}
	
	\begin{lemma}\label{Lem3.16}
		The extraction process stops after finitely many steps. If \(m\) bubbles are extracted from an original nonnegative Palais-Smale sequence, then the final remainder \(r_{m,k}\) satisfies
		\[
		r_{m,k}\to0
		\quad\hbox{strongly in }H^1_0(\Omega).
		\]
	\end{lemma}
	
	\begin{proof}
		After \(j\) extractions, repeated use of Lemma~\ref{Lem3.13} gives
		\[
		\|v_k-v\|_a^2
		=
		jS^{N/2}
		+
		\|r_{j,k}\|_a^2
		+o(1).
		\]
		Since \((v_k)\) is bounded in \(H^1_0(\Omega)\), the integer \(j\) is bounded. Hence the process stops after finitely many steps.
		
		Let \(r_{m,k}\) be the final remainder. It is still a Palais-Smale sequence for \(J\), weakly convergent to zero. If \(r_{m,k}\not\to0\) in \(L^{2^*}(\Omega)\), then Lemma~\ref{Lem3.6}, Proposition~\ref{Prop3.7}, Lemma~\ref{Lem3.9}, and Lemma~\ref{Lem3.14} would produce another positive Talenti bubble, contradicting maximality. Hence
		\[
		r_{m,k}\to0
		\quad\hbox{in }L^{2^*}(\Omega).
		\]
		Since \(J'(r_{m,k})\to0\),
		\[
		\|r_{m,k}\|_a^2
		=
		\int_\Omega |r_{m,k}|^{2^*}\,dx+o(1)\|r_{m,k}\|_a=o(1),
		\]
		and therefore \(r_{m,k}\to0\) in \(H^1_0(\Omega)\).
	\end{proof}
	
	\begin{lemma}\label{Lem3.17}
		Let \(m\) be the number of extracted bubbles. Then
		\[
		J(v_k)=J(v)+mI_\infty+o(1),
		\qquad
		\|v_k\|_a^2=\|v\|_a^2+mS^{N/2}+o(1).
		\]
	\end{lemma}
	
	\begin{proof}
		Apply Lemma~\ref{Lem3.3} to \(v_k=v+(v_k-v)\). Then apply Lemma~\ref{Lem3.13} successively to the remainders. Each extracted positive Talenti bubble contributes \(I_\infty\) to the energy and \(S^{N/2}\) to the squared \(a\)-norm. The final remainder is \(o(1)\) in \(H^1_0(\Omega)\) by Lemma~\ref{Lem3.16}.  This gives both identities.
	\end{proof}
	
	\begin{corollary}
		If \((v_k)\) is a nonnegative \((PS)_c\) sequence with \(c<I_\infty\), then \(v_k\to v\) strongly in \(H^1_0(\Omega)\). More generally, if
		\[
		c\in[\ell I_\infty,(\ell+1)I_\infty)
		\]
		for some \(\ell\in\mathbb N\), then at most \(\ell\) bubbles can occur.
	\end{corollary}
	
	\begin{proof}
		By Lemma~\ref{Lem3.17},
		\[
		c=J(v)+mI_\infty.
		\]
		Since \(v\) is a critical point,
		\[
		J(v)=\frac1N\|v\|_a^2\ge0.
		\]
		The assertions follow.
	\end{proof}
	
	\begin{theorem}\label{Thm3.19}
		Let \((u_k)\) be a nonnegative \((PS)_c\) sequence for \(I\), and set
		\[
		v_k=Tu_k.
		\]
		Then, after passing to a subsequence, there exist a nonnegative weak solution \(v\) of \eqref{eq3.1}, an integer \(m\ge0\), points \(\xi_{i,k}\in\Omega\), and scales \(\lambda_{i,k}\to+\infty\), such that
		\[
		\lambda_{i,k}\operatorname{dist}(\xi_{i,k},\partial\Omega)\to+\infty,
		\]
		\[
		\frac{\lambda_{i,k}}{\lambda_{h,k}}
		+
		\frac{\lambda_{h,k}}{\lambda_{i,k}}
		+
		\lambda_{i,k}\lambda_{h,k}|\xi_{i,k}-\xi_{h,k}|^2
		\to+\infty
		\qquad (i\ne h),
		\]
		and
		\begin{equation}\label{eq3.13}
			v_k
			=
			v+\sum_{i=1}^m PU_{\lambda_{i,k},\xi_{i,k}}+o(1)
			\quad\hbox{strongly in }H^1_0(\Omega).
		\end{equation}
		Moreover,
		\[
		J(v_k)=J(v)+mI_\infty+o(1),
		\qquad
		\|v_k\|_a^2=\|v\|_a^2+mS^{N/2}+o(1).
		\]
		Equivalently, in the original hyperbolic variables,
		\[
		u_k
		=
		u+\sum_{i=1}^m\phi^{-1}PU_{\lambda_{i,k},\xi_{i,k}}+o(1)
		\quad\hbox{strongly in }H^1_0(\Omega;g_{\Hh}),
		\qquad
		u=\phi^{-1}v,
		\]
		and
		\[
		I(u_k)=I(u)+mI_\infty+o(1).
		\]
	\end{theorem}
	
	\begin{proof}
		By Lemma~\ref{Lem3.2}, after passing to a subsequence,
		\[
		v_k\rightharpoonup v
		\quad\hbox{in }H^1_0(\Omega),
		\]
		where \(v\) solves \eqref{eq3.1}. Since \(v_k\ge0\), also \(v\ge0\). Set
		\[
		r_{0,k}=v_k-v.
		\]
		By Lemma~\ref{Lem3.3},
		\[
		J'(r_{0,k})\to0,
		\qquad
		J(v_k)=J(v)+J(r_{0,k})+o(1).
		\]
		If \(r_{0,k}\to0\) strongly in \(H^1_0(\Omega)\), the theorem holds with \(m=0\).
		
		Otherwise Lemma~\ref{Lem3.6}, Proposition~\ref{Prop3.7}, and Lemma~\ref{Lem3.9} give a first nontrivial interior profile. Since \(v_k\ge0\) and \(v\) disappears in the blow-up scale, this profile is nonnegative.  After translating and dilating the parameters, it is the normalized Talenti bubble \(U\).  Define
		\[
		B_{1,k}=PU_{\lambda_{1,k},\xi_{1,k}},
		\qquad
		r_{1,k}=r_{0,k}-B_{1,k}.
		\]
		Lemma~\ref{Lem3.13} gives
		\[
		J'(r_{1,k})\to0,
		\qquad
		J(r_{0,k})=I_\infty+J(r_{1,k})+o(1),
		\]
		and the rescaling of \(r_{1,k}\) around the first bubble converges weakly to zero.
		
		Assume inductively that \(j\) positive bubbles have been extracted and that
		\[
		r_{j,k}
		=
		v_k-v-\sum_{i=1}^jB_{i,k}
		\]
		satisfies
		\[
		J'(r_{j,k})\to0,\qquad r_{j,k}\rightharpoonup0,
		\]
		and has zero weak profile around each previously extracted bubble. If \(r_{j,k}\to0\) strongly in \(H^1_0(\Omega)\), the extraction stops. Otherwise, Lemma~\ref{Lem3.6} and Proposition~\ref{Prop3.7} produce a new interior profile. Lemma~\ref{Lem3.14} shows that the new parameters are orthogonal to all previous ones and that the new profile is nonnegative. Hence, by Lemma~\ref{Lem3.9}, it is a positive Talenti bubble. After normalizing the parameters, set
		\[
		B_{j+1,k}=PU_{\lambda_{j+1,k},\xi_{j+1,k}},
		\qquad
		r_{j+1,k}=r_{j,k}-B_{j+1,k}.
		\]
		Lemma~\ref{Lem3.13} shows that the new remainder is again a Palais-Smale sequence, has zero weak profile around the new bubble, and lowers the energy and squared norm by \(I_\infty\) and \(S^{N/2}\), respectively.
		
		Each extraction lowers the energy by one copy of the Talenti energy.  More precisely, after the \(j\)-th extraction Lemma~\ref{Lem3.13} gives
		\[
		J(r_{j,k})=J(r_{0,k})-jI_\infty+o(1).
		\]
		Since the level of the original Palais-Smale sequence is finite and \(J(r_{j,k})\ge -o(1)\) for a residual Palais-Smale sequence at the critical exponent, only finitely many nontrivial profiles can be extracted. Let \(m\) be the maximal number of extracted profiles.  If the final remainder \(r_{m,k}\) did not converge strongly to zero, the concentration lemma would produce a further nontrivial interior profile, contradicting this maximality.  Hence
		\[
		r_{m,k}\to0
		\quad\hbox{strongly in }H^1_0(\Omega).
		\]
		This gives \eqref{eq3.13}. Lemma~\ref{Lem3.15} gives the orthogonality and boundary separation conditions, and Lemma~\ref{Lem3.17} gives the energy and norm identities. Finally, since \(T\) is an isomorphism and \(I(u)=J(Tu)\), applying \(T^{-1}\) gives the hyperbolic-variable decomposition.
	\end{proof}

	\subsection{Deformation of quotient sublevels to bubble tubes}

	We pass from the energy functional \(J\) to the quotient functional used in the barycenter argument.  Set
	\begin{equation*}
		\Sigma_{a,+}
		=
		\{v\in H^1_0(\Omega): v\ge 0\ \text{almost everywhere in }\Omega,\ \|v\|_a=1\}
	\end{equation*}
	and
	\begin{equation*}
		\mathcal J_a(v)=
		\left(\int_\Omega v^{2^*}\,dx\right)^{-1},
		\qquad v\in\Sigma_{a,+}.
	\end{equation*}
	We realize this quotient as the restriction to the positive cone of the \(C^1\) (in fact \(C^2\); see Lemma~\ref{Lem3.26}) functional
	\[
	\mathcal J_{a,+}(v)=
	\left(\int_\Omega (v^+)^{2^*}\,dx\right)^{-1},
	\qquad
	v\in\Sigma_a,
	\quad
	v^+\not\equiv0,
	\]
	where \(\Sigma_a=\{v\in H^1_0(\Omega):\|v\|_a=1\}\).  Since \(2^*>2\), the map \(v\mapsto\int_\Omega(v^+)^{2^*}\,dx\) is \(C^1\) on \(H^1_0(\Omega)\), with
	\[
	D\left(\int_\Omega(v^+)^{2^*}\,dx\right)[\psi]
	=2^*\int_\Omega(v^+)^{2^*-1}\psi\,dx.
	\]
	We use \(\mathcal J_{a,+}\) only on the open subset \(\{v\in\Sigma_a:v^+\not\equiv0\}\) of the smooth Hilbert sphere. If \((v_k)\subset\Sigma_a\) is a quotient Palais-Smale sequence for \(\mathcal J_{a,+}\), then its negative part tends to zero.  Writing
	\[
	F(v)=\int_\Omega (v^+)^{2^*}\,dx,
	\]
	the constrained Palais-Smale condition gives real multipliers \(\mu_k\) such that
	\[
	-2^*F(v_k)^{-2}(v_k^+)^{2^*-1}-\mu_k L_av_k=o(1)
	\quad\hbox{in }H^{-1}(\Omega).
	\]
	Testing with \(v_k\) yields
	\[
	-2^*F(v_k)^{-1}-\mu_k=o(1),
	\]
	and hence \(|\mu_k|\ge c>0\), since \(F(v_k)\) stays bounded away from zero and infinity on every bounded quotient level.  Testing the same residual with \(v_k^-\) gives
	\[
	\mu_k\|v_k^-\|_a^2=o(1)\|v_k^-\|_a,
	\]
	because \((v_k^+)^{2^*-1}v_k^-=0\) almost everywhere.  Thus \(v_k^-\to0\) in \(H^1_0(\Omega)\). After replacing \(v_k\) by
	\[
	\widehat v_k=\frac{v_k^+}{\|v_k^+\|_a},
	\]
	one obtains an equivalent Palais-Smale sequence in \(\Sigma_{a,+}\).  Critical points obtained in this way are nonnegative, and the strong maximum principle makes them positive.
	
	If \(v\in\Sigma_{a,+}\) is a critical point of \(\mathcal J_a\), then there is a Lagrange multiplier \(\ell=\mathcal J_a(v)\) such that
	\[
	L_av=\ell v^{2^*-1}.
	\]
	Consequently
	\begin{equation*}
		w=\ell^{\frac{N-2}{4}}v
	\end{equation*}
	solves
	\[
	L_aw=w^{2^*-1}
	\quad\hbox{in }\Omega,
	\qquad w=0\quad\hbox{on }\partial\Omega.
	\]
	Conversely, every positive solution of this equation gives a critical point of \(\mathcal J_a\) after normalization in \(\|\cdot\|_a\).
	
	Let
	\begin{equation*}
		\sigma_N=
		\left(\int_{\mathbb R^N}\delta_{1,0}^{2^*}\,dx\right)^{-1}
		=S^{\frac{N}{N-2}},
		\qquad
		b_n=\sigma_N n^{\frac{2}{N-2}}.
	\end{equation*}
	We also set \(b_0=0\). The number \(b_n\) is the value of \(\mathcal J_a\) on a normalized sum of \(n\) mutually orthogonal normalized bubbles with equal weights.  We choose regular levels \(c_j\) satisfying
	\begin{equation*}
		b_j<c_j<b_{j+1},
	\end{equation*}
	and define
	\[
	W_{a,j}=\{v\in\Sigma_{a,+}:\mathcal J_a(v)\le c_j\}.
	\]
	We also set \(W_{a,0}=\varnothing\).  Equivalently, choose a level \(c_0<b_1\) for which the Sobolev inequality gives \(\{\mathcal J_a\le c_0\}=\varnothing\).  With the normalization used here,
	\begin{equation}\label{eq3.15}
		b_m^{\frac{N-2}{2}}=mS^{N/2},
		\qquad
		\frac1N b_m^{\frac{N-2}{2}}=mI_\infty.
	\end{equation}
	Here \(\sigma_N=S^{N/(N-2)}\) and \(I_\infty=N^{-1}S^{N/2}\). For each fixed \(n\), the level \(c_{n-1}\in(b_{n-1},b_n)\) will be chosen after the fixed-\(n\) estimates have been established.  In particular, at the single multiplicity \(m\) used in the proof, Section~\ref{sec:matched} gives a strict inequality
	\[
	\sup_{\mu\in B_m(K)}
	\mathcal J_a(g_{a,\lambda,m}(\mu))<b_m,
	\]
	and \(c_{m-1}\) is then chosen strictly between this supremum (and \(b_{m-1}\)) and \(b_m\).
	\begin{definition}\label{Def3.20}
		For \(n\ge1\) and \(\varepsilon>0\), let \(\mathcal V_{a,\le n}(\varepsilon)\subset\Sigma_{a,+}\) consist of the functions \(v\) for which there are an integer \(1\le m\le n\), numbers \(\alpha_i>0\), points \(x_i\in\Omega\), and scales \(\lambda_i>\varepsilon^{-1}\) such that
		\[
		\left\|
		v-
		\frac{\sum_{i=1}^m\alpha_iP_a\delta_{\lambda_i,x_i}}
		{\left\|\sum_{i=1}^m\alpha_iP_a\delta_{\lambda_i,x_i}\right\|_a}
		\right\|_a<\varepsilon,
		\]
		\[
		\lambda_i\operatorname{dist}(x_i,\partial\Omega)>\varepsilon^{-1}
		\quad(1\le i\le m),
		\]
		and
		\[
		\varepsilon_{ij}
		=
		\left(
		\frac{\lambda_i}{\lambda_j}
		+
		\frac{\lambda_j}{\lambda_i}
		+
		\lambda_i\lambda_j|x_i-x_j|^2
		\right)^{-\frac{N-2}{2}}
		<\varepsilon
		\quad(i\ne j).
		\]
		Parameter tuples are identified up to permutation, and the coefficients range independently over \((0,\infty)^m\).
	\end{definition}
	
	For \(n\ge1\), \(0<\rho<1\), \(L>1\), and
	\[
	0<\zeta<2^{-\frac{N-2}{2}},
	\]
	let \(\mathscr Q_n(\rho,L,\zeta)\) be the labelled parameter manifold of tuples
	\[
	q=(\alpha_1,\ldots,\alpha_n;
	x_1,\ldots,x_n;\lambda_1,\ldots,\lambda_n)
	\]
	satisfying
	\[
	\alpha_i>0,\qquad
	\sum_i\alpha_i=1,\qquad
	|n\alpha_i-1|<\rho,
	\]
	\[
	x_i\in\Omega,\qquad
	\lambda_i>L,\qquad
	\lambda_i\operatorname{dist}(x_i,\partial\Omega)>L,
	\]
	and
	\[
	\left(
	\frac{\lambda_i}{\lambda_k}
	+
	\frac{\lambda_k}{\lambda_i}
	+
	\lambda_i\lambda_k|x_i-x_k|^2
	\right)^{-\frac{N-2}{2}}
	<\zeta
	\qquad(i\ne k).
	\]
	The notation \(\overline{\mathscr Q}_n(\rho,L,\zeta)\) means that the strict weight-balance, lower-scale, scale-boundary, and interaction inequalities above are replaced by their weak counterparts; positivity of the coefficients remains automatic because \(\rho<1\).  This closure is taken within the finite-scale parameter space.  The normalization \(\sum_i\alpha_i=1\) removes the common coefficient scaling.  The symmetric group \(\mathfrak S_n\) acts by simultaneously permuting \((\alpha_i,x_i,\lambda_i)\).  The action is free, because the interaction cutoff excludes two identical bubble parameters. We use the quotient \(\mathscr Q_n(\rho,L,\zeta)/\mathfrak S_n\), rather than imposing a discontinuous global ordering.  In a tower with the same physical centre we label the factors by increasing scale; in every local quotient chart the labels are fixed by the unique matching with its base tuple.  Two such charts therefore differ on an overlap by one constant permutation.
	
	For the one-parameter notation used in this section, set
	\[
	\mathcal P_n^{<}(\varepsilon)
	=
	\mathscr Q_n(\varepsilon,\varepsilon^{-1},\varepsilon)
	/\mathfrak S_n,
	\qquad
	\mathcal P_n^{\le}(\varepsilon)
	=
	\overline{\mathscr Q}_n(\varepsilon,\varepsilon^{-1},\varepsilon)
	/\mathfrak S_n,
	\]
	where
	\[
	0<\varepsilon<
	\min\{1,2^{-\frac{N-2}{2}}\}.
	\]
	Thus every weight-balance, lower-scale, scale-boundary, and interaction inequality is strict in \(\mathcal P_n^{<}(\varepsilon)\) and weak in \(\mathcal P_n^{\le}(\varepsilon)\).  In both charts \(x_i\in\Omega\), the coefficients are positive and sum to one, and the tuples are taken modulo the simultaneous \(\mathfrak S_n\)-action described above.  We use the closed-parameter convention by default:
	\[
	\mathcal P_n(\varepsilon)=\mathcal P_n^{\le}(\varepsilon).
	\]
	Define
	\[
	\Phi_{a,n}(q)
	=
	\frac{\sum_{i=1}^n\alpha_iP_a\delta_{\lambda_i,x_i}}
	{\left\|\sum_{i=1}^n\alpha_iP_a\delta_{\lambda_i,x_i}\right\|_a}.
	\]
	
	\begin{definition}\label{Def3.21}
		\[
		\mathcal T_{a,n}^{\square}(\varepsilon)
		=
		\left\{
		v\in\Sigma_{a,+}:
		\inf_{q\in\mathcal P_n^{\square}(\varepsilon)}
		\|v-\Phi_{a,n}(q)\|_a<\varepsilon
		\right\},
		\qquad \square\in\{<,\le\}.
		\]
		We use the closed-parameter convention \(\mathcal T_{a,n}(\varepsilon)= \mathcal T_{a,n}^{\le}(\varepsilon)\) unless a superscript is displayed.
	\end{definition}
	
	Repeated centers are allowed when their scales are sufficiently separated; all amplitudes in this chart are positive. Both tubes are open in \(\Sigma_{a,+}\), because their residual-distance inequality is strict.  If \(0<\delta<\varepsilon\), then
	\begin{equation}\label{eq3.strict-closed-sandwich}
		\mathcal T_{a,n}^{<}(\delta)
		\subset
		\mathcal T_{a,n}^{\le}(\delta)
		\subset
		\mathcal T_{a,n}^{<}(\varepsilon)
		\subset
		\mathcal T_{a,n}^{\le}(\varepsilon).
	\end{equation}
	Indeed, the corresponding parameter sets have the same inclusions, and the residual bound \(<\delta\) implies the residual bound \(<\varepsilon\).
	
	\begin{remark}\label{Rem3.22}
		The distinction between \(\mathcal V_{a,\le n}(\varepsilon)\) and \(\mathcal T_{a,n}(\varepsilon)\) is essential.  A bounded quotient Palais-Smale sequence below \(b_{n+1}\) may split into any number \(m\le n\) of bubbles, which is why the compactness tube must contain all lower multiplicities and arbitrary preliminary coefficients.  On the strip
		\[
		c_{n-1}<\mathcal J_a<c_n,
		\qquad
		b_{n-1}<c_{n-1}<b_n<c_n<b_{n+1},
		\]
		quantization forces \(m=n\), and normalization forces the coefficients to be asymptotically equal.  Only the balanced tube describes this relative layer.
	\end{remark}
	
	\begin{lemma}\label{Lem3.23}
		Let \((v_k)\subset\Sigma_{a,+}\) satisfy
		\[
		\mathcal J_a(v_k)\to\ell>0,
		\qquad
		\|\nabla_{\Sigma_a}\mathcal J_a(v_k)\|_a\to0.
		\]
		Set
		\[
		w_k=\mathcal J_a(v_k)^{\frac{N-2}{4}}v_k.
		\]
		Then \((w_k)\) is a Palais-Smale sequence for \(J\), and, for every \(k\),
		\[
		J(w_k)
		=
		\frac1N\mathcal J_a(v_k)^{\frac{N-2}{2}}.
		\]
	\end{lemma}
	
	\begin{proof}
		The constrained Palais-Smale condition gives numbers \(\theta_k\) such that
		\[
		L_av_k-\theta_kv_k^{2^*-1}\to0
		\quad\hbox{in }H^{-1}(\Omega).
		\]
		Testing with \(v_k\), and using \(\|v_k\|_a=1\) and \(\int_\Omega v_k^{2^*}=\mathcal J_a(v_k)^{-1}\), gives
		\[
		\theta_k=\mathcal J_a(v_k)+o(1).
		\]
		Multiplication by \(\mathcal J_a(v_k)^{(N-2)/4}\) now yields
		\[
		L_aw_k-w_k^{2^*-1}\to0
		\quad\hbox{in }H^{-1}(\Omega).
		\]
		Moreover,
		\[
		\|w_k\|_a^2
		=
		\int_\Omega w_k^{2^*}\,dx
		=
		\mathcal J_a(v_k)^{\frac{N-2}{2}},
		\]
		which proves the exact energy identity.
	\end{proof}
	
	\begin{proposition}\label{Prop3.24}
		Assume that
		\[
		L_av=v^{2^*-1},\qquad
		v>0\ \hbox{in }\Omega,\qquad
		v=0\ \hbox{on }\partial\Omega
		\]
		has no solution.  Let \((v_k)\subset\Sigma_a\) satisfy
		\[
		v_k^+\not\equiv0,\qquad
		\mathcal J_{a,+}(v_k)\to\ell>0,\qquad
		\|\nabla_{\Sigma_a}\mathcal J_{a,+}(v_k)\|_a\to0.
		\]
		Then \(v_k^-\to0\) in \(H^1_0(\Omega)\).  After replacing \(v_k\) by
		\[
		\widehat v_k=\frac{v_k^+}{\|v_k^+\|_a},
		\]
		there is an integer \(m\ge1\) such that
		\[
		\ell=b_m
		\]
		and, for every \(\varepsilon>0\),
		\[
		\widehat v_k\in\mathcal V_{a,\le m}(\varepsilon)
		\]
		for all sufficiently large \(k\).  In particular, if \(\ell\le c_n<b_{n+1}\), then \(m\le n\).
	\end{proposition}
	
	\begin{proof}
		The negative-part test in the constrained Euler residual gives
		\[
		v_k^-\to0\quad\hbox{in }H^1_0(\Omega).
		\]
		Thus the normalized positive parts form an equivalent quotient Palais-Smale sequence, which we again denote by \(v_k\). Lemma~\ref{Lem3.23} gives an energy Palais-Smale sequence
		\[
		w_k=\mathcal J_a(v_k)^{\frac{N-2}{4}}v_k
		\]
		at level
		\[
		\frac1N\ell^{\frac{N-2}{2}}.
		\]
		Theorem~\ref{Thm3.19} and the nonexistence assumption imply
		\[
		w_k
		=
		\sum_{i=1}^mPU_{\lambda_{i,k},\xi_{i,k}}+o(1)
		\quad\hbox{in }H^1_0(\Omega),
		\]
		where the parameters are mutually orthogonal and scale-interior.  Energy quantization gives
		\[
		\frac1N\ell^{\frac{N-2}{2}}=mI_\infty,
		\]
		and hence \(\ell=b_m\).
		
		After division by
		\[
		b_m^{\frac{N-2}{4}}
		=
		\sqrt m\,S^{N/4},
		\]
		and after passing from \(U\) to the normalization \(\|\nabla\delta\|_2=1\), one obtains
		\begin{equation}\label{eq3.equal-coefficient-decomposition}
			v_k
			=
			\frac1{\sqrt m}
			\sum_{i=1}^mP\delta_{\lambda_{i,k},\xi_{i,k}}
			+o(1)
			\quad\hbox{in }H^1_0(\Omega).
		\end{equation}
		The replacement of \(P\delta\) by \(P_a\delta\) is uniform for every scale-interior sequence.  Indeed, if
		\[
		R_{\lambda,\xi}
		=
		P_a\delta_{\lambda,\xi}-P\delta_{\lambda,\xi},
		\]
		then
		\[
		L_aR_{\lambda,\xi}=-aP\delta_{\lambda,\xi}
		\]
		and coercivity gives
		\[
		\|R_{\lambda,\xi}\|_a
		\le
		C\|aP\delta_{\lambda,\xi}\|_{H^{-1}(\Omega)}
		=o(1)
		\]
		whenever \(\lambda\to\infty\) and \(\lambda\operatorname{dist}(\xi,\partial\Omega)\to\infty\). This follows from \(0\le P\delta_{\lambda,\xi}\le\delta_{\lambda,\xi}\), the boundedness of \(a\), and the vanishing of the subcritical norm \(\|P\delta_{\lambda,\xi}\|_{L^{(2^*)'}(\Omega)}\). Thus \eqref{eq3.equal-coefficient-decomposition} puts \(v_k\) in the broad tube for every fixed width.  The last assertion follows from \eqref{eq3.15}.
	\end{proof}
	
	\begin{corollary}
		\label{Cor3.25}
		Assume the nonexistence hypothesis of Proposition~\ref{Prop3.24}.  If \((v_k)\subset\Sigma_a\) satisfies
		\[
		c_{n-1}<\mathcal J_{a,+}(v_k)\le c_n,
		\qquad
		\|\nabla_{\Sigma_a}\mathcal J_{a,+}(v_k)\|_a\to0,
		\]
		then, after normalization of the positive part,
		\[
		\mathcal J_a(\widehat v_k)\to b_n
		\]
		and, for every \(\varepsilon>0\),
		\[
		\widehat v_k\in\mathcal T_{a,n}(\varepsilon)
		\]
		for all sufficiently large \(k\).
	\end{corollary}
	
	\begin{proof}
		Proposition~\ref{Prop3.24} gives a quantized limit \(b_m\).  The choice
		\[
		b_{n-1}<c_{n-1}<b_n<c_n<b_{n+1}
		\]
		forces \(m=n\).  Formula \eqref{eq3.equal-coefficient-decomposition} then gives exactly \(n\) mutually orthogonal scale-interior bubbles with asymptotically equal coefficients, which is the defining condition of \(\mathcal T_{a,n}(\varepsilon)\).
	\end{proof}
	
	Put
	\[
	\widetilde\Sigma_a
	=
	\{v\in\Sigma_a:v^+\not\equiv0\},
	\qquad
	f(v)=\mathcal J_{a,+}(v)
	=
	\left(\int_\Omega(v^+)^{2^*}\,dx\right)^{-1}.
	\]
	
	\begin{lemma}
		\label{Lem3.26}
		Then \(f\in C^2(\widetilde\Sigma_a)\).  For each \(c>0\), its constrained Hessian is uniformly bounded on
		\[
		\{v\in\widetilde\Sigma_a:f(v)\le c\}.
		\]
	\end{lemma}
	
	\begin{proof}
		Set \(p=2^*>2\) and
		\[
		G(v)=\int_\Omega(v^+)^p\,dx.
		\]
		The scalar function \(s\mapsto(s^+)^p\) is \(C^2\), and the corresponding Nemytskii functional on \(L^p(\Omega)\) satisfies
		\[
		DG(v)[h]
		=
		p\int_\Omega(v^+)^{p-1}h\,dx,
		\]
		\[
		D^2G(v)[h,k]
		=
		p(p-1)\int_\Omega(v^+)^{p-2}hk\,dx.
		\]
		H\"older's inequality and the Sobolev embedding give, on \(\|v\|_a=1\),
		\[
		\|DG(v)\|\le C,
		\qquad
		\|D^2G(v)\|\le C.
		\]
		On \(f\le c\) one has \(G(v)\ge c^{-1}\), and hence
		\[
		Df=-G^{-2}DG,
		\qquad
		D^2f
		=
		2G^{-3}DG\otimes DG-G^{-2}D^2G
		\]
		are uniformly bounded.  Restriction to the Hilbert sphere adds only the bounded second fundamental form term.  This proves the assertion.
	\end{proof}
	
	\begin{remark}\label{Rem3.27}
		The preceding lemma is needed in the modified-functional argument below: the map
		\[
		v\longmapsto
		\|\nabla_{\Sigma_a}\mathcal J_{a,+}(v)\|_a^2
		\]
		is \(C^1\) by the \(C^2\)-regularity of \(\mathcal J_{a,+}\) established above.  This regularity supplies the Hessian estimate required below, beyond the \(C^1\) regularity used in an ordinary deformation lemma.  The estimate is uniform at bubbling sequences because the unit \(H^1_0\)-norm and the lower bound \(G(v)\ge c_n^{-1}\) control all terms in the Hessian.
	\end{remark}
	
	\begin{lemma}\label{Lem3.28}
		Assume that the \(L_a\)-Dirichlet problem has no positive solution, and fix \(n\ge1\).  Choose
		\[
		b_{n-1}<c_{n-1}<b_n<c_n<b_{n+1},
		\]
		with \(c_{n-1}\) sufficiently close to \(b_n\). For every sufficiently small \(\varepsilon>0\), there are
		\[
		0<\delta<\varepsilon
		\]
		and a closed set \(\mathcal F_{a,n}\subset W_{a,n}\) such that
		\begin{enumerate}
			\item
			\[
			W_{a,n-1}\subset\mathcal F_{a,n};
			\]
			\item
			\[
			(\mathcal F_{a,n},W_{a,n-1})
			\hookrightarrow
			(W_{a,n},W_{a,n-1})
			\]
			is a strong deformation equivalence of pairs;
			\item
			\[
			\mathcal F_{a,n}\setminus W_{a,n-1}
			\subset\mathcal T_{a,n}^{\le}(\delta);
			\]
			\item the widths can be chosen so that
			\begin{equation}\label{eq3.tube-nesting}
				\overline{\mathcal T_{a,n}^{\le}(\delta)}
				\subset\mathcal T_{a,n}^{\le}(\varepsilon)
			\end{equation}
			in the strong topology of \(\Sigma_{a,+}\).
		\end{enumerate}
		Moreover, there are
		\[
		\eta_n,\ s_{0,n},\ s_{1,n},\ \vartheta_{n,*},\
		\mu_{n,*}>0,
		\]
		depending only on \(n,\Omega,\varepsilon\), such that the same deformation data work for every regular pair of levels satisfying
		\[
		b_n-\vartheta_{n,*}<c_{n-1}<b_n<c_n<b_{n+1}
		\]
		and \(0<s_{0,n}<s_{1,n}\).
	\end{lemma}
	
	\begin{proof}
		We first isolate the only quantized level relevant to the \(n\)-th core. Choose
		\[
		0<\eta_n<
		\frac14\min\{b_n-b_{n-1},\,b_{n+1}-b_n\}.
		\]
		For every sufficiently small \(\delta>0\), there is \(\gamma_{\delta,n}>0\), independent of the exact levels, such that
		\begin{equation}\label{eq3.balanced-gradient-bound}
			\|\nabla_{\Sigma_a}\mathcal J_a(v)\|_a
			\ge\gamma_{\delta,n}
			\quad\hbox{if}\quad
			|\,\mathcal J_a(v)-b_n\,|\le2\eta_n,\qquad
			v\notin\mathcal T_{a,n}(\delta).
		\end{equation}
		Otherwise there would be a sequence in the \(n\)-th strip with constrained gradient converging to zero, with quotient in the fixed window \([b_n-2\eta_n,b_n+2\eta_n]\), but remaining outside the balanced tube. Proposition~\ref{Prop3.24} quantizes its limit; the choice of \(\eta_n\) forces that limit to be \(b_n\), and Corollary~\ref{Cor3.25} gives a contradiction.
		
		Write \(f=\mathcal J_{a,+}\) on \(\widetilde\Sigma_a\) and set
		\[
		\mathscr G(v)=\|\nabla_{\Sigma_a}f(v)\|_a^2.
		\]
		By Lemma~\ref{Lem3.26}, there is \(C_n>0\) such that
		\begin{equation}\label{eq3.G-bound}
			|d\mathscr G(v)[\nabla_{\Sigma_a}f(v)]|
			\le2C_n\mathscr G(v)
			\qquad(f(v)\le b_{n+1}).
		\end{equation}
		Choose
		\[
		0<s_{0,n}<s_{1,n}<\gamma_{\delta,n}^2.
		\]
		Let \(\chi_{n,\mathrm{cut}}\in C^1(\mathbb R,[0,1])\) satisfy
		\[
		\chi_{n,\mathrm{cut}}=1
		\quad\hbox{on }[b_n-\eta_n,b_n+\eta_n],
		\qquad
		\chi_{n,\mathrm{cut}}=0\quad\hbox{outside }
		(b_n-2\eta_n,b_n+2\eta_n).
		\]
		Choose, before fixing \(c_{n-1}\) and \(c_n\),
		\begin{equation}\label{eq3.tau-choice}
			0<\vartheta_{n,*}<
			\min\left\{
			\frac{\eta_n}{4},\
			\frac{b_n-b_{n-1}}4,\
			\frac{s_{1,n}-s_{0,n}}{32C_n},\
			\frac1{8(1+\|\chi_{n,\mathrm{cut}}'\|_\infty)}
			\right\}.
		\end{equation}
		We henceforth require
		\[
		b_n-\vartheta_{n,*}<c_{n-1}<b_n;
		\]
		the upper level remains any regular value \(b_n<c_n<b_{n+1}\). Let \(\zeta\in C^1([0,\infty),[0,2\vartheta_{n,*}])\) satisfy
		\[
		\zeta=2\vartheta_{n,*}\quad\hbox{on }[0,s_{0,n}],
		\qquad
		\zeta=0\quad\hbox{on }[s_{1,n},\infty),
		\qquad
		\|\zeta'\|_\infty\le\frac1{8C_n}.
		\]
		Such a cutoff exists by the third bound in \eqref{eq3.tau-choice}. Define the modified functional
		\[
		\mathscr F_n(v)
		=
		f(v)-\chi_{n,\mathrm{cut}}(f(v))\zeta(\mathscr G(v))
		\]
		and the core
		\[
		\mathcal F_{a,n}
		=
		\{v\in W_{a,n}:\mathscr F_n(v)\le c_{n-1}\}.
		\]
		It is closed and contains \(W_{a,n-1}\).  If
		\[
		v\in\mathcal F_{a,n}\setminus W_{a,n-1},
		\]
		then \(\chi_{n,\mathrm{cut}}(f(v))\zeta(\mathscr G(v))>0\).  Hence \(|f(v)-b_n|<2\eta_n\) and \(\mathscr G(v)<s_{1,n}<\gamma_{\delta,n}^2\). The gradient bound \eqref{eq3.balanced-gradient-bound} implies
		\[
		v\in\mathcal T_{a,n}^{\le}(\delta).
		\]
		
		Let
		\[
		X(v)=-\nabla_{\Sigma_a}f(v)
		\]
		and let \(\varphi_t\) be its flow.  Lemma~\ref{Lem3.26} makes \(X\) locally Lipschitz on \(\widetilde\Sigma_a\), and \(X(v)\in T_v\Sigma_a\).  The first-derivative formula in the proof of that lemma gives
		\[
		\sup_{\{f\le c_n\}}\|X(v)\|_a<\infty.
		\]
		Thus the vector field is bounded on the invariant sublevel.  It is defined for all positive time on \(W_{a,n}\): the quotient decreases, the Hilbert sphere is complete, and
		\[
		f\le c_n
		\quad\to\quad
		\int_\Omega(v^+)^{2^*}\,dx\ge c_n^{-1},
		\]
		so the positive lower bound keeps the flow away from the boundary of \(\widetilde\Sigma_a\).  Hence the flow remains in \(\widetilde\Sigma_a\).  The positive cone is invariant.  Indeed, the constrained gradient is
		\[
		-\nabla_{\Sigma_a}f(v)
		=
		pG(v)^{-2}L_a^{-1}((v^+)^{p-1})
		-
		pG(v)^{-1}v,
		\qquad p=2^*.
		\]
		Since \(L_a^{-1}\) is positivity-preserving, multiplication of the flow equation by the integrating factor \(\exp(\int_0^t pG(\varphi_s(v))^{-1}\,ds)\) shows directly that a nonnegative initial datum remains nonnegative. In particular, every quotient sublevel contained in \(W_{a,n}\), including \(W_{a,n-1}\), is forward invariant.
		
		Along this flow, \eqref{eq3.G-bound} gives
		\[
		\begin{split}
			\frac{d}{dt}\mathscr F_n(\varphi_t(v))
			&=
			-\mathscr G
			+\chi_{n,\mathrm{cut}}'(f)\zeta(\mathscr G)\mathscr G
			+\chi_{n,\mathrm{cut}}(f)\zeta'(\mathscr G)
			d\mathscr G[\nabla_{\Sigma_a}f]
			\\
			&\le-\frac12\mathscr G.
		\end{split}
		\]
		Here all quantities on the middle line are evaluated at \(\varphi_t(v)\); \eqref{eq3.tau-choice}, the bound on \(\zeta'\), and \eqref{eq3.G-bound} give the last inequality.
		
		There is a level-independent transversality constant
		\begin{equation}\label{eq3.uniform-transversality}
			\mathscr G(v)\ge\mu_{n,*}
			\quad\hbox{whenever}\quad
			v\in W_{a,n},\qquad \mathscr F_n(v)=c_{n-1}.
		\end{equation}
		Indeed, if this failed for a sequence of admissible regular levels and boundary points \(v_k\), Proposition~\ref{Prop3.24} would give, after a subsequence,
		\[
		f(v_k)\to b_r
		\]
		for an integer \(r\).  The boundary identity and \(0\le\chi_{n,\mathrm{cut}}\zeta\le2\vartheta_{n,*}\) exclude \(r\ge n+1\). If \(r\le n-1\), then
		\[
		\mathscr F_n(v_k)\to b_r
		\le b_{n-1}
		<b_n-\vartheta_{n,*}<c_{n-1},
		\]
		whereas if \(r=n\), the cutoffs are fully active and
		\[
		\mathscr F_n(v_k)\to b_n-2\vartheta_{n,*}
		<b_n-\vartheta_{n,*}<c_{n-1}.
		\]
		Both alternatives contradict the boundary identity.  This proves \eqref{eq3.uniform-transversality}, uniformly over the stated level window.
		
		Every trajectory from \(W_{a,n}\) reaches \(\mathcal F_{a,n}\) in finite time. Otherwise, since \(\mathscr F_n\) is decreasing and bounded below, one could choose \(t_k\to\infty\) with
		\[
		\mathscr G(\varphi_{t_k}(v))\to0.
		\]
		The quotient levels remain in \((c_{n-1},c_n]\), so Corollary~\ref{Cor3.25} gives
		\[
		f(\varphi_{t_k}(v))\to b_n.
		\]
		But then
		\[
		\mathscr F_n(\varphi_{t_k}(v))
		\to b_n-2\vartheta_{n,*}<c_{n-1},
		\]
		a contradiction.
		
		For completeness, define the first hitting time on all of \(W_{a,n}\) by
		\[
		\tau_n(v)
		=
		\inf\{t\ge0:\mathscr F_n(\varphi_t(v))\le c_{n-1}\},
		\qquad
		\tau_n(v)=0\quad(v\in\mathcal F_{a,n}).
		\]
		It is finite by the preceding paragraph.  It is continuous, including on the target.  At an exterior point this follows from continuous dependence of the locally Lipschitz flow and the strict crossing.  At a boundary point, \eqref{eq3.uniform-transversality} and the preceding differential inequality give, in a neighbourhood of that point,
		\[
		\tau_n(w)
		\le
		\frac{4}{\mu_{n,*}}
		\bigl(\mathscr F_n(w)-c_{n-1}\bigr)^+,
		\]
		after shrinking the neighbourhood; the right-hand side tends to zero. At an interior target point the hitting time is identically zero nearby. Consequently
		\begin{equation}\label{eq3.stopped-deformation}
			H_n:[0,1]\times W_{a,n}\to W_{a,n},
			\qquad
			H_n(s,v)=\varphi_{s\tau_n(v)}(v),
		\end{equation}
		is continuous, \(H_n(0,\cdot)=\operatorname{id}\), \(H_n(1,W_{a,n})\subset\mathcal F_{a,n}\), and \(H_n(s,v)=v\) for \(v\in\mathcal F_{a,n}\).  Since every quotient sublevel is flow invariant and \(W_{a,n-1}\subset\mathcal F_{a,n}\), the same formula is a map of pairs for every \(s\).  The semigroup property and the definition of first hitting give
		\[
		\tau_n(\varphi_t(v))=\tau_n(v)-t
		\qquad(0\le t\le\tau_n(v)),
		\]
		so \(H_n(1,H_n(s,v))=H_n(1,v)\) and the following map-of-pairs diagram commutes:
		\[
		\begin{array}{ccc}
			(W_{a,n},W_{a,n-1})
			&\xrightarrow{\ H_n(s,\cdot)\ }&
			(W_{a,n},W_{a,n-1})\\[0.3em]
			{\scriptstyle H_n(1,\cdot)}\downarrow\phantom{\scriptstyle H_n(1,\cdot)}
			&&
			\phantom{\scriptstyle H_n(1,\cdot)}\downarrow{\scriptstyle H_n(1,\cdot)}
			\\[0.3em]
			(\mathcal F_{a,n},W_{a,n-1})
			&\xrightarrow{\ \operatorname{id}\ }&
			(\mathcal F_{a,n},W_{a,n-1}).
		\end{array}
		\]
		This proves the asserted strong deformation equivalence of pairs, with identity on every relative subspace already contained in the target. This is the \(L_a\)-analogue of the modified-functional construction in \cite[Proposition~6, pp.~256-258]{BahriCoron1988}.
		
		For later use, the scale-adapted cutoff and \(L_a\)-Riesz estimates give constants \(\delta_0,c_0,C_0>0\), depending only on \(n\) and \(\Omega\), such that
		\[
		c_0\le\|P_a\delta_{\lambda_i,x_i}\|_a\le C_0
		\]
		for every tuple in \(\mathcal P_n^{\le}(\delta)\) with \(0<\delta\le\delta_0\).  Indeed, the upper bound follows by testing the \(L_a\)-projection equation and using \(0\le P_a\delta_{\lambda,x}\le\delta_{\lambda,x}\).  For the lower bound, put \(d_x=\operatorname{dist}(x,\partial\Omega)\) and choose \(\chi_x\in C_c^\infty(B(x,3d_x/4))\) such that \(\chi_x=1\) on \(B(x,d_x/2)\) and \(|\nabla\chi_x|\le C/d_x\).  With \(\phi_{\lambda,x}=\chi_x\delta_{\lambda,x}\), the projection equation gives
		\[
		\|P_a\delta_{\lambda,x}\|_a
		\ge
		\frac{\langle P_a\delta_{\lambda,x},\phi_{\lambda,x}\rangle_a}
		{\|\phi_{\lambda,x}\|_a}
		=
		\frac{\sigma_N\int_\Omega
			\chi_x\delta_{\lambda,x}^{2^*}\,dz}
		{\|\phi_{\lambda,x}\|_a}.
		\]
		The Talenti tail and the cutoff bound yield, uniformly in \(x\),
		\[
		\sigma_N\int_\Omega
		\chi_x\delta_{\lambda,x}^{2^*}\,dz
		\ge 1-C(\lambda d_x)^{-N},
		\qquad
		\|\phi_{\lambda,x}\|_a^2
		\le C_\Omega\bigl(1+(\lambda d_x)^{-(N-2)}\bigr).
		\]
		Every tuple in \(\mathcal P_n^{\le}(\delta)\) satisfies \(\lambda d_x\ge\delta^{-1}\).  Decreasing \(\delta_0\) makes the numerator at least \(1/2\), and the two displayed inequalities give a uniform constant \(c_0>0\). Set
		\[
		c_{\rm prof}=\frac{(1-\delta_0)c_0}{nC_0}>0.
		\]
		Finally, fix the outer width \(\varepsilon\), and choose
		\[
		0<\delta<\min\{\varepsilon/4,\delta_0,c_{\rm prof}/2\}
		\]
		small enough for the modified-functional construction above, with its associated constants chosen accordingly.  We verify the required sequential closure statement directly for the weak inner parameter inequalities.  Suppose
		\[
		v_k\in\mathcal T_{a,n}^{\le}(\delta),
		\qquad
		v_k\to v
		\quad\hbox{strongly in }H^1_0(\Omega),
		\]
		and choose an inner-tube parameter tuple for each \(v_k\).  For fixed \(n\) and sufficiently small \(\delta\), we show directly that, after permutation and passage to a subsequence, these tuples are precompact.  If they were not, then, because \(\Omega\) is bounded, the weights remain in a compact positive simplex, and \(\lambda_i\ge\delta^{-1}\), some scale would satisfy \(\lambda_{i,k}\to\infty\).  Put
		\[
		B_{i,k}
		=\frac{P_a\delta_{\lambda_{i,k},x_{i,k}}}
		{\|P_a\delta_{\lambda_{i,k},x_{i,k}}\|_a}.
		\]
		The scale-boundary bound implies \(B_{i,k}\rightharpoonup0\) in \(H^1_0(\Omega)\).  To see this directly, for every smooth \(\phi\),
		\[
		\langle P_a\delta_{\lambda_{i,k},x_{i,k}},\phi\rangle_a
		=\sigma_N\int_\Omega
		\delta_{\lambda_{i,k},x_{i,k}}^{2^*-1}\phi
		=O(\lambda_{i,k}^{-(N-2)/2})\|\phi\|_\infty,
		\]
		and density, together with the uniform norm bounds, gives the claim. Positivity, balanced weights, and the preceding diagonal norm bounds give, on writing \(Q_{j,k}=P_a\delta_{\lambda_{j,k},x_{j,k}}\),
		\[
		\begin{aligned}
			\bigl\langle\Phi_{a,n}(q_k),B_{i,k}\bigr\rangle_a
			&=
			\frac{\sum_j\alpha_{j,k}\langle Q_{j,k},B_{i,k}\rangle_a}
			{\|\sum_j\alpha_{j,k}Q_{j,k}\|_a}\\
			&\ge
			\frac{\alpha_{i,k}\|Q_{i,k}\|_a}
			{\sum_j\alpha_{j,k}\|Q_{j,k}\|_a}
			\ge c_{\rm prof}.
		\end{aligned}
		\]
		Here \(\langle Q_{j,k},Q_{i,k}\rangle_a\ge0\) follows from the positive Green representation of \(L_a^{-1}\). Since \(\|v_k-\Phi_{a,n}(q_k)\|_a<\delta\),
		\[
		\langle v_k,B_{i,k}\rangle_a\ge c_{\rm prof}-\delta.
		\]
		On the other hand, strong convergence of \(v_k\) and weak convergence of \(B_{i,k}\) make the left-hand side tend to zero, a contradiction.  Hence all scales are bounded above.  The scale-boundary inequality then keeps all centers in a fixed compact subset of \(\Omega\), and the remaining parameters are precompact modulo permutation.  At a limiting tuple, the inner weak inequalities remain
		\[
		|n\alpha_i-1|\le\delta,\qquad
		\lambda_i\ge\delta^{-1},\qquad
		\lambda_i\operatorname{dist}(x_i,\partial\Omega)\ge\delta^{-1},
		\qquad
		\varepsilon_{ij}\le\delta,
		\]
		with \(x_i\in\Omega\), positive coefficients summing to one.  The limiting residual has norm at most \(\delta\).  Since \(\delta<\varepsilon\), this limiting tuple belongs to \(\mathcal P_n^{<}(\varepsilon)\), and its residual is strictly smaller than \(\varepsilon\).  Hence, also using \eqref{eq3.strict-closed-sandwich},
		\[
		v\in\mathcal T_{a,n}^{<}(\varepsilon)
		\subset\mathcal T_{a,n}^{\le}(\varepsilon),
		\]
		which proves \eqref{eq3.tube-nesting} using only the preceding localized compactness.
	\end{proof}
	
	\begin{proposition}
		\label{Prop3.29}
		Assume that the \(L_a\)-Dirichlet problem has no positive solution and fix \(n\ge1\).  Let \(\mathcal F_{a,n}\) be the modified-functional core from Lemma~\ref{Lem3.28}.  For every sufficiently small \(\varepsilon>0\), the inner width \(\delta\) in that lemma has the following property.  If \(\mathcal O_{a,n}\subset\Sigma_{a,+}\) is any open set satisfying
		\[
		\overline{\mathcal T_{a,n}^{\le}(\delta)}
		\subset\mathcal O_{a,n},
		\]
		then the inclusion
		\begin{equation}\label{eq3.correct-tube-inclusion}
			\begin{split}
				\iota_n:\bigl(
				&\mathcal F_{a,n}\cap\mathcal O_{a,n},\\
				&W_{a,n-1}\cap\mathcal O_{a,n}
				\bigr)
				\hookrightarrow
				(W_{a,n},W_{a,n-1})
			\end{split}
		\end{equation}
		induces an isomorphism in homology with \(\mathbb F_2\)-coefficients. In particular one may take \(\mathcal O_{a,n}=\mathcal T_{a,n}^{\le}(\varepsilon)\).
	\end{proposition}
	
	\begin{proof}
		Choose \(\delta<\varepsilon\) as in Lemma~\ref{Lem3.28}, and put
		\[
		E=\mathcal F_{a,n}
		\setminus\mathcal O_{a,n}.
		\]
		Since
		\[
		\mathcal F_{a,n}\setminus W_{a,n-1}
		\subset\mathcal T_{a,n}^{\le}(\delta)
		\subset\mathcal O_{a,n},
		\]
		one has \(E\subset W_{a,n-1}\).  Moreover, \(\mathcal O_{a,n}\) is open in the strong topology, so \(E\) is closed in \(\mathcal F_{a,n}\).  By the closure nesting,
		\[
		\mathcal F_{a,n}\setminus
		\overline{\mathcal T_{a,n}^{\le}(\delta)}
		\]
		is an open neighbourhood of \(E\) in \(\mathcal F_{a,n}\), and it is contained in \(W_{a,n-1}\).  Therefore
		\[
		\overline E^{\,\mathcal F_{a,n}}=E
		\subset
		\operatorname{int}_{\mathcal F_{a,n}}(W_{a,n-1}).
		\]
		Excision gives
		\[
		\begin{split}
			H_*\bigl(
			\mathcal F_{a,n}\cap\mathcal O_{a,n},
			W_{a,n-1}\cap\mathcal O_{a,n}
			\bigr)
			\cong
			H_*(\mathcal F_{a,n},W_{a,n-1}).
		\end{split}
		\]
		The strong deformation in Lemma~\ref{Lem3.28} gives
		\[
		H_*(\mathcal F_{a,n},W_{a,n-1})
		\cong
		H_*(W_{a,n},W_{a,n-1}).
		\]
		The composition is precisely the homomorphism induced by \eqref{eq3.correct-tube-inclusion}.
	\end{proof}
	
	\section{Barycenter test maps and energy estimates}\label{sec:test}

	For a compact metric space \(K\), let
	\[
	B_n(K)=\left\{\sum_{i=1}^{n}t_i\delta_{x_i}:x_i\in K,
	\ t_i\ge0,
	\sum_{i=1}^{n}t_i=1\right\}.
	\]
	The subset \(B_{n-1}(K)\subset B_n(K)\) consists of barycenters with support of cardinality at most \(n-1\).  We use the convention
	\[
	B_0(K)=\varnothing.
	\]
	Throughout this section \(K\Subset\Omega\) is a fixed compact set contained in the interior of a smooth compact core of \(\Omega\).  In Section~\ref{sec:topology-module} it will be chosen to contain the image of a closed manifold which realizes the prescribed homology class.
	
	Let \(\delta_{\lambda,x}\) be the normalized Talenti bubble introduced in \eqref{eq2.9}. Thus
	\[
	\int_{\mathbb R^N}|\nabla\delta_{\lambda,x}|^2\,dx=1,
	\qquad
	-\Delta\delta_{\lambda,x}=\sigma_N\delta_{\lambda,x}^{2^*-1},
	\qquad
	\sigma_N=\left(\int_{\mathbb R^N}\delta_{1,0}^{2^*}\,dx\right)^{-1}.
	\]
	Set
	\[
	p=2^*=\frac{2N}{N-2}.
	\]
	For \(\mu=\sum_i t_i\delta_{x_i}\in B_n(K)\) we put
	\begin{equation*}
		S_{a,\lambda,n}(\mu)=
		\sum_{i=1}^{n}t_iP_a\delta_{\lambda,x_i},
	\end{equation*}
	and define the quotient test map
	\begin{equation*}
		g_{a,\lambda,n}(\mu)=
		\frac{S_{a,\lambda,n}(\mu)}{\|S_{a,\lambda,n}(\mu)\|_a}
		\in\Sigma_{a,+}.
	\end{equation*}
	The map is continuous on \(B_n(K)\) and sends \(B_{n-1}(K)\) into the corresponding lower-particle family.  Moreover,
	\begin{equation*}
		\mathcal J_a(g_{a,\lambda,n}(\mu))
		=
		\mathfrak q_a(S_{a,\lambda,n}(\mu)),
	\end{equation*}
	where
	\[
	\mathfrak q_a(v)=
	\frac{\|v\|_a^{\frac{2N}{N-2}}}{\int_\Omega v^{2^*}\,dx}.
	\]
	After a positive solution \(v\) of the Euclidean equation has been obtained, \(u=\phi^{-1}v\) gives the corresponding hyperbolic solution.  Therefore all estimates below are written in the Euclidean conformal variables.

	\subsection{Projected \texorpdfstring{\(L_a\)}{La}-bubble interactions}

	The estimates in this subsection incorporate the positive potential directly into the Hilbert structure.  Let \(G_a\) be the Dirichlet Green kernel of \(L_a=-\Delta+a\) and write
	\[
	L_aG_a(\cdot,y)=\delta_y,
	\qquad
	G_a(\cdot,y)=0\quad\hbox{on }\partial\Omega.
	\]
	Since \(a\ge0\) and \(K\Subset\Omega\), the Green kernel satisfies, by \eqref{eq2.7},
	\begin{equation*}
		0<g_K\le G_a(x,y)\quad \text{for }x,y\in K,\ x\ne y,
	\end{equation*}
	with the usual convention that \(G_a(x,y)\to+\infty\) as \(x\to y\).
	
	\begin{lemma}\label{Lem4.1}
		Let \(N\ge3\) and \(K\Subset\Omega\).  Uniformly for \(x\in K\) and \(\lambda\ge2\),
		\begin{equation}\label{eq4.1}
			\|P_a\delta_{\lambda,x}\|_a^2
			=1+O(\varpi_N(\lambda)),
			\qquad
			\int_\Omega (P_a\delta_{\lambda,x})^{2^*}\,dx
			=\sigma_N^{-1}+O(\varpi_N(\lambda)).
		\end{equation}
	\end{lemma}
	
	\begin{proof}
		For the diagonal terms, compare \(P_a\delta_{\lambda,x}\) with the whole space bubble through the Green representation.  Since
		\[
		\delta_{\lambda,x}(z)
		=
		\sigma_N\int_{\mathbb R^N}\Gamma(z-\zeta)\delta_{\lambda,x}(\zeta)^{2^*-1}\,d\zeta,
		\]
		we have
		\[
		\delta_{\lambda,x}(z)-P_a\delta_{\lambda,x}(z)
		=
		\sigma_N\int_\Omega \bigl(\Gamma(z-\zeta)-G_a(z,\zeta)\bigr)
		\delta_{\lambda,x}(\zeta)^{2^*-1}\,d\zeta
		+
		\sigma_N\int_{\mathbb R^N\setminus\Omega}\Gamma(z-\zeta)
		\delta_{\lambda,x}(\zeta)^{2^*-1}\,d\zeta.
		\]
		For \(x\in K\), the boundary term is \(O(\lambda^{-(N+2)/2})\) uniformly on a fixed neighbourhood of \(K\).  Choose \(K'\) such that \(K\Subset K'\Subset\Omega\).  Since \(0\le\Gamma-G_a\le\Gamma\), the Talenti decay and the whole space Green representation show that the part with \(z\notin K'\) or \(\zeta\notin K'\) is \(O(\lambda^{-N})\).  On \(K'\times K'\), \eqref{eq2.6} and the change of variables \(z=x+\xi/\lambda\), \(\zeta=x+\eta/\lambda\) give
		\[
		\begin{aligned}
			&\iint_{K'\times K'}
			|\Gamma(z-\zeta)-G_a(z,\zeta)|
			\delta_{\lambda,x}(z)^{2^*-1}
			\delta_{\lambda,x}(\zeta)^{2^*-1}\,dz\,d\zeta\\
			&\quad\le
			C\lambda^{2-N}
			\iint_{\mathbb R^N\times\mathbb R^N}
			\Psi_N\!\left(\frac{|\xi-\eta|}{\lambda}\right)
			(1+|\xi|^2)^{-\frac{N+2}{2}}
			(1+|\eta|^2)^{-\frac{N+2}{2}}
			\,d\xi\,d\eta.
		\end{aligned}
		\]
		The last expression is \(O(\lambda^{-1})\) for \(N=3\), \(O(\lambda^{-2}\log\lambda)\) for \(N=4\), and \(O(\lambda^{-2})\) for \(N\ge5\).  Hence
		\[
		\iint_{\Omega\times\Omega}
		|\Gamma(z-\zeta)-G_a(z,\zeta)|
		\delta_{\lambda,x}(z)^{2^*-1}
		\delta_{\lambda,x}(\zeta)^{2^*-1}\,dz\,d\zeta
		=O(\varpi_N(\lambda)).
		\]
		The exterior-source contribution is covered explicitly by the same tail estimate: Tonelli's theorem and the whole space representation give
		\[
		\begin{aligned}
			&\int_\Omega \delta_{\lambda,x}(z)^{2^*-1}
			\int_{\mathbb R^N\setminus\Omega}\Gamma(z-\zeta)
			\delta_{\lambda,x}(\zeta)^{2^*-1}\,d\zeta\,dz\\
			&\qquad\le
			\sigma_N^{-1}\int_{\mathbb R^N\setminus\Omega}
			\delta_{\lambda,x}(\zeta)^{2^*}\,d\zeta
			=O_K(\lambda^{-N}).
		\end{aligned}
		\]
		Therefore
		\[
		\int_\Omega \delta_{\lambda,x}^{2^*-1}
		\bigl(\delta_{\lambda,x}-P_a\delta_{\lambda,x}\bigr)\,dz
		=
		O(\varpi_N(\lambda)).
		\]
		Since
		\[
		L_a(\delta_{\lambda,x}-P_a\delta_{\lambda,x})=a\,\delta_{\lambda,x}\ge0
		\quad\hbox{in }\Omega,
		\qquad
		\delta_{\lambda,x}-P_a\delta_{\lambda,x}=\delta_{\lambda,x}>0
		\quad\hbox{on }\partial\Omega,
		\]
		the maximum principle gives \(0\le P_a\delta_{\lambda,x}\le\delta_{\lambda,x}\). Moreover, uniformly for \(x\in K\),
		\[
		\int_{\mathbb R^N\setminus\Omega}\delta_{\lambda,x}^{2^*}\,dz
		=O_K(\lambda^{-N})=O_K(\varpi_N(\lambda)).
		\]
		Using \(\sigma_N\int_{\mathbb R^N}\delta_{\lambda,x}^{2^*}=1\), we obtain
		\[
		\begin{aligned}
			\|P_a\delta_{\lambda,x}\|_a^2
			&=
			\sigma_N\int_\Omega \delta_{\lambda,x}^{2^*-1}P_a\delta_{\lambda,x}\,dz
			\\
			&=1-
			\sigma_N\int_{\mathbb R^N\setminus\Omega}\delta_{\lambda,x}^{2^*}\,dz
			-\sigma_N\int_\Omega\delta_{\lambda,x}^{2^*-1}
			\bigl(\delta_{\lambda,x}-P_a\delta_{\lambda,x}\bigr)\,dz
			\\
			&=1+O(\varpi_N(\lambda)).
		\end{aligned}
		\]
		Using
		\[
		0\le s^{2^*}-t^{2^*}\le 2^*s^{2^*-1}(s-t)
		\qquad (0\le t\le s),
		\]
		we also obtain
		\[
		\begin{aligned}
			0\le
			\sigma_N^{-1}-\int_\Omega (P_a\delta_{\lambda,x})^{2^*}\,dz
			&=
			\int_{\mathbb R^N\setminus\Omega}\delta_{\lambda,x}^{2^*}\,dz
			+\int_\Omega\left(\delta_{\lambda,x}^{2^*}
			-(P_a\delta_{\lambda,x})^{2^*}\right)\,dz
			\\
			&\le
			\int_{\mathbb R^N\setminus\Omega}\delta_{\lambda,x}^{2^*}\,dz
			+2^*\int_\Omega \delta_{\lambda,x}^{2^*-1}
			\bigl(\delta_{\lambda,x}-P_a\delta_{\lambda,x}\bigr)\,dz
			\\
			&=O(\varpi_N(\lambda)).
		\end{aligned}
		\]
		This proves \eqref{eq4.1}.
	\end{proof}
	
	For \(L>1\), set
	\[
	\mathscr I_L
	=
	\{(\lambda,x)\in(0,\infty)\times\Omega:\lambda\ge L,\
	\lambda\operatorname{dist}(x,\partial\Omega)\ge L\},
	\]
	and introduce the normalized parameter fields
	\[
	X_0=\lambda\partial_\lambda,
	\qquad
	X_\ell=\lambda^{-1}\partial_{x_\ell},
	\quad 1\le\ell\le N.
	\]
	
	For \(N\ge3\), define
	\[
	\eta_N(\lambda)=
	\begin{cases}
		\lambda^{-(N-2)/2},&3\le N\le5,\\[0.2em]
		\lambda^{-2}(\log\lambda)^{2/3},&N=6,\\[0.2em]
		\lambda^{-2},&N\ge7.
	\end{cases}
	\]
	
	\begin{lemma}\label{Lem4.2}
		\begin{equation}\label{eq4.4}
			\sup_{(\lambda,x)\in\mathscr I_L}
			\left\{
			\begin{aligned}
				&\|P_a\delta_{\lambda,x}-P\delta_{\lambda,x}\|_a
				+\lambda\|\partial_\lambda(P_a\delta_{\lambda,x}
				-P\delta_{\lambda,x})\|_a\\
				&\quad+\lambda^{-1}\sum_{\ell=1}^N
				\|\partial_{x_\ell}(P_a\delta_{\lambda,x}
				-P\delta_{\lambda,x})\|_a
			\end{aligned}
			\right\}\to0
			\quad\hbox{as }L\to\infty.
		\end{equation}
		Uniformly for \((\lambda,x)\in\mathscr I_L\), the expression in braces is \(O_\Omega(\eta_N(\lambda))\).  Moreover,
		\begin{equation}\label{eq4.5}
			\sup_{(\lambda,x)\in\mathscr I_L}
			\sum_{\alpha,\beta=0}^N
			\left\|
			X_\alpha X_\beta
			\bigl(P_a\delta_{\lambda,x}-P\delta_{\lambda,x}\bigr)
			\right\|_a
			\to0
			\quad\hbox{as }L\to\infty.
		\end{equation}
		Uniformly for \((\lambda,x)\in\mathscr I_L\), the sum in \eqref{eq4.5} is \(O_\Omega(\eta_N(\lambda))\).
	\end{lemma}
	
	\begin{proof}
		Let \(R_{\lambda,x}=P_a\delta_{\lambda,x}-P\delta_{\lambda,x}\).  Then
		\[
		L_aR_{\lambda,x}=-aP\delta_{\lambda,x},
		\qquad R_{\lambda,x}=0\quad\hbox{on }\partial\Omega.
		\]
		By coercivity of \(L_a\),
		\[
		\|R_{\lambda,x}\|_a
		\le C\|aP\delta_{\lambda,x}\|_{H^{-1}(\Omega)}.
		\]
		The coefficient \(a\) is bounded on \(\overline\Omega\), and the Sobolev-H\"older estimate
		\[
		\|aP\delta_{\lambda,x}\|_{H^{-1}}
		\le C\|P\delta_{\lambda,x}\|_{L^{2N/(N+2)}(\Omega)}
		\le C\|\delta_{\lambda,x}\|_{L^{2N/(N+2)}(\Omega)}
		=o_\lambda(1)
		\]
		holds uniformly in \(x\).  Indeed, with \(q=2N/(N+2)\), scaling and radial integration give
		\[
		\begin{aligned}
			\|\delta_{\lambda,x}\|_{L^q(\Omega)}^q
			&\le
			C\lambda^{\frac{N-2}{2}q-N}
			\int_0^{C_\Omega\lambda}
			\frac{r^{N-1}}{(1+r^2)^{\frac{N-2}{2}q}}\,dr\\
			&\le C\eta_N(\lambda)^q.
		\end{aligned}
		\]
		This proves the zeroth-order assertion with the stated quantitative bound. Differentiating the identity for \(R_{\lambda,x}\) gives
		\[
		L_a\partial_\lambda R_{\lambda,x}=-a\partial_\lambda P\delta_{\lambda,x},
		\qquad
		L_a\partial_{x_\ell} R_{\lambda,x}=-a\partial_{x_\ell}P\delta_{\lambda,x}.
		\]
		Put
		\[
		Z_{\lambda,x,0}=\lambda\partial_\lambda P\delta_{\lambda,x},
		\qquad
		Z_{\lambda,x,\ell}
		=\lambda^{-1}\partial_{x_\ell}P\delta_{\lambda,x}.
		\]
		The corresponding normalized whole space tangents satisfy
		\[
		|\lambda\partial_\lambda\delta_{\lambda,x}|
		+\lambda^{-1}|\nabla_x\delta_{\lambda,x}|
		\le C\delta_{\lambda,x}.
		\]
		Differentiate the Green representation of \(P\delta_{\lambda,x}\). Since \(0<G_\Omega\le\Gamma\), the preceding bound and the whole space representation of \(\delta_{\lambda,x}\) give, for \(\alpha=0,\ldots,N\),
		\[
		|Z_{\lambda,x,\alpha}(y)|
		\le
		C\int_\Omega\Gamma(y-z)\delta_{\lambda,x}(z)^{2^*-1}\,dz
		\le C\delta_{\lambda,x}(y).
		\]
		Thus their \(L^{2N/(N+2)}(\Omega)\)-norms have the same uniform vanishing bound as the bubble itself.  Multiplication by the bounded coefficient \(a\), followed by coercivity in the differentiated equations, proves \eqref{eq4.4} with the bound \(O_\Omega(\eta_N(\lambda))\).
		
		For the second derivatives, the normalized whole space bubble satisfies
		\[
		\left|
		X_\alpha X_\beta
		\delta_{\lambda,x}(y)
		\right|
		\le C\delta_{\lambda,x}(y),
		\qquad 0\le\alpha,\beta\le N.
		\]
		The same Green-representation argument gives
		\[
		\left|
		X_\alpha X_\beta P\delta_{\lambda,x}(y)
		\right|
		\le
		C\int_\Omega
		\Gamma(y-z)\delta_{\lambda,x}(z)^{2^*-1}\,dz
		\le C\delta_{\lambda,x}(y).
		\]
		The fields \(X_\alpha\) act only on the parameters and therefore commute with \(L_a\).  Applying \(X_\alpha X_\beta\) to the equation for \(R_{\lambda,x}\) yields
		\[
		L_aX_\alpha X_\beta R_{\lambda,x}
		=
		-aX_\alpha X_\beta P\delta_{\lambda,x}.
		\]
		Coercivity and the preceding pointwise estimate prove \eqref{eq4.5}. The bound \(O_\Omega(\eta_N(\lambda))\) follows from the same \(L^{2N/(N+2)}\)-estimate as before.
	\end{proof}
	
	For two bubble parameters put
	\[
	\zeta_{ik}
	=
	\left(
	\frac{\lambda_i}{\lambda_k}
	+\frac{\lambda_k}{\lambda_i}
	+\lambda_i\lambda_k|x_i-x_k|^2
	\right)^{-\frac{N-2}{2}} .
	\]
	On the \(i\)-th parameter factor we use
	\[
	X_{i,0}=\lambda_i\partial_{\lambda_i},
	\qquad
	X_{i,\ell}=\lambda_i^{-1}\partial_{x_{i,\ell}},
	\quad 1\le\ell\le N .
	\]
	Fix \(m_0\in\mathbb N\).  For \(\zeta_0>0\) and \(1\le m\le m_0\), define
	\[
	\mathscr S_m(L,\zeta_0)
	=
	\left\{
	(\lambda_i,x_i)_{i=1}^m:
	(\lambda_i,x_i)\in\mathscr I_L,\quad
	\max_{i\ne k}\zeta_{ik}\le\zeta_0
	\right\}.
	\]
	If
	\[
	S_P(\alpha,q)=\sum_{i=1}^m\alpha_iP\delta_{\lambda_i,x_i},
	\qquad
	S_a(\alpha,q)=\sum_{i=1}^m\alpha_iP_a\delta_{\lambda_i,x_i},
	\qquad
	\sum_i\alpha_i^2=1,
	\]
	set
	\[
	\Phi_P=\frac{S_P}{\|S_P\|_a},
	\qquad
	\Phi_a=\frac{S_a}{\|S_a\|_a},
	\qquad
	\Phi_P^0=\frac{S_P}{\|\nabla S_P\|_2},
	\]
	\[
	Q_0(w)=
	\frac{\|\nabla w\|_2^{2^*}}
	{\displaystyle\int_\Omega(w^+)^{2^*}}.
	\]
	Here \(\mathscr D^r\) denotes any composition of \(r\) of the fields \(X_{i,\ell}\) and ambient coefficient derivatives of a fixed smooth extension from \(\{\sum\alpha_i^2=1\}\subset\mathbb R^m\).  These derivatives remain well defined on zero-coefficient faces.
	
	\begin{lemma}\label{Lem4.3}
		There is \(\zeta_0=\zeta_0(m_0)>0\) such that, uniformly for \(q\in\mathscr S_m(L,\zeta_0)\) and \(\alpha_i\ge0\),
		\begin{equation}\label{eq4.6}
			\min\{\|S_P(\alpha,q)\|_a,\|S_a(\alpha,q)\|_a\}
			\ge\frac12
		\end{equation}
		for all sufficiently large \(L\).  Moreover, with the notation above, one has the quantitative \(C^2\)-estimate
		\begin{equation}\label{eq4.7}
			\max_{0\le r\le2}
			\sup_{\mathscr S_m(L,\zeta_0)}
			\left\|\mathscr D^r(\Phi_a-\Phi_P)\right\|_a
			\le C_{m_0,\Omega}\eta_N(L)
			\to0.
		\end{equation}
		On the same set,
		\begin{equation}\label{eq4.7-dirichlet}
			\max_{0\le r\le2}
			\sup_{\mathscr S_m(L,\zeta_0)}
			\left\{
			\|\mathscr D^r(\Phi_P-\Phi_P^0)\|_a
			+
			\left|
			\mathscr D^r\bigl(\mathfrak q_a(S_P)-Q_0(S_P)\bigr)
			\right|
			\right\}
			\to0.
		\end{equation}
		Thus both the projection and the choice of normalization are controlled.
		
		The same convergence holds for the scalar quotient through order two:
		\begin{equation}\label{eq4.8}
			\max_{0\le r\le2}
			\sup_{\mathscr S_m(L,\zeta_0)}
			\left|
			\mathscr D^r\left[
			\mathcal J_a(\Phi_a)-\mathcal J_a(\Phi_P)
			\right]\right|
			\to0.
		\end{equation}
		After identifying the two tangent hyperplanes by the \(a\)-orthogonal projections, it also holds for the constrained gradient and Hessian:
		\begin{equation}\label{eq4.9}
			\begin{aligned}
				\sup_{\mathscr S_m(L,\zeta_0)}
				\bigl\|
				\nabla_{\Sigma_a}\mathcal J_a(\Phi_a)
				-
				\nabla_{\Sigma_a}\mathcal J_a(\Phi_P)
				\bigr\|_a&\to0,\\
				\sup_{\mathscr S_m(L,\zeta_0)}
				\bigl\|
				\operatorname{Hess}_{\Sigma_a}\mathcal J_a(\Phi_a)
				-
				\operatorname{Hess}_{\Sigma_a}\mathcal J_a(\Phi_P)
				\bigr\|_{\mathcal L(H^1_0,H^1_0{}^*)}
				&\to0.
			\end{aligned}
		\end{equation}
		Thus the \(P_a\)-family is a uniform \(C^2\)-small perturbation of the \(P\)-family across the whole retained separated parameter set.
	\end{lemma}
	
	\begin{proof}
		We first make the scale-interior uniformity explicit.  If \(d=\operatorname{dist}(x,\partial\Omega)\), choose a cutoff which is one on \(B(x,d/2)\), vanishes outside \(B(x,3d/4)\), and has gradient bounded by \(C/d\).  The Talenti tail gives
		\[
		\|\nabla[(1-\chi)\delta_{\lambda,x}]\|_2
		+
		\|\delta_{\lambda,x}\nabla\chi\|_2
		\to0
		\]
		uniformly as \(\lambda d\ge L\to\infty\).  Since \(P\delta_{\lambda,x}\) is the Dirichlet-orthogonal projection of \(\delta_{\lambda,x}|_\Omega\), its projection error is bounded by the error of this cutoff.  Hence its diagonal Dirichlet norm tends uniformly to one. The whole space two-bubble overlap is \(O(\zeta_{ik})\); the same cutoff argument transfers that bound to the projected bubbles.  Together with \(\int_\Omega a(P\delta_{\lambda,x})^2=o(1)\), this gives
		\[
		\max_i\left|
		\langle P\delta_i,P\delta_i\rangle_a-1
		\right|
		+
		\max_{i\ne k}
		\left|
		\langle P\delta_i,P\delta_k\rangle_a
		\right|
		\le \omega(L,\zeta_0),
		\]
		where \(\omega(L,\zeta_0)\to0\) first as \(\zeta_0\to0\) and then as \(L\to\infty\), uniformly for \(m\le m_0\).  Choose \(\zeta_0\) and then \(L\) so that \(\omega(L,\zeta_0)\le(4m_0)^{-1}\).  Gershgorin's theorem then gives
		\[
		\sum_{i,k}\alpha_i\alpha_k
		\langle P\delta_i,P\delta_k\rangle_a
		\ge\frac12\sum_i\alpha_i^2.
		\]
		Lemma~\ref{Lem4.2} transfers the same estimate, with a harmless further decrease of the constant, to the \(P_a\)-Gram matrix.  This proves \eqref{eq4.6}.
		
		For \(r=0,1,2\), Lemma~\ref{Lem4.2} and \(\sum_i|\alpha_i|\le m_0^{1/2}\) give
		\begin{equation}\label{eq4.10}
			\left\|
			\mathscr D^r(S_a-S_P)
			\right\|_a
			\le C_{m_0,\Omega}\eta_N(L).
		\end{equation}
		The corresponding derivatives of \(S_a\) and \(S_P\) are uniformly bounded.  To pass from \eqref{eq4.10} to normalized sums, we record the normalization derivatives rather than invoking them implicitly.  If \(\mathcal N(u)=u/\|u\|_a\), \(A=\|u\|_a\), then
		\[
		D\mathcal N(u)[h]
		=
		\frac hA-\frac{u\langle u,h\rangle_a}{A^3}
		\]
		and
		\[
		\begin{aligned}
			D^2\mathcal N(u)[h,k]
			&=
			-\frac{
				h\langle u,k\rangle_a+
				k\langle u,h\rangle_a+
				u\langle h,k\rangle_a}{A^3}\\
			&\quad+
			3\frac{
				u\langle u,h\rangle_a\langle u,k\rangle_a}{A^5}.
		\end{aligned}
		\]
		The lower bound \eqref{eq4.6}, these formulas, and \eqref{eq4.10} prove \eqref{eq4.7}.
		
		The Green-representation estimates in the proof of Lemma~\ref{Lem4.2} give, for \(0\le r\le2\),
		\[
		|\mathscr D^rP\delta_{\lambda,x}(y)|
		\le C\delta_{\lambda,x}(y).
		\]
		Radial integration on the bounded domain gives
		\[
		\|\delta_{\lambda,x}\|_2^2
		\le C
		\begin{cases}
			\lambda^{-1},&N=3,\\
			\lambda^{-2}\log\lambda,&N=4,\\
			\lambda^{-2},&N\ge5,
		\end{cases}
		\]
		uniformly in \(x\).  Cauchy-Schwarz, applied after differentiating the finite sum, therefore gives
		\[
		\max_{0\le r\le2}
		\left|
		\mathscr D^r\int_\Omega aS_P^2
		\right|
		\to0
		\]
		uniformly on \(\mathscr S_m(L,\zeta_0)\).  Since
		\[
		\|S_P\|_a^2
		=
		\|\nabla S_P\|_2^2+\int_\Omega aS_P^2
		\]
		and both leading norms have the Gram lower bound used in \eqref{eq4.6}, the displayed normalization formulas prove the first part of \eqref{eq4.7-dirichlet}.  Expanding
		\[
		\mathfrak q_a(S_P)
		=
		\frac{
			\bigl(\|\nabla S_P\|_2^2+\int_\Omega aS_P^2\bigr)^{2^*/2}}
		{\int_\Omega S_P^{2^*}}
		\]
		and differentiating twice proves its second part.
		
		On this family the two quantities
		\[
		\int_\Omega\Phi_P^{{2^*}},
		\qquad
		\int_\Omega\Phi_a^{{2^*}}
		\]
		are uniformly separated from zero.  Indeed, one coefficient has size at least \(m_0^{-1/2}\), all bubbles are nonnegative, and the diagonal \(L^{2^*}\)-norm tends uniformly to \(\sigma_N^{-1}\).  The maps
		\[
		v\longmapsto
		\left(\int_\Omega(v^+)^{2^*}\right)^{-1},
		\qquad
		v\longmapsto\nabla_{\Sigma_a}\mathcal J_a(v),
		\qquad
		v\longmapsto\operatorname{Hess}_{\Sigma_a}\mathcal J_a(v)
		\]
		have the following regularity on the resulting bounded subset: the first is uniformly \(C^2\), the second is uniformly \(C^1\), and the third is uniformly continuous.  The last assertion follows from the Nemytskii formulas in the proof of Lemma~\ref{Lem3.26} (Lipschitz when \(2^*\ge3\), and H\"older when \(2<2^*<3\)).  More explicitly, on every bounded set on which \(\int_\Omega(v^+)^{2^*}\) is separated from zero, after identifying tangent spaces by the \(a\)-orthogonal projections,
		\[
		\bigl\|
		\operatorname{Hess}_{\Sigma_a}\mathcal J_a(v)
		-\operatorname{Hess}_{\Sigma_a}\mathcal J_a(w)
		\bigr\|_{\mathcal L(H^1_0,H^1_0{}^*)}
		\le C\|v-w\|_a^{\vartheta},
		\qquad
		\vartheta=\min\{1,2^*-2\}.
		\]
		Applying their derivative formulas to \eqref{eq4.7} gives \eqref{eq4.8}-\eqref{eq4.9}.  The orthogonal projections identifying the tangent hyperplanes differ by \(O(\|\Phi_a-\Phi_P\|_a)\), so they do preserve the conclusion.
	\end{proof}
	
	\section{Dimension-dependent energy-drop mechanisms}\label{sec:matched}

	Throughout this section
	\[
	p=2^*=\frac{2N}{N-2},\qquad N\in\{3,4,5\},
	\]
	and \(K\Subset\Omega\) is fixed.  We use the common projected-bubble notation
	\[
	Q_{\lambda,x}=P_a\delta_{\lambda,x},
	\]
	whereas \(P\delta_{\lambda,x}\) continues to denote the ordinary Dirichlet projection.  For \( \mu=\sum_{i=1}^j t_i\delta_{x_i}\in B_j(K) \) we write
	\[
	S_{a,\lambda,j}(\mu)=\sum_{i=1}^j t_iQ_{\lambda,x_i},
	\qquad
	g_{a,\lambda,j}(\mu)
	=\frac{S_{a,\lambda,j}(\mu)}{\|S_{a,\lambda,j}(\mu)\|_a}.
	\]
	Terms of zero weight are omitted.  Thus
	\[
	\mathcal J_a(g_{a,\lambda,j}(\mu))
	=\mathfrak q_a(S_{a,\lambda,j}(\mu)),
	\qquad
	b_j=\sigma_Nj^{\frac{2}{N-2}}.
	\]
	The proof now separates into two genuinely different energy-drop mechanisms and then returns to a common fixed-stratum analytic interface.
	
	\subsection{The fixed-multiplicity mechanism in dimension three}
	
	In this subsection
	\[
	N=3,\qquad p=6,
	\]
	and the number of bubbles is fixed.  Every compactness constant, collision collar, and parameter-tube width below may therefore depend on that fixed number.  We also set
	\[
	A_2(t)=\sum_i t_i^2,\qquad
	B_6(t)=\sum_i t_i^6,\qquad
	\mathcal A_j(t)=\sigma_3\frac{A_2(t)^3}{B_6(t)},
	\qquad b_j=\sigma_3j^2.
	\]
	The power-mean inequality gives \( \mathcal A_j(t)\le b_j \), with equality precisely at \(t_i=1/j\) for every \(i\).
	
	\subsubsection{The \texorpdfstring{\(L_a\)}{La}-Green and Robin coefficients}
	
	In dimension three,
	\[
	\Gamma(z)=\frac1{4\pi|z|}.
	\]
	Define the regular part of the \(L_a\)-Green kernel and its Robin function by
	\[
	H_a(x,y)=\Gamma(x-y)-G_a(x,y),
	\qquad
	R_a(x)=H_a(x,x).
	\]
	The resolvent identity
	\[
	G_\Omega(x,y)-G_a(x,y)
	=
	\int_\Omega G_a(x,z)a(z)G_\Omega(z,y)\,dz
	\]
	and the three-dimensional convolution estimate show that \(H_a\) extends continuously across the diagonal on every compact subset of \(\Omega\times\Omega\).  In particular,
	\[
	M_K=\max_{x\in K}R_a(x)<\infty.
	\]
	The same identity and \(G_a\le G_\Omega\) give
	\[
	H_a(x,y)
	=
	H_\Omega(x,y)+G_\Omega(x,y)-G_a(x,y)
	\ge H_\Omega(x,y).
	\]
	Consequently
	\begin{equation}\label{eq5.robin-boundary}
		R_a(x)\ge R_\Omega(x)
		\ge\frac1{12\pi\operatorname{dist}(x,\partial\Omega)}
	\end{equation}
	whenever \(\operatorname{dist}(x,\partial\Omega)\le r_{\rm e}\), where \(r_{\rm e}\) is the exterior-sphere radius in Lemma~\ref{Lem2.2}.  This is the only boundary Robin estimate used below; it follows from an actual exterior-domain barrier, whose fixed radius gives uniformity in the pole. Moreover, by positivity of \(G_a\) and its positive diagonal singularity,
	\[
	g_K
	=
	\inf\{G_a(x,y):x,y\in K,\ x\ne y\}>0.
	\]
	We shall use the dimensional constant
	\[
	\mathfrak m_3
	=
	\sigma_3\int_{\mathbb R^3}\delta_{1,0}^{5}\,dz>0.
	\]
	The scaling of the normalized Talenti bubble gives
	\begin{equation}\label{eq5.mass}
		\sigma_3\int_{\mathbb R^3}\delta_{\lambda,x}^{5}\,dz
		=
		\mathfrak m_3\lambda^{-1/2}.
	\end{equation}
	
	The next result sharpens Lemma~\ref{Lem4.1} in dimension three. For \(x,y\in K\), set
	\[
	\varepsilon_\lambda(x,y)
	=
	(1+\lambda^2|x-y|^2)^{-1/2}.
	\]
	
	We first isolate the two-core estimates and fix the normalized parameter fields used below.  Put
	\[
	d\nu_{\lambda,x}(z)
	=\sigma_3\delta_{\lambda,x}(z)^5\,dz.
	\]
	For a family of centers \(x_1,\ldots,x_j\), put \(D_{ik}=\lambda|x_i-x_k|\), and let
	\[
	X_0=\lambda\partial_\lambda,
	\qquad
	X_{i,\ell}=\lambda^{-1}\partial_{x_{i,\ell}}
	\quad(1\le i\le j,\ 1\le\ell\le3).
	\]
	
	\begin{lemma}\label{Lem5.1}
		There are \(D_*>1\) and \(C_K>0\) such that, whenever \(\lambda\ge1\), \(x,y\in K\), and \(D=\lambda|x-y|\ge D_*\),
		\begin{align}
			\left|
			\iint_{\mathbb R^3\times\mathbb R^3}\Gamma(z-\zeta)\,
			d\nu_{\lambda,x}(z)d\nu_{\lambda,y}(\zeta)
			-\frac{\mathfrak m_3^2}{\lambda}\Gamma(x-y)
			\right|
			&\le C_K(1+D^2)^{-1},
			\label{eq5.two-core-Gamma}\\
			\lambda^{-1/2}
			\int_{\mathbb R^3}\delta_{\lambda,x}^4\delta_{\lambda,y}\,dz
			&\le C_K\bigl(\lambda^{-2}+(1+D^2)^{-1}\bigr),
			\label{eq5.two-core-41}\\
			\int_{\mathbb R^3}\delta_{\lambda,x}^4\delta_{\lambda,y}^2\,dz
			&\le C_K(1+D^2)^{-1}.
			\label{eq5.two-core-42}
		\end{align}
		
		If \(\mathscr X\) is any composition of at most two of these fields, then, in the density sense and pointwise in \(z\),
		\begin{equation}\label{eq5.normalized-source-derivatives}
			|\mathscr X(d\nu_{\lambda,x_i})|
			\le C\,d\nu_{\lambda,x_i},
			\qquad
			|\mathscr X Q_{\lambda,x_i}(z)|
			\le C_K Q_{\lambda,x_i}(z)
			\le C_K\delta_{\lambda,x_i}(z),
		\end{equation}
		where a field belonging to another center acts trivially.  Consequently, for arbitrary centers, every degree-six mixed monomial which involves at least two centers and is not of type \((5,1)\), as well as any of its derivatives \(\mathscr X\) of order at most two, has integral bounded by
		\begin{equation}\label{eq5.two-core-monomials}
			C_{j,K}\sum_{i<k}(1+D_{ik}^2)^{-1}.
		\end{equation}
		The estimate is uniform for weights \(0\le t_i\le1\).
	\end{lemma}
	
	\begin{proof}
		Write \(r=|x-y|\), choose \(e\in\mathbb S^2\) with \(y-x=re\), and use \(u=\lambda(z-x)\).  The normalized probability measure
		\[
		d\widehat\nu_{\lambda,x}
		=\frac{\lambda^{1/2}}{\mathfrak m_3}
		d\nu_{\lambda,x}
		\]
		satisfies
		\begin{equation}\label{eq5.two-core-moments}
			\int|z-x|\,d\widehat\nu_{\lambda,x}(z)\le\frac C\lambda,
			\qquad
			\widehat\nu_{\lambda,x}
			\bigl(\mathbb R^3\setminus B(x,s)\bigr)
			\le C(\lambda s)^{-2}.
		\end{equation}
		Both inequalities follow directly from the radial density \(C(1+|u|^2)^{-5/2}\,du\).
		
		The whole space Green representation gives
		\[
		I_{xy}:=
		\iint\Gamma(z-\zeta)\,
		d\nu_{\lambda,x}(z)d\nu_{\lambda,y}(\zeta)
		=\int\delta_{\lambda,y}(z)\,d\nu_{\lambda,x}(z).
		\]
		Put \(M_\lambda=\mathfrak m_3\lambda^{-1/2}\), \(A_x=B(x,r/4)\), and \(A_y=B(y,r/4)\).  On \(A_x\),
		\[
		\sup_{A_x}|\nabla\delta_{\lambda,y}|
		\le C\lambda^{3/2}D^{-2},
		\]
		whereas \eqref{eq5.two-core-moments} gives
		\[
		\int |z-x|\,d\widehat\nu_{\lambda,x}(z)\le C\lambda^{-1},
		\qquad
		\widehat\nu_{\lambda,x}(A_x^c)
		+\widehat\nu_{\lambda,x}(A_y)
		\le CD^{-2}.
		\]
		Moreover,
		\[
		\sup_{A_y}\delta_{\lambda,y}\le C\lambda^{1/2},
		\qquad
		\sup_{(A_x\cup A_y)^c}\delta_{\lambda,y}
		+\delta_{\lambda,y}(x)
		\le C\lambda^{1/2}D^{-1}.
		\]
		Consequently
		\[
		\begin{aligned}
			|I_{xy}-M_\lambda\delta_{\lambda,y}(x)|
			&\le M_\lambda
			\int_{A_x}|\delta_{\lambda,y}(z)-\delta_{\lambda,y}(x)|
			\,d\widehat\nu_{\lambda,x}(z)\\
			&\quad+M_\lambda\delta_{\lambda,y}(x)
			\widehat\nu_{\lambda,x}(A_x^c)
			+M_\lambda\int_{A_y}\delta_{\lambda,y}
			\,d\widehat\nu_{\lambda,x}\\
			&\quad+M_\lambda\int_{(A_x\cup A_y)^c}
			\delta_{\lambda,y}\,d\widehat\nu_{\lambda,x}
			\le CD^{-2}.
		\end{aligned}
		\]
		If \(c_3=\delta_{1,0}(0)\), the far-field identity in the whole space Green representation is \(\mathfrak m_3=4\pi c_3\).  Hence
		\[
		M_\lambda\delta_{\lambda,y}(x)
		=\mathfrak m_3c_3(1+D^2)^{-1/2}
		=\frac{\mathfrak m_3^2}{\lambda}\Gamma(x-y)
		+O(D^{-3}).
		\]
		This proves \eqref{eq5.two-core-Gamma}.
		
		The same change of variables gives, up to dimensional constants,
		\begin{align*}
			\lambda^{-1/2}\int
			\delta_{\lambda,x}^4\delta_{\lambda,y}
			&=\lambda^{-1}F_{41}(D),\\
			\int\delta_{\lambda,x}^4\delta_{\lambda,y}^2
			&=F_{42}(D),
		\end{align*}
		where
		\begin{align*}
			F_{41}(D)
			&=\int_{\mathbb R^3}
			(1+|u|^2)^{-2}(1+|u-De|^2)^{-1/2}\,du,\\
			F_{42}(D)
			&=\int_{\mathbb R^3}
			(1+|u|^2)^{-2}(1+|u-De|^2)^{-1}\,du.
		\end{align*}
		Split each integral into
		\[
		|u|\le D/4,\qquad |u-De|\le D/4,
		\qquad\text{and the complement}.
		\]
		On the first region, integrability of \((1+|u|^2)^{-2}\) gives \(CD^{-1}\) for \(F_{41}\) and \(CD^{-2}\) for \(F_{42}\).  On the second, integrate the other factor radially; the bounds are respectively \(CD^{-2}\) and \(CD^{-3}\).  On the complement, first intersect with \(\{|u|\le2D\}\) and then use \(|u-De|\ge|u|/2\) on \(\{|u|>2D\}\); radial integration gives \(CD^{-2}\) and \(CD^{-3}\), respectively.  Hence
		\[
		F_{41}(D)\le CD^{-1},
		\qquad
		F_{42}(D)\le CD^{-2}.
		\]
		Since \(\lambda^{-1}D^{-1}\le(\lambda^{-2}+D^{-2})/2\), \eqref{eq5.two-core-41}-\eqref{eq5.two-core-42} follow.
		
		For the derivative statement, direct differentiation of
		\[
		\delta_{\lambda,x}(z)^5
		=C\lambda^{5/2}
		(1+\lambda^2|z-x|^2)^{-5/2}
		\]
		shows that the density divided by itself remains bounded after every composition of at most two normalized fields.  Since \(G_a>0\), its Green representation then gives
		\[
		|\mathscr XQ_{\lambda,x}(z)|
		\le C\int_\Omega G_a(z,\zeta)\,d\nu_{\lambda,x}(\zeta)
		=CQ_{\lambda,x}(z),
		\]
		which proves \eqref{eq5.normalized-source-derivatives}; the final inequality is the maximum-principle bound.  After differentiating a mixed monomial, each resulting factor is therefore controlled by its original \(Q_i\). On the region where \(Q_i\) is the largest active bubble, a monomial not of type \((5,1)\) contains at least two foreign factors and is bounded by
		\[
		C_jQ_i^4\left(\sum_{k\ne i}Q_k\right)^2
		\le C_j\sum_{k\ne i}
		\delta_{\lambda,x_i}^4\delta_{\lambda,x_k}^2.
		\]
		If \(D_{ik}<D_*\), H\"older's inequality and scale invariance give \(\int\delta_{\lambda,x_i}^4\delta_{\lambda,x_k}^2\le C\), while \((1+D_{ik}^2)^{-1}\ge(1+D_*^2)^{-1}\).  Thus the same bound holds in the complementary, nonseparated regime.  Integration, \eqref{eq5.two-core-42}, and summation over the finitely many regions, monomials, and derivative terms prove \eqref{eq5.two-core-monomials}.  The weights are harmless because they lie in \([0,1]\).
	\end{proof}
	
	\begin{lemma}\label{Lem5.2}
		Uniformly for \(x\in K\), as \(\lambda\to\infty\),
		\begin{align}
			\|Q_{\lambda,x}\|_a^2
			&=
			1-\frac{\mathfrak m_3^2}{\lambda}R_a(x)
			+o_K(\lambda^{-1}),
			\label{eq5.diagN}\\
			\int_\Omega Q_{\lambda,x}^{6}\,dz
			&=
			\sigma_3^{-1}
			\left(
			1-\frac{6\mathfrak m_3^2}{\lambda}R_a(x)
			\right)
			+o_K(\lambda^{-1}).
			\label{eq5.diagD}
		\end{align}
		
		There are \(\varepsilon_*>0\), \(C_K>0\), and a function \(\rho_K:[1,\infty)\to[0,\infty)\), with \(\rho_K(\lambda)\to0\), such that for \(x,y\in K\), \(x\ne y\), whenever \(\varepsilon_\lambda(x,y)\le\varepsilon_*\),
		\begin{align}
			\langle Q_{\lambda,x},Q_{\lambda,y}\rangle_a
			&=
			\frac{\mathfrak m_3^2}{\lambda}G_a(x,y)
			+\mathcal E_\lambda(x,y),
			\label{eq5.mixedN}\\
			\int_\Omega Q_{\lambda,x}^{5}Q_{\lambda,y}\,dz
			&=
			\frac{\mathfrak m_3^2}{\sigma_3\lambda}G_a(x,y)
			+\mathcal N_\lambda(x,y),
			\label{eq5.mixedD}
		\end{align}
		where
		\begin{equation}\label{eq5.pairrem}
			|\mathcal E_\lambda(x,y)|
			+
			\sigma_3|\mathcal N_\lambda(x,y)|
			\le
			\frac{\rho_K(\lambda)}{\lambda}
			+C_K\varepsilon_\lambda(x,y)^2.
		\end{equation}
		For each fixed \(j\), after expanding \(\int_\Omega(\sum_i t_iQ_{\lambda,x_i})^6\), the sum of all terms not among the displayed diagonal and one-foreign-bubble terms is bounded by
		\begin{equation}\label{eq5.multirem}
			C_{j,K}\sum_{i<k}\varepsilon_\lambda(x_i,x_k)^2.
		\end{equation}
		The same bound holds after any composition of at most two of the normalized fields
		\[
		\lambda\partial_\lambda,
		\qquad
		\lambda^{-1}\partial_{x_{i,\ell}},
		\quad 1\le i\le j,\quad1\le\ell\le3,
		\]
		is applied to this mixed-monomial sum.  The constants are independent of the centers and weights.  Thus the derivative estimate is uniform in the normalized parameter geometry fixed above.
	\end{lemma}
	
	\begin{proof}
		With the measure defined above,
		\[
		\nu_{\lambda,x}(\mathbb R^3)
		=\mathfrak m_3\lambda^{-1/2},
		\qquad
		Q_{\lambda,x}(z)
		=\int_\Omega G_a(z,\zeta)\,d\nu_{\lambda,x}(\zeta).
		\]
		Put \(d_0=\frac14\operatorname{dist}(K,\partial\Omega)\), and let \(K^+=\{z:\operatorname{dist}(z,K)\le2d_0\}\Subset\Omega\). Let \(\omega_K\) be a modulus of continuity of \(H_a\) on \(K^+\times K^+\).  In the scaled variables the probability measure
		\[
		\widehat\nu_{\lambda,x}
		=\frac{\lambda^{1/2}}{\mathfrak m_3}\nu_{\lambda,x}
		\]
		has density
		\[
		h(u)=\frac{\sigma_3}{\mathfrak m_3}\delta_{1,0}(u)^5,
		\qquad h(u)\le C(1+|u|)^{-5}.
		\]
		Besides the tail bound in \eqref{eq5.two-core-moments}, the elementary three-region split relative to \(0\) and \(b\) gives, for \(L\ge2\),
		\begin{equation}\label{eq5.normalized-tail-convolution}
			\sup_{b\in\mathbb R^3}
			\int_{|u|>L}\frac{h(u)}{1+|u-b|}\,du
			\le CL^{-3}.
		\end{equation}
		Indeed, if \(|b|\le L/2\), then \(|u-b|\ge|u|/2\); if \(|b|>L/2\), split further into \(|u-b|\le|b|/2\) and its complement and integrate radially.
		
		Choose \(L=\lambda^{1/2}\) and \(s_\lambda=L/\lambda=\lambda^{-1/2}\). For large \(\lambda\), the balls \(B(x,s_\lambda)\) and \(B(y,s_\lambda)\) lie in \(K^+\).  On their product the oscillation of \(H_a\) is at most \(\omega_K(2s_\lambda)\).  On the complementary part use \(0\le H_a\le\Gamma\) and
		\[
		\int_{\mathbb R^3}\Gamma(z-\zeta)\,
		d\widehat\nu_{\lambda,y}(\zeta)
		=\frac{\lambda^{1/2}}{\mathfrak m_3}
		\delta_{\lambda,y}(z)
		\le\frac{C\lambda}{1+\lambda|z-y|}.
		\]
		Equations \eqref{eq5.two-core-moments} and \eqref{eq5.normalized-tail-convolution} therefore yield, uniformly for \(x,y\in K\), including \(x=y\),
		\[
		1-\widehat\nu_{\lambda,x}(\Omega)
		\le C_K\lambda^{-2};
		\]
		the same holds with \(y\) in place of \(x\), and these mass losses are included in the following estimate:
		\[
		\left|
		\iint_{\Omega\times\Omega}H_a(z,\zeta)\,
		d\widehat\nu_{\lambda,x}(z)d\widehat\nu_{\lambda,y}(\zeta)
		-H_a(x,y)
		\right|
		\le
		\omega_K(2\lambda^{-1/2})+C_K\lambda^{-1/2}.
		\]
		Consequently
		\begin{equation}\label{eq5.Haverage}
			\iint_{\Omega\times\Omega}H_a(z,\zeta)\,
			d\nu_{\lambda,x}(z)d\nu_{\lambda,y}(\zeta)
			=
			\frac{\mathfrak m_3^2}{\lambda}H_a(x,y)
			+\mathcal R_{K,H}(\lambda;x,y),
		\end{equation}
		where
		\[
		\sup_{x,y\in K}|\mathcal R_{K,H}(\lambda;x,y)|
		\le\frac{\eta_{K,H}(\lambda)}{\lambda},
		\qquad
		\eta_{K,H}(\lambda)
		\le C_K\bigl(\omega_K(2\lambda^{-1/2})+\lambda^{-1/2}\bigr)
		\to0.
		\]
		The same tail convolution estimate, now with \(L=\lambda\operatorname{dist}(K,\partial\Omega)\), also gives
		\begin{equation}\label{eq5.exterior-source-tail}
			0\le
			\iint_{(\mathbb R^3\times\mathbb R^3)\setminus(\Omega\times\Omega)}
			\Gamma(z-\zeta)\,
			d\nu_{\lambda,x}(z)d\nu_{\lambda,y}(\zeta)
			\le C_K\lambda^{-3}.
		\end{equation}
		Thus both the regular-kernel average and the source lost outside \(\Omega\) are uniform in every joint near/far regime of the two centers.
		
		Set \(\theta_{\lambda,x}=\delta_{\lambda,x}-Q_{\lambda,x}\). Green representation and the whole space representation of \(\delta_{\lambda,x}\) give
		\begin{align*}
			\delta_{\lambda,x}(z)-Q_{\lambda,x}(z)
			&=
			\sigma_3\int_\Omega
			H_a(z,\zeta)\delta_{\lambda,x}(\zeta)^5\,d\zeta\\
			&\quad+
			\sigma_3\int_{\mathbb R^3\setminus\Omega}
			\Gamma(z-\zeta)\delta_{\lambda,x}(\zeta)^5\,d\zeta.
		\end{align*}
		Equivalently,
		\[
		L_a\theta_{\lambda,x}=a\delta_{\lambda,x}\quad\hbox{in }\Omega,
		\qquad
		\theta_{\lambda,x}=\delta_{\lambda,x}\quad\hbox{on }\partial\Omega.
		\]
		Let \(h_{\lambda,x}\) be the \(L_a\)-harmonic extension of the boundary value.  Since \(x\in K\Subset\Omega\), the maximum principle gives \(0\le h_{\lambda,x}\le C_K\lambda^{-1/2}\).  Green representation, \(G_a(z,\zeta)\le C|z-\zeta|^{-1}\), and the change of variables \(q=\lambda(z-x)\), \(s=\lambda(\zeta-x)\) give
		\[
		\begin{aligned}
			0\le\theta_{\lambda,x}(z)
			&\le C_K\lambda^{-1/2}
			+C\lambda^{-3/2}
			\int_{|s|\le C_\Omega\lambda}
			|q-s|^{-1}(1+|s|^2)^{-1/2}\,ds.
		\end{aligned}
		\]
		A split into \(|s|\le|q|/2\), \(|q-s|\le|q|/2\), and the complement shows, uniformly for \(|q|\le C_\Omega\lambda\),
		\[
		\int_{|s|\le C_\Omega\lambda}
		|q-s|^{-1}(1+|s|^2)^{-1/2}\,ds
		\le C_\Omega\lambda.
		\]
		Hence
		\begin{equation}\label{eq5.global-projection-defect}
			0\le\theta_{\lambda,x}(z)\le C_K\lambda^{-1/2}
			\qquad(z\in\Omega,\ x\in K).
		\end{equation}
		Taking \(y=x\) in \eqref{eq5.Haverage} and using \eqref{eq5.exterior-source-tail} gives the literal defect identity
		\[
		\begin{aligned}
			\sigma_3\int_\Omega\delta_{\lambda,x}^5
			\theta_{\lambda,x}\,dz
			&=
			\iint_{\Omega\times\Omega}H_a(z,\zeta)\,
			d\nu_{\lambda,x}(z)d\nu_{\lambda,x}(\zeta)\\
			&\quad+
			\iint_{\Omega\times(\mathbb R^3\setminus\Omega)}
			\Gamma(z-\zeta)\,
			d\nu_{\lambda,x}(z)d\nu_{\lambda,x}(\zeta)\\
			&=\frac{\mathfrak m_3^2}{\lambda}R_a(x)
			+o_K(\lambda^{-1}).
		\end{aligned}
		\]
		The Talenti tail also gives
		\[
		\sigma_3\int_\Omega\delta_{\lambda,x}^6\,dz
		=1+O_K(\lambda^{-3}).
		\]
		Therefore
		\[
		\|Q_{\lambda,x}\|_a^2
		=
		\sigma_3\int_\Omega
		\delta_{\lambda,x}^{5}Q_{\lambda,x}\,dz
		=
		1-\frac{\mathfrak m_3^2}{\lambda}R_a(x)
		+o_K(\lambda^{-1}),
		\]
		which proves \eqref{eq5.diagN}.  Expanding \((\delta_{\lambda,x}-(\delta_{\lambda,x}-Q_{\lambda,x}))^6\), the linear term is
		\[
		-6\int_\Omega
		\delta_{\lambda,x}^{5}
		(\delta_{\lambda,x}-Q_{\lambda,x})\,dz
		=
		-\frac{6\mathfrak m_3^2}{\sigma_3\lambda}R_a(x)
		+o_K(\lambda^{-1}).
		\]
		Since \(0\le\theta_{\lambda,x}\le\delta_{\lambda,x}\), \eqref{eq5.global-projection-defect} and scaling give
		\[
		\sum_{k=2}^{6}\int_\Omega
		\delta_{\lambda,x}^{6-k}\theta_{\lambda,x}^{k}
		\le C\|\theta_{\lambda,x}\|_\infty^2
		\int_{\mathbb R^3}\delta_{\lambda,x}^{4}
		\le C_K\lambda^{-2}=o_K(\lambda^{-1}).
		\]
		Thus all terms at least quadratic in the defect are \(o_K(\lambda^{-1})\). This proves \eqref{eq5.diagD}.
		
		The \(O((\lambda|x-y|)^{-2})\) interaction scale is consistent with the estimates in the proof of \cite[Proposition~B.5]{BahriCoron1988}; the precise identity needed here is \eqref{eq5.two-core-Gamma}.  Decrease \(\varepsilon_*\), if necessary, so that \(\varepsilon_\lambda(x,y)\le\varepsilon_*\) implies \(\lambda|x-y|\ge D_*\).  Uniformly on this region,
		\begin{equation}\label{eq5.Gammaaverage}
			\iint_{\mathbb R^3\times\mathbb R^3}\Gamma(z-\zeta)\,
			d\nu_{\lambda,x}(z)d\nu_{\lambda,y}(\zeta)
			=
			\frac{\mathfrak m_3^2}{\lambda}\Gamma(x-y)
			+O\!\left(\varepsilon_\lambda(x,y)^2\right).
		\end{equation}
		The complete three-region proof is contained in Lemma~\ref{Lem5.1}. Uniform continuity of \(H_a\) suffices for \eqref{eq5.Haverage}.
		
		Since \(G_a=\Gamma-H_a\),
		\[
		\langle Q_{\lambda,x},Q_{\lambda,y}\rangle_a
		=
		\iint_{\Omega\times\Omega}G_a(z,\zeta)\,
		d\nu_{\lambda,x}(z)d\nu_{\lambda,y}(\zeta).
		\]
		Combining \eqref{eq5.Haverage}, \eqref{eq5.Gammaaverage}, and \eqref{eq5.exterior-source-tail} proves \eqref{eq5.mixedN} and its remainder estimate, initially with
		\[
		\rho_K(\lambda)
		=C_K\bigl(\eta_{K,H}(\lambda)+\lambda^{-2}\bigr)
		\to0.
		\]
		
		Moreover,
		\[
		\begin{aligned}
			\sigma_3\left|
			\int_\Omega(Q_{\lambda,x}^5-\delta_{\lambda,x}^5)
			Q_{\lambda,y}
			\right|
			&\le
			C_K\lambda^{-1/2}
			\int_\Omega\delta_{\lambda,x}^4\delta_{\lambda,y}\\
			&\le
			C_K\left(
			\lambda^{-2}+\varepsilon_\lambda(x,y)^2
			\right).
		\end{aligned}
		\]
		Here the last inequality is \eqref{eq5.two-core-41}. Enlarging \(\rho_K(\lambda)\) by \(C_K/\lambda\), which still tends to zero, absorbs the term \(C_K\lambda^{-2}\) into \(\rho_K(\lambda)/\lambda\). Together with \eqref{eq5.mixedN}, this proves \eqref{eq5.mixedD} and \eqref{eq5.pairrem}.
		
		Finally, \eqref{eq5.two-core-monomials} gives the undifferentiated and differentiated mixed-monomial bounds, because \((1+D_{ik}^2)^{-1}=\varepsilon_\lambda(x_i,x_k)^2\).  This proves \eqref{eq5.multirem} and its normalized first- and second-derivative version.
	\end{proof}
	
	\begin{remark}\label{Rem5.3}
		The two parts of \eqref{eq5.pairrem} have different origins and play complementary roles.  The term \(\rho_K(\lambda)/\lambda\) collects the uniform-continuity error of the regular Green kernel and the vanishing exterior and projection-tail errors.  If
		\[
		r=r_\lambda\to0,
		\qquad
		\lambda r_\lambda\to\infty,
		\]
		then
		\[
		\frac{G_a(x,y)}{\lambda}
		\asymp\frac1{\lambda r_\lambda},
		\qquad
		\varepsilon_\lambda(x,y)^2
		=O((\lambda r_\lambda)^{-2}).
		\]
		Thus the quadratic interaction error is absorbed by the singular negative pair term, whereas the center-independent \(\rho_K(\lambda)/\lambda\) is absorbed by the fixed Robin-Green gap.
		
		The Robin-Green formulas \eqref{eq5.diagN}-\eqref{eq5.mixedD} are used later only in the scalar \(C^0\)-norm stated in Lemma~\ref{Lem5.2}. The differentiated input for modulation is instead the uniform \(C^2\)-comparison \eqref{eq4.7}-\eqref{eq4.9} on the retained separated set.  This distinction is essential: an ordinary physical center derivative of an interaction at distance \(d/\lambda\) carries a factor \(\lambda\), whereas the normalized field \(\lambda^{-1}\partial_x\) remains uniformly controlled.
	\end{remark}
	
	\subsubsection{The fixed-dimensional three-regime estimate}
	
	For \(j\) fixed, define
	\[
	\varepsilon_{ik}
	=
	(1+\lambda^2|x_i-x_k|^2)^{-1/2},
	\qquad
	\mathcal E_{\lambda,j}(t,x)
	=
	\frac1{A_2(t)}
	\sum_{i<k}t_it_k\varepsilon_{ik}.
	\]
	When all \(\varepsilon_{ik}\) are small, Lemma~\ref{Lem5.2} gives
	\begin{align}
		\|S_{a,\lambda,j}(\mu)\|_a^2
		&=
		A_2(t)
		+\frac{\mathfrak m_3^2}{\lambda}
		\left[
		-\sum_i t_i^2R_a(x_i)
		+2\sum_{i<k}t_it_kG_a(x_i,x_k)
		\right]
		+\mathcal R_N,
		\label{eq5.sumN}\\
		\sigma_3\int_\Omega
		S_{a,\lambda,j}(\mu)^6\,dz
		&=
		B_6(t)
		+\frac{6\mathfrak m_3^2}{\lambda}
		\left[
		-\sum_i t_i^6R_a(x_i)
		+\sum_{i\ne k}t_i^5t_kG_a(x_i,x_k)
		\right]
		+\mathcal R_D,
		\label{eq5.sumD}
	\end{align}
	where, for a function \(\rho_{j,K}(\lambda)\to0\) independent of \((t,x)\),
	\begin{equation}\label{eq5.sumrem}
		|\mathcal R_N|+|\mathcal R_D|
		\le
		\frac{\rho_{j,K}(\lambda)}{\lambda}
		+
		C_{j,K}\sum_{i<k}\varepsilon_{ik}^2.
	\end{equation}
	
	Taylor expansion of the quotient in the weak-interaction region gives
	\begin{equation}\label{eq5.quotexp}
		\begin{aligned}
			\mathfrak q_a(S_{a,\lambda,j}(\mu))
			&=
			\mathcal A_j(t)
			\bigg[
			1+\frac{\mathfrak m_3^2}{\lambda}\mathscr L_j(t,x)
			+\mathcal R_{\lambda,j}(t,x)
			\bigg],
			\\
			\mathscr L_j(t,x)
			&=
			3\,
			\frac{-\sum_i t_i^2R_a(x_i)
				+2\sum_{i<k}t_it_kG_a(x_i,x_k)}
			{A_2(t)}
			\\
			&\quad
			-6\,
			\frac{-\sum_i t_i^6R_a(x_i)
				+\sum_{i\ne k}t_i^5t_kG_a(x_i,x_k)}
			{B_6(t)}.
		\end{aligned}
	\end{equation}
	where, for every fixed lower weight bound \(\theta>0\), uniformly when \(t_i\ge\theta\) and \(\max_{i<k}\varepsilon_{ik}\le\varepsilon_*\),
	\[
	\bigl|\mathcal R_{\lambda,j}(t,x)\bigr|
	\le
	\frac{\widetilde\rho_{j,K}(\lambda)}{\lambda}
	+C_{j,K}\sum_{i<k}\varepsilon_{ik}^2,
	\qquad
	\widetilde\rho_{j,K}(\lambda)\to0.
	\]
	Indeed \(A_2(t)\ge j^{-1}\) and \(B_6(t)\ge j^{-5}\).  The square of the first-order correction produced by Taylor's formula is bounded, for fixed \(j\), by
	\[
	C_{j,K}\left(\lambda^{-2}
	+\sum_{i<k}\varepsilon_{ik}^2\right)
	\]
	and is therefore included in the displayed remainder. At the balanced weight \(e_j=(1/j,\ldots,1/j)\), this reduces to the sharp Robin-Green formula
	\begin{equation}\label{eq5.equalexp}
		\begin{aligned}
			\mathfrak q_a
			\left(
			\frac1j\sum_{i=1}^jQ_{\lambda,x_i}
			\right)
			&=
			b_j\bigg[
			1+\frac{3\mathfrak m_3^2}{j\lambda}
			\left(
			\sum_{i=1}^jR_a(x_i)
			-2\sum_{i<k}G_a(x_i,x_k)
			\right)
			\\
			&\hspace{2.8cm}
			+\mathcal R_{\lambda,j}(e_j,x)
			\bigg].
		\end{aligned}
	\end{equation}
	
	\begin{lemma}\label{Lem5.4}
		Fix \(j\ge2\) and \(K\Subset\Omega\).
		
		\begin{enumerate}
			\item[\rm (a)] If \(F\subset\Delta_{j-1}\) is closed and \(e_j=(1/j,\ldots,1/j)\notin F\), then there are \(\eta_F>0\) and \(\Lambda_F>1\) such that
			\begin{equation}\label{eq5.weight-rigidity}
				\mathfrak q_a(S_{a,\lambda,j}(t,x))
				\le b_j(1-\eta_F)
			\end{equation}
			for \(\lambda\ge\Lambda_F\), \(t\in F\), and \(x\in K^j\), including coincident centers.
			
			\item[\rm (b)] For every \(\theta>0\) and \(R<\infty\), there are \(\eta_{\rm c}=\eta_{\rm c}(j,K,\theta,R)>0\) and \(\Lambda_{\rm c}>1\) such that
			\begin{equation}\label{eq5.collision-rigidity}
				\mathfrak q_a(S_{a,\lambda,j}(t,x))
				\le b_j(1-\eta_{\rm c})
			\end{equation}
			whenever \(\lambda\ge\Lambda_{\rm c}\), \(\min_i t_i\ge\theta\), and
			\[
			\min_{r<s}\lambda|x_r-x_s|\le R.
			\]
		\end{enumerate}
		Both estimates remain valid on arbitrary intersections of the indicated weight and collision sets, with the smaller of the corresponding gaps.
	\end{lemma}
	
	\begin{proof}
		We first isolate the transfer from the ordinary Dirichlet family.  Put
		\[
		T_{\lambda,j}(t,x)=\sum_i t_iP\delta_{\lambda,x_i}.
		\]
		For \(w^+\not\equiv0\), define the ordinary Dirichlet quotient by
		\[
		\mathfrak q_0(w)
		=\frac{\|\nabla w\|_2^6}
		{\displaystyle\int_\Omega(w^+)^6\,dx}.
		\]
		For some \(i_0\), \(t_{i_0}\ge1/j\).  Positivity and the diagonal estimates give, uniformly on the entire compactified common-scale parameter space,
		\[
		\begin{aligned}
			\|S_{a,\lambda,j}\|_a^2&\ge\frac1{2j},&
			\int_\Omega S_{a,\lambda,j}^6&\ge c_j>0,\\
			\|\nabla T_{\lambda,j}\|_2^2&\ge\frac1{2j},&
			\int_\Omega T_{\lambda,j}^6&\ge c_j>0
		\end{aligned}
		\]
		for all sufficiently large \(\lambda\).  Lemma~\ref{Lem4.2} and the three-dimensional \(L^2\)-scaling yield
		\[
		\sup_{t\in\Delta_{j-1},\,x\in K^j}
		\left\{
		\|S_{a,\lambda,j}-T_{\lambda,j}\|_a
		+
		\int_\Omega aT_{\lambda,j}^2
		\right\}\to0.
		\]
		Consequently
		\begin{equation}\label{eq5.global-transfer}
			\sup_{t\in\Delta_{j-1},\,x\in K^j}
			\left|
			\mathfrak q_a(S_{a,\lambda,j})
			-
			\mathfrak q_0(T_{\lambda,j})
			\right|\to0.
		\end{equation}
		This \(C^0\)-transfer is uniform over all separations and all nonnegative weights, including zero weights.
		
		We recall the ordinary equality mechanism and include the stronger compactness argument needed here.  We use the common-scale specialization of Appendix~B of Bahri-Coron: in their notation,
		\[
		n=j,\qquad \alpha=t,\qquad
		\widetilde\psi(t,x,\lambda)=\mathfrak q_0(T_{\lambda,j}(t,x)).
		\]
		Lemma~B.2 gives the finite-bubble convexity bound, Lemma~B.4 treats a small coefficient, and Lemma~B.7 gives a threshold bound, under a uniform positive lower bound on the coefficients, when the physical separation \(\min_{r<s}|x_r-x_s|\) is sufficiently small \cite{BahriCoron1988}.  Thus bounded \(\lambda_n\min_{r<s}|x_{r,n}-x_{s,n}|\) is covered along common-scale sequences with \(\lambda_n\to\infty\).  Corollary~B.3 supplies
		\[
		\sup_{t\in\Delta_{j-1},\,x\in K^j}
		\mathfrak q_0(T_{\lambda,j}(t,x))
		\le b_j+o_j(1).
		\]
		We now prove the stronger sequential rigidity, including convergence to equal weights and a uniform loss on closed degeneration strata:
		\begin{equation}\label{eq5.ordinary-rigidity}
			\begin{split}
				\lambda_n\to\infty,\qquad
				\mathfrak q_0(T_{\lambda_n,j}(t_n,x_n))\to b_j
				\quad\to\quad
				t_n\to e_j,\qquad
				\min_{r<s}\lambda_n|x_{r,n}-x_{s,n}|\to\infty .
			\end{split}
		\end{equation}
		The following cluster argument proves this conclusion. After taking \(t_n\to t\), partition the labels into clusters according to boundedness of \(\lambda_n|x_{r,n}-x_{s,n}|\).  Translate each cluster by one of its centers and pass to limits of all bounded relative centers.  Projection errors vanish because \(K\Subset\Omega\), interactions between distinct clusters vanish, and the numerator and denominator split into the finite sum of their whole space cluster contributions.  A zero-weight label is deleted.  Within each cluster, first merge labels whose limiting Talenti profiles coincide; the corresponding positive coefficients add and the limiting sum is unchanged.  If this leaves fewer than \(j\) effective positive profiles, the finite-bubble convexity bound and the power-mean inequality give at most \(b_{j-1}<b_j\).  Otherwise, every nonsingleton cluster contains two distinct nonproportional limiting profiles, and the strict-convexity step in Lemma~B.2 is strict.  Thus equality at \(b_j\) requires exactly \(j\) singleton clusters; the last finite-dimensional power-mean equality then forces \(t=e_j\).  This proves \eqref{eq5.ordinary-rigidity} and also proves that the loss is uniform on every closed set which excludes one of these equality conditions.
		
		If (a) failed, a sequence with \(\lambda_n\to\infty\), \(t_n\in F\), and quotient tending to \(b_j\) would, by \eqref{eq5.global-transfer} and \eqref{eq5.ordinary-rigidity}, satisfy \(t_n\to e_j\), contradicting the closedness of \(F\).  If (b) failed, the same argument would force every rescaled pair distance to infinity, contradicting the displayed collision bound.  Passing to the minimum of finitely many positive gaps proves the assertion about intersections.
	\end{proof}
	
	For later collar constructions it is useful to record the overlapping parameter regions before selecting their widths.  Given \(\theta,\tau>0\) and \(R>1\), set
	\begin{align*}
		\mathscr Z_j(\theta)
		&=
		\{(t,x):\min_i t_i\le\theta\},\\
		\mathscr U_j(\theta,\tau/2)
		&=
		\{(t,x):\min_i t_i\ge\theta,\
		\max_i|jt_i-1|\ge\tau/2\},\\
		\mathscr C_j(\theta,\tau,2R;\lambda)
		&=
		\{(t,x):\min_i t_i\ge\theta,\
		\max_i|jt_i-1|\le\tau,\
		\min_{r<s}\lambda|x_r-x_s|\le2R\},\\
		\mathscr A_j(\theta,\tau,R;\lambda)
		&=
		\{(t,x):\min_i t_i\ge\theta,\
		\max_i|jt_i-1|\le\tau,\
		\min_{r<s}\lambda|x_r-x_s|\ge R\},\\
		\mathscr W_j(\theta,\tau/2,2R;\lambda)
		&=
		\{(t,x):\min_i t_i>\theta,\
		\max_i|jt_i-1|<\tau/2,\
		\min_{r<s}\lambda|x_r-x_s|>2R\}.
	\end{align*}
	Thus the weight interface is covered twice on \(\tau/2\le\max_i|jt_i-1|\le\tau\), and the collision interface is covered twice on \(R\le\min_{r<s}\lambda|x_r-x_s|\le2R\), by \(\mathscr C_j\cap\mathscr A_j\).  The zero-weight and positive pieces meet on \(\min_i t_i=\theta\).  These wider analytic overlaps are exactly the ones used by the artificial interfaces in Section~\ref{sec:topology-module}; the strict fine stratum \(\mathscr W_j\) is disjoint from all three degeneration collars.
	
	\begin{theorem}\label{Thm5.5}
		Fix \(j\ge2\) and \(K\Subset\Omega\).  There are explicitly ordered constants
		\[
		0<\theta_j<\frac1{4j},\qquad
		0<\delta_j<\frac12,\qquad
		R_j>2,\qquad
		\lambda_j>1,
		\]
		and positive gaps \(\eta_{j,\rm z},\eta_{j,\rm u},\eta_{j,\rm c}\) such that, for every \(\lambda\ge\lambda_j\), the following assertions hold.
		
		\begin{enumerate}
			\item[\rm (i)] On the vanishing-weight collar \(\mathscr Z_j(\theta_j)\),
			\begin{equation}\label{eq5.zero-gap}
				\mathfrak q_a(S_{a,\lambda,j})\le
				b_j(1-\eta_{j,\rm z}).
			\end{equation}
			
			\item[\rm (ii)] On the positive unbalanced region \(\mathscr U_j(\theta_j,\delta_j/2)\),
			\begin{equation}\label{eq5.unbalanced}
				\mathfrak q_a(S_{a,\lambda,j})\le
				b_j(1-\eta_{j,\rm u}).
			\end{equation}
			
			\item[\rm (iii)] On the balanced collision collar \(\mathscr C_j(\theta_j,\delta_j,2R_j;\lambda)\),
			\begin{equation}\label{eq5.strong}
				\mathfrak q_a(S_{a,\lambda,j})\le
				b_j(1-\eta_{j,\rm c}).
			\end{equation}
			
			\item[\rm (iv)] On the strict balanced weak-interaction common-scale stratum \(\mathscr W_j(\theta_j,\delta_j/2,2R_j;\lambda)\), \eqref{eq5.sumN}-\eqref{eq5.quotexp} hold with the uniform scalar remainder
			\begin{equation}\label{eq5.stratified-remainder}
				|\mathcal R_N|+|\mathcal R_D|
				+|\mathcal R_{\lambda,j}|
				\le
				\frac{o_{j,K}(1)}{\lambda}
				+C_{j,K}\sum_{r<s}\varepsilon_{rs}^2.
			\end{equation}
			All degree-six terms involving three or more centers are included in this estimate by \eqref{eq5.multirem}.
		\end{enumerate}
		
		The three degeneration collars together with the strict fine stratum cover \(\Delta_{j-1}\times K^j\).  All estimates applicable on an overlap hold simultaneously, so every degeneration intersection carries the minimum of its applicable positive gaps.  The fixed-gap estimates apply on the degeneration strata, while the Robin--Green expansion applies on the strict fine stratum.
		
		The constants can also be chosen so that, on the \(\delta_j/2\)-balanced strict fine stratum
		\[
		\mathscr W_j(\theta_j,\delta_j/2,2R_j;\lambda),
		\]
		every Green pair coefficient in \(\mathscr L_j\) is at most \(-c_{j,K}<0\),
		\[
		\frac{G_a(x_r,x_s)}{\lambda}
		\ge\gamma_K\varepsilon_{rs},
		\]
		and the quadratic pair remainder is at most one half of the absolute first-order pair contribution.  Finally, for every prescribed \(\bar\delta,\bar\varepsilon\in(0,1)\), the fixed-gap conclusions hold, with new positive gaps and a new lower scale, on
		\[
		\max_i|jt_i-1|\ge\bar\delta
		\quad\hbox{or}\quad
		\left[
		\max_i|jt_i-1|<\bar\delta,\quad
		\max_{r<s}\varepsilon_{rs}\ge\bar\varepsilon
		\right].
		\]
	\end{theorem}
	
	\begin{proof}
		Choose any \(0<\delta_j<1/2\).  Since
		\[
		\max_i|jt_i-1|<\delta_j
		\quad\to\quad
		t_i>\frac{1-\delta_j}{j},
		\]
		choose
		\[
		0<\theta_j<
		\min\left\{\frac1{4j},\frac{1-\delta_j}{4j}\right\}.
		\]
		The two closed weight sets
		\[
		F_{\rm z}=\{\min_i t_i\le\theta_j\},
		\qquad
		F_{\rm u}=\{\min_i t_i\ge\theta_j,\
		\max_i|jt_i-1|\ge\delta_j/2\}
		\]
		are disjoint from \(e_j\).  Lemma~\ref{Lem5.4}(a) gives their separate gaps.  Choose \(R_j>2\) so large that
		\[
		(1+R_j^2)^{-1/2}<\varepsilon_*
		\]
		and provisionally fix it; it may be enlarged below. Lemma~\ref{Lem5.4}(b), with \(2R_j\), supplies the collision gap on the smaller balanced collision set. Taking the largest of the three lower scales proves (i)-(iii), including all their intersections.
		
		On \(\mathscr W_j(\theta_j,\delta_j/2,2R_j;\lambda)\), Lemma~\ref{Lem5.2} gives \eqref{eq5.sumN}-\eqref{eq5.sumrem}; the lower weight bound makes the Taylor expansion uniform and gives \eqref{eq5.quotexp}.  The diagonal and one-foreign-bubble errors contribute \(o_{j,K}(1)/\lambda\), while \eqref{eq5.multirem} accounts for every other monomial.  This proves \eqref{eq5.stratified-remainder}.
		
		At \(e_j\), every Green pair coefficient in \(\mathscr L_j\) equals \(-6/j\).  By decreasing \(\delta_j\), and then decreasing \(\theta_j\) if necessary, all those coefficients remain below a fixed negative number.  For every fixed \(R_j>1\),
		\[
		\inf_{\substack{\lambda\ge1,\ x,y\in K,\ x\ne y\\
				\lambda|x-y|\ge R_j}}
		\frac{G_a(x,y)}
		{\lambda(1+\lambda^2|x-y|^2)^{-1/2}}>0.
		\]
		Near the diagonal this follows from \(|x-y|G_a(x,y)\to(4\pi)^{-1}\), uniformly on \(K\), and away from it from positivity and compactness.  Hence \(G_a/\lambda\ge\gamma_K\varepsilon\).  Since
		\[
		\sum_{r<s}\varepsilon_{rs}^2
		\le
		\left(\max_{r<s}\varepsilon_{rs}\right)
		\sum_{r<s}\varepsilon_{rs},
		\]
		enlarging \(R_j\) absorbs the quadratic remainder pair by pair.  The center-independent \(o_{j,K}(1)/\lambda\) term is deliberately retained in \eqref{eq5.stratified-remainder}; it is absorbed only after the multiplicity \(m\) has been fixed.
		
		The decreases of \(\delta_j,\theta_j\) and the enlargement of \(R_j\) change the closed collars.  Reapply Lemma~\ref{Lem5.4} to these final weight sets and to the collision bound \(2R_j\), enlarge \(\lambda_j\), and rename the three positive gaps.  This proves {\rm(i)}-{\rm(iii)} for the final constants, including every collar intersection.  Apply the preceding expansion argument with the final positive lower weight and separation bounds to obtain {\rm(iv)} on its final strict fine stratum.
		
		For prescribed \(\bar\delta,\bar\varepsilon\), apply Lemma~\ref{Lem5.4}(a) to the closed unbalanced weight set.  On its complement all weights are bounded below, and \(\varepsilon_{rs}\ge\bar\varepsilon\) is equivalent to
		\[
		\lambda|x_r-x_s|
		\le(\bar\varepsilon^{-2}-1)^{1/2}.
		\]
		Lemma~\ref{Lem5.4}(b) then gives the other fixed gap.
	\end{proof}
	
	\subsubsection{Choice of one multiplicity and the energy drop}
	
	Choose an integer \(m\ge2\) such that
	\begin{equation}\label{eq5.mchoice}
		D_{m,K}:=(m-1)g_K-M_K>0.
	\end{equation}
	This is possible because \(g_K>0\) and \(M_K<\infty\).  At equal weights, \eqref{eq5.equalexp} gives
	\[
	\frac{3\mathfrak m_3^2}{m}
	\left(
	\sum_iR_a(x_i)-2\sum_{i<k}G_a(x_i,x_k)
	\right)
	\le
	-3\mathfrak m_3^2D_{m,K}.
	\]
	For later use, we record why this remains true near equal weights.  The coefficient of \(R_a(x_i)\) in \(\mathscr L_m\) is
	\[
	-3\frac{t_i^2}{A_2(t)}
	+6\frac{t_i^6}{B_6(t)},
	\]
	whereas the coefficient of \(G_a(x_i,x_k)\) is
	\[
	6t_it_k
	\left[
	\frac1{A_2(t)}
	-\frac{t_i^4+t_k^4}{B_6(t)}
	\right].
	\]
	At \(e_m\) these coefficients are \(3/m\) and \(-6/m\), respectively. Since \(m\) is now fixed and \eqref{eq5.mchoice} is strict, we may decrease \(\delta_m\) in Theorem~\ref{Thm5.5} so that throughout the \(\delta_m/2\)-balanced region,
	\begin{equation}\label{eq5.coeffnegative}
		\mathscr L_m(t,x)
		\le
		-\frac32D_{m,K}
		-c_{\rm pair}\sum_{i<k}
		\bigl(G_a(x_i,x_k)-g_K\bigr),
	\end{equation}
	for some \(c_{\rm pair}>0\).  The last nonpositive term records the part of a singular Green interaction which is stronger than the uniform lower bound \(g_K\). Indeed, the Robin contribution is at most \((3+o_{\delta_m}(1))M_K\), while every pair coefficient remains negative and the total pair contribution is at most \(-(3-o_{\delta_m}(1))(m-1)g_K\). Enlarge \(R_m\), if necessary, so that the pairwise absorption in the proof of Theorem~\ref{Thm5.5} also holds with \(c_{\rm pair}\) from \eqref{eq5.coeffnegative}.  After this final decrease of \(\delta_m\) and enlargement of \(R_m\), reapply the fixed-\(m\) rigidity lemma and rename the resulting collar gaps and lower scale.  Put
	\[
	\eta_m
	=
	\min\{\eta_{m,\rm z},\eta_{m,\rm u},\eta_{m,\rm c}\}>0.
	\]
	Thus all four strata below use the same final thresholds.
	
	\begin{proposition}\label{Prop5.6}
		Assume \(N=3\), and let \(K\Subset\Omega\).  Choose \(m\) by \eqref{eq5.mchoice}.  There are \(\lambda_0>1\) and \(\kappa_{m,K}>0\) such that, for every \(\lambda\ge\lambda_0\),
		\begin{equation}\label{eq5.globaldrop}
			\sup_{\mu\in B_m(K)}
			\mathcal J_a(g_{a,\lambda,m}(\mu))
			=
			\sup_{\mu\in B_m(K)}
			\mathfrak q_a(S_{a,\lambda,m}(\mu))
			\le
			b_m\left(1-\frac{\kappa_{m,K}}{\lambda}\right)
			<b_m.
		\end{equation}
	\end{proposition}
	
	\begin{proof}
		Apply the four-region cover in Theorem~\ref{Thm5.5}.  The three degeneration collars carry the fixed gap \(\eta_mb_m\), by \eqref{eq5.zero-gap}, \eqref{eq5.unbalanced}, and \eqref{eq5.strong}, respectively.  This remains true on all collar intersections.  After increasing \(\lambda_0\), it is stronger than the right-hand side of \eqref{eq5.globaldrop}.
		
		It remains only the strict balanced weak-interaction stratum.  By \eqref{eq5.quotexp}, \eqref{eq5.coeffnegative}, and \(\mathcal A_m(t)\le b_m\), the displayed first-order term is bounded by
		\[
		-\frac{3\mathfrak m_3^2D_{m,K}}{2\lambda}
		-\frac{\mathfrak m_3^2c_{\rm pair}}{\lambda}
		\sum_{i<k}\bigl(G_a(x_i,x_k)-g_K\bigr).
		\]
		The absolute value of the remainder is bounded by
		\[
		\frac{\widetilde\rho_{m,K}(\lambda)}{\lambda}
		+C_{m,K}\sum_{i<k}\varepsilon_{ik}^2,
		\qquad
		\widetilde\rho_{m,K}(\lambda)\to0.
		\]
		
		Choose \(r_0>0\), depending only on \(K\), so small that
		\[
		G_a(x,y)-g_K\ge\frac1{16\pi|x-y|}
		\qquad
		(x,y\in K,\ 0<|x-y|<r_0).
		\]
		This follows from \(|x-y|G_a(x,y)\to(4\pi)^{-1}\), uniformly for \(x,y\in K\), after also requiring \(r_0\le(16\pi g_K)^{-1}\).
		
		For the finitely many far pairs, \(|x_i-x_k|\ge r_0\), one has
		\[
		\varepsilon_{ik}^2
		\le\frac1{\lambda^2r_0^2}.
		\]
		Consequently their total quadratic remainder is \(O_{m,K,r_0}(\lambda^{-2})\).  Together with \(\widetilde\rho_{m,K}(\lambda)/\lambda\), it is bounded by
		\[
		\frac{3\mathfrak m_3^2D_{m,K}}{8\lambda}
		\]
		after increasing \(\lambda_0\).
		
		For a near pair, \(0<|x_i-x_k|<r_0\), the weak-interaction condition gives \(\lambda|x_i-x_k|>2R_m\).  Its negative singular contribution is at least
		\[
		\frac{\mathfrak m_3^2c_{\rm pair}}
		{16\pi\lambda|x_i-x_k|},
		\]
		whereas
		\[
		C_{m,K}\varepsilon_{ik}^2
		\le
		\frac{C_{m,K}}
		{\lambda^2|x_i-x_k|^2}.
		\]
		The final choice of \(R_m\) above was made so large that
		\[
		\frac{C_{m,K}}{2R_m}
		\le \frac{\mathfrak m_3^2c_{\rm pair}}{32\pi}.
		\]
		Then, pair by pair, one half of the negative singular contribution absorbs the quadratic remainder.  The other half remains nonpositive and may be dropped.  Thus, after one final increase of \(\lambda_0\),
		\[
		\mathfrak q_a(S_{a,\lambda,m}(\mu))
		\le
		b_m\left(
		1-\frac{9\mathfrak m_3^2D_{m,K}}{8\lambda}
		\right),
		\]
		and hence, a fortiori,
		\[
		\mathfrak q_a(S_{a,\lambda,m}(\mu))
		\le
		b_m\left(
		1-\frac{3\mathfrak m_3^2D_{m,K}}{4\lambda}
		\right).
		\]
		After the final value of \(\lambda_0\) has been fixed, take
		\[
		\kappa_{m,K}
		=
		\min\left\{
		\frac{3\mathfrak m_3^2D_{m,K}}4,\,
		\lambda_0\eta_m
		\right\}>0.
		\]
		The two fixed-gap branches satisfy the resulting estimate for every \(\lambda\ge\lambda_0\), and \eqref{eq5.globaldrop} follows.
	\end{proof}
	
	\subsection{The matched-scale mechanism in dimensions four and five}
	
	Throughout this subsection
	\[
	N\in\{4,5\},\qquad
	\gamma=\frac p2.
	\]
	The elementary inequality below is the source of uniformity in the number of bubbles.  Its pointwise application to projected bubbles yields this uniformity directly.  For nonnegative numbers \(z_i\) and coefficients \(\beta_i\ge0\), put
	\[
	Z=\sum_i\beta_i z_i,
	\qquad
	D_\beta(z)=
	Z^p-\gamma Z\sum_i\beta_i z_i^{p-1}
	+(\gamma-1)\sum_i z_i^p.
	\]
	
	\begin{lemma}\label{Lem5.7}
		The defect \(D_\beta\) is nonnegative.  Moreover:
		\begin{enumerate}
			\item[\rm (i)] If \(\beta_i\ge a>0\) for every \(i\), then
			\[
			D_\beta(z)
			\ge c_{p,a}\sum_{i\ne j}z_i^{p-1}z_j,
			\]
			where the constant is independent of \(n\).
			\item[\rm (ii)] For every \(0<a<1<b\), there is \(c_{p,a,b}>0\) such that, if \(\beta_\ell\notin[a,b]\), then
			\[
			D_\beta(z)\ge c_{p,a,b}z_\ell^p.
			\]
			The other coefficients are only required to be nonnegative.
		\end{enumerate}
	\end{lemma}
	
	\begin{proof}
		Let \(m=\max_i z_i\).  The assertion is immediate when \(m=0\); by homogeneity we may otherwise assume \(m=1\).  Choose \(k\) with \(z_k=1\), write \(y_i=z_i\), and set
		\[
		r=\sum_i\beta_i y_i,
		\qquad
		h_p(r)=r^p-\gamma r^2+\gamma-1.
		\]
		Direct expansion gives
		\begin{equation}\label{eq5.defect-decomposition}
			\begin{aligned}
				D_\beta(y)
				={}&h_p(r)
				+(\gamma-1)\sum_{j\ne k}y_j^p\\
				&+\gamma r\sum_{j\ne k}
				\beta_jy_j\bigl(1-y_j^{p-2}\bigr).
			\end{aligned}
		\end{equation}
		Since
		\[
		h_p'(r)=pr(r^{p-2}-1),\qquad h_p(1)=0,
		\]
		every term in \eqref{eq5.defect-decomposition} is nonnegative.
		
		Assume first that \(\beta_i\ge a>0\), and put
		\[
		u=\sum_{j\ne k}y_j.
		\]
		The last two terms in \eqref{eq5.defect-decomposition} give
		\[
		\begin{aligned}
			D_\beta(y)
			&\ge h_p(r)
			+\sum_{j\ne k}y_j
			\left[
			\gamma a^2(1-y_j^{p-2})
			+(\gamma-1)y_j^{p-1}
			\right]\\
			&\ge h_p(r)+c_{p,a}u.
		\end{aligned}
		\]
		On the other hand,
		\[
		\sum_{i\ne j}y_i^{p-1}y_j\le2u+u^2.
		\]
		If \(u\) is bounded, the linear term controls the last expression.  If \(u\) is large, then \(r\ge a(1+u)\), and hence
		\[
		h_p(r)\ge c_pr^p\ge c_{p,a}u^p\ge c_{p,a}u^2.
		\]
		This proves (i).
		
		For (ii), first suppose \(z_\ell<m\).  Choosing a maximal index different from \(\ell\) in \eqref{eq5.defect-decomposition} gives
		\[
		D_\beta(z)\ge(\gamma-1)z_\ell^p.
		\]
		It remains to consider \(z_\ell=m\); normalize \(m=1\) and choose \(k=\ell\).  If \(\beta_\ell\ge b\), then \(r\ge b>1\) and
		\[
		D_\beta\ge h_p(r)\ge h_p(b)>0.
		\]
		Suppose \(\beta_\ell\le a\), and fix \(\delta=\min\{1-a,b-1\}\).  If \(|r-1|\ge\delta/2\), the first term of \eqref{eq5.defect-decomposition} has a fixed positive lower bound.  Otherwise
		\[
		\sum_{j\ne\ell}\beta_jy_j\ge\delta/2.
		\]
		If at least half of this sum comes from \(y_j\le1/2\), the last term of \eqref{eq5.defect-decomposition} has a fixed positive lower bound.  In the opposite case at least one \(y_j>1/2\), and the middle term has the lower bound \((\gamma-1)2^{-p}\).  Rescaling by \(m^p\) proves (ii).
	\end{proof}
	
	\subsubsection{The all-pairs projected source estimate}
	
	For \(x,y\in K\), abbreviate
	\[
	\begin{aligned}
		A_{\lambda}(x,y)
		&=\int_\Omega
		\delta_{\lambda,x}^{p-1}Q_{\lambda,y}\,dz,\\
		B_{\lambda}(x,y)
		&=\int_\Omega
		Q_{\lambda,x}^{p-1}Q_{\lambda,y}\,dz.
	\end{aligned}
	\]
	We use the notation
	\[
	\log^+s=\max\{\log s,0\}\qquad(s>0).
	\]
	Set
	\[
	\varepsilon_N(\lambda)=
	\begin{cases}
		\lambda^{-2}(1+\log\lambda)^2,&N=4,\\[0.2em]
		\lambda^{-2}(1+\log\lambda),&N=5.
	\end{cases}
	\]
	
	\begin{lemma}\label{Lem5.8}
		There are \(C>0\) and \(\lambda_0>1\), depending only on \(\Omega\), \(a\), \(K\), and \(N\), such that
		\begin{equation}\label{eq5.source-transfer}
			0\le
			A_{\lambda}(x,y)-B_{\lambda}(x,y)
			\le
			C\varepsilon_N(\lambda)A_{\lambda}(x,y)
		\end{equation}
		for all \(x,y\in K\) and \(\lambda\ge\lambda_0\).  The estimate includes \(x=y\).  In addition,
		\begin{equation}\label{eq5.pair-lower}
			\sigma_N\int_\Omega
			Q_{\lambda,x}^{p-1}Q_{\lambda,y}\,dz
			\ge
			c_K(1+\lambda|x-y|)^{-(N-2)}
			\ge c_K'\lambda^{-(N-2)}.
		\end{equation}
	\end{lemma}
	
	\begin{proof}
		Write \(\delta_x=\delta_{\lambda,x}\), \(Q_x=Q_{\lambda,x}\), and \(w_x=\delta_x-Q_x\).  The maximum principle gives
		\[
		0<Q_x\le\delta_x.
		\]
		Moreover,
		\[
		L_aw_x=a\delta_x\quad\hbox{in }\Omega,
		\qquad
		w_x=\delta_x\quad\hbox{on }\partial\Omega.
		\]
		Let \(H_x\) be the \(L_a\)-harmonic extension of \(\delta_x|_{\partial\Omega}\).  Since \(K\Subset\Omega\),
		\[
		0\le H_x(z)
		\le\|\delta_x\|_{L^\infty(\partial\Omega)}
		\le C\lambda^{-\frac{N-2}{2}}
		\]
		by the maximum principle.  Green representation now gives the exact splitting
		\[
		w_x(z)
		=H_x(z)+\int_\Omega G_a(z,\zeta)a(\zeta)\delta_x(\zeta)\,d\zeta .
		\]
		Using the boundedness of \(a\) and \(G_a(z,\zeta)\le C|z-\zeta|^{2-N}\), we obtain
		\[
		w_x(z)
		\le C\lambda^{-\frac{N-2}{2}}
		+C\int_\Omega |z-\zeta|^{2-N}\delta_x(\zeta)\,d\zeta.
		\]
		Put
		\[
		\mathbf q=\lambda(z-x),\qquad q=|\mathbf q|.
		\]
		After the change of variables \(\mathbf s=\lambda(\zeta-x)\), the integral on the right is bounded by
		\[
		C\lambda^{\frac{N-6}{2}}
		\int_{|\mathbf s|\le C\lambda}
		|\mathbf q-\mathbf s|^{2-N}
		(1+|\mathbf s|^2)^{-\frac{N-2}{2}}\,d\mathbf s.
		\]
		For \(q\le2\), the following bounds are immediate by radial integration.  For \(q>2\), split the integral into
		\[
		|\mathbf s|\le q/2,\qquad
		|\mathbf q-\mathbf s|\le q/2,\qquad
		|\mathbf s|,|\mathbf q-\mathbf s|\ge q/2.
		\]
		The two near regions and the remaining radial integral give, uniformly for \(0\le q\le C\lambda\),
		\[
		\begin{aligned}
			\int_{|\mathbf s|\le C\lambda}
			|\mathbf q-\mathbf s|^{-2}(1+|\mathbf s|^2)^{-1}\,d\mathbf s
			&\le
			C\left(1+\log^+\frac{\lambda}{1+q}\right)
			\quad(N=4),\\
			\int_{|\mathbf s|\le C\lambda}
			|\mathbf q-\mathbf s|^{-3}(1+|\mathbf s|^2)^{-3/2}\,d\mathbf s
			&\le C(1+q)^{-1}
			\quad(N=5).
		\end{aligned}
		\]
		Indeed, the far radial integrals are respectively \(\int_{1+q}^{C\lambda}r^{-1}\,dr\) and \(\int_{1+q}^{\infty}r^{-2}\,dr\). It follows that
		\[
		w_x(z)\le
		\begin{cases}
			C\lambda^{-1}
			\left(1+\log^+\dfrac{\lambda}{1+q}\right),&N=4,\\[1.1em]
			C\lambda^{-1/2}(1+q)^{-1},&N=5.
		\end{cases}
		\]
		Consequently,
		\[
		\frac{w_x(z)}{\delta_x(z)}
		\le
		C\lambda^{-2}(1+q^2)L_N(q),
		\quad
		L_4(q)=1+\log^+\frac{\lambda}{1+q},
		\quad
		L_5(q)=1.
		\]
		Since \(Q_y\le\delta_y\), we also have the global envelope
		\[
		Q_y(z)\le
		C\lambda^{\frac{N-2}{2}}
		(1+\lambda|z-y|)^{-(N-2)}.
		\]
		Since
		\[
		0\le \delta_x^{p-1}-Q_x^{p-1}
		\le C\delta_x^{p-2}w_x,
		\]
		put
		\[
		\mathbf d=\lambda(y-x),\qquad D=|\mathbf d|.
		\]
		The change of variables \(\mathbf q=\lambda(z-x)\) bounds the numerator in \eqref{eq5.source-transfer} by
		\[
		C\lambda^{-2}
		\int_{|\mathbf q|\le C\lambda}
		(1+|\mathbf q|)^{-N}L_N(|\mathbf q|)
		(1+|\mathbf q-\mathbf d|)^{-(N-2)}\,d\mathbf q.
		\]
		We claim that
		\begin{equation}\label{eq5.convolution-bound}
			\begin{aligned}
				&\int_{|\mathbf q|\le C\lambda}
				(1+|\mathbf q|)^{-N}L_N(|\mathbf q|)
				(1+|\mathbf q-\mathbf d|)^{-(N-2)}\,d\mathbf q\\
				&\qquad\le
				C(1+\log\lambda)^{1+\mathbf 1_{\{N=4\}}}
				(1+D)^{-(N-2)}.
			\end{aligned}
		\end{equation}
		If \(D\le4\), the tail of the product is integrable, and \eqref{eq5.convolution-bound} follows directly from \(L_4\le C(1+\log\lambda)\).  Suppose \(D>4\), and split the domain into
		\[
		E_1=\{|\mathbf q|\le D/2\},\qquad
		E_2=\{|\mathbf q-\mathbf d|\le D/2\},\qquad
		E_3=(E_1\cup E_2)^c.
		\]
		On \(E_1\cup E_3\), the second factor is bounded by \(C(1+D)^{-(N-2)}\).  Hence these two contributions are at most
		\[
		C(1+D)^{-(N-2)}
		\int_{|\mathbf q|\le C\lambda}
		(1+|\mathbf q|)^{-N}L_N(|\mathbf q|)\,d\mathbf q.
		\]
		The last integral is \(O(1+\log\lambda)\) for \(N=5\) and \(O((1+\log\lambda)^2)\) for \(N=4\).  On \(E_2\), \(|\mathbf q|\ge D/2\), so
		\[
		\begin{aligned}
			&\int_{E_2}
			(1+|\mathbf q|)^{-N}L_N(|\mathbf q|)
			(1+|\mathbf q-\mathbf d|)^{-(N-2)}\,d\mathbf q\\
			&\qquad\le
			C D^{-N}(1+\log\lambda)^{\mathbf 1_{\{N=4\}}}
			\int_0^{D/2}(1+r)^{-(N-2)}r^{N-1}\,dr\\
			&\qquad\le
			C(1+\log\lambda)^{\mathbf 1_{\{N=4\}}}D^{-(N-2)}.
		\end{aligned}
		\]
		This proves \eqref{eq5.convolution-bound}.  In particular, the same estimate covers
		\[
		D=0,\qquad 0<D=O(1),\qquad
		1\ll D\ll\lambda,\qquad D\asymp\lambda.
		\]
		We have therefore proved
		\[
		A_{\lambda}(x,y)-B_{\lambda}(x,y)
		\le
		C\varepsilon_N(\lambda)
		(1+\lambda|x-y|)^{-(N-2)}.
		\]
		
		It remains to compare the last factor with \(A_{\lambda}(x,y)\). Choose
		\[
		K\Subset K_0\Subset K_1\Subset\Omega.
		\]
		After assigning the diagonal value \(1\), the quotient \(G_a(z,\zeta)/\Gamma(z-\zeta)\) is positive and continuous on \(\overline K_1\times\overline K_0\).  Fix \(R>1\).  For all sufficiently large \(\lambda\), \(B(y,R/\lambda)\subset K_0\) uniformly for \(y\in K\), and the Green representation gives
		\[
		Q_y(z)\ge
		c\sigma_N\int_{B(y,R/\lambda)}
		\Gamma(z-\zeta)\delta_y(\zeta)^{p-1}\,d\zeta,
		\qquad z\in K_1.
		\]
		With \(\mathbf q=\lambda(z-y)\) and \(\boldsymbol\eta=\lambda(\zeta-y)\), the last expression is a positive dimensional constant times
		\[
		\lambda^{\frac{N-2}{2}}F_R(\mathbf q),
		\qquad
		F_R(\mathbf q)=
		\int_{|\boldsymbol\eta|\le R}
		|\mathbf q-\boldsymbol\eta|^{2-N}
		(1+|\boldsymbol\eta|^2)^{-\frac{N+2}{2}}
		\,d\boldsymbol\eta .
		\]
		There is \(c_R>0\) such that
		\[
		F_R(\mathbf q)\ge
		c_R(1+|\mathbf q|^2)^{-\frac{N-2}{2}}
		\qquad(\mathbf q\in\mathbb R^N).
		\]
		For \(|\mathbf q|\le2R\), this follows from positivity and compactness; for \(|\mathbf q|>2R\), integrate over \(|\boldsymbol\eta|\le R\) and use \(|\mathbf q-\boldsymbol\eta|\asymp|\mathbf q|\). Consequently, after increasing \(\lambda_0\),
		\[
		Q_y(z)\ge c\delta_y(z)
		\qquad(z\in K_1,\ y\in K,\ \lambda\ge\lambda_0).
		\]
		For \(z\in B(x,\lambda^{-1})\),
		\[
		\delta_x(z)^{p-1}\ge c\lambda^{\frac{N+2}{2}},
		\qquad
		\delta_y(z)\ge
		c\lambda^{\frac{N-2}{2}}(1+D)^{-(N-2)}.
		\]
		Integration over this ball now yields
		\[
		A_{\lambda}(x,y)
		\ge c(1+\lambda|x-y|)^{-(N-2)},
		\]
		which proves \eqref{eq5.source-transfer}.
		
		After enlarging \(\lambda_0\), \eqref{eq5.source-transfer} gives
		\[
		B_{\lambda}(x,y)\ge\frac12A_{\lambda}(x,y)
		\ge
		c(1+\lambda|x-y|)^{-(N-2)}.
		\]
		The second inequality in \eqref{eq5.pair-lower} follows from the boundedness of \(K\).  This proof is unchanged when the centers coincide.
	\end{proof}
	
	\begin{remark}
		The weighted second moment on the intermediate annulus produces the logarithms in \(\varepsilon_N\) when the second center remains at a fixed positive distance; Lemma~\ref{Lem4.1} gives the separate diagonal single-bubble rates.  The all-pairs form of \eqref{eq5.source-transfer}, rather than its pointwise sharpness for one pair, is what prevents an \(n^2\) loss below.
	\end{remark}
	
	\subsubsection{The matched-scale energy drop}
	
	\begin{proposition}\label{Prop5.10}
		Fix \(A>0\), and set
		\[
		\lambda_n=A n^{\frac{2}{N-2}}.
		\]
		There are \(n_A\ge2\) and \(\kappa_A>0\) such that, for every \(n\ge n_A\),
		\begin{equation}\label{eq5.matched-drop}
			\sup_{\mu\in B_n(K)}
			\mathcal J_a(g_{a,\lambda_n,n}(\mu))
			\le
			b_n\left(1-\frac{\kappa_A}{n}\right)
			<b_n.
		\end{equation}
		All constants are uniform over zero weights, repeated centers, and arbitrary collision patterns.
	\end{proposition}
	
	\begin{proof}
		Fix \(\mu\in B_n(K)\), let \(r\le n\) be its support cardinality, and write
		\[
		\mu=\sum_{i=1}^r t_i\delta_{x_i},
		\qquad t_i>0.
		\]
		Put
		\[
		S=\sum_i t_iQ_{\lambda_n,x_i},
		\qquad
		V=\frac{\sqrt n\,S}{\|S\|_a}
		=\sum_i\beta_iQ_i.
		\]
		Then
		\[
		\beta_i=\frac{\sqrt n\,t_i}{\|S\|_a}.
		\]
		In particular, \(\|V\|_a^2=n\).  Define
		\[
		\mathcal R
		=
		\sigma_N\int_\Omega
		D_\beta(Q_1,\ldots,Q_r)\,dz.
		\]
		
		The equation defining \(Q_i\) gives the exact norm identity
		\[
		n
		=
		\sigma_N\int_\Omega
		V\sum_i\beta_i\delta_i^{p-1}\,dz.
		\]
		Let
		\[
		\mathcal E
		=
		\sigma_N\int_\Omega
		V\sum_i\beta_i
		\left(\delta_i^{p-1}-Q_i^{p-1}\right)\,dz.
		\]
		Expanding the definition of \(\mathcal R\) yields
		\begin{equation}\label{eq5.energy-identity}
			\sigma_N\int_\Omega V^p\,dz-n
			=
			\mathcal R-\gamma\mathcal E
			+(\gamma-1)
			\left(
			n-\sigma_N\sum_{i=1}^r\int_\Omega Q_i^p\,dz
			\right).
		\end{equation}
		Set
		\[
		\mathcal Q
		=
		(\gamma-1)
		\left(
		n-\sigma_N\sum_{i=1}^r\int_\Omega Q_i^p\,dz
		\right).
		\]
		Since \(Q_i\le\delta_i\) and \(\sigma_N\int_{\mathbb R^N}\delta_i^p=1\),
		\[
		\mathcal Q\ge(\gamma-1)(n-r).
		\]
		In particular, \(\mathcal Q\ge\gamma-1\) when \(r<n\).  Suppose now that \(r=n\).  If some \(\beta_i\notin[1/2,2]\), Lemma~\ref{Lem5.7}(ii) and the single-core estimate
		\[
		\sigma_N\int_\Omega Q_i^p\,dz\ge c_K
		\]
		give \(\mathcal R\ge c_K\).  If all coefficients belong to \([1/2,2]\), Lemma~\ref{Lem5.7}(i) and \eqref{eq5.pair-lower} give
		\[
		\mathcal R
		\ge
		c_K\sum_{i\ne j}
		\sigma_N\int_\Omega Q_i^{p-1}Q_j\,dz
		\ge
		c_Kn(n-1)\lambda_n^{-(N-2)}.
		\]
		Consequently,
		\begin{equation}\label{eq5.total-gap}
			\mathcal R+\mathcal Q
			\ge
			c_K\min\left\{
			1,n(n-1)\lambda_n^{-(N-2)}
			\right\}
			\ge c_A>0
		\end{equation}
		for all sufficiently large \(n\), including the case \(r<n\). Here the matched scale gives the exact relation
		\[
		n(n-1)\lambda_n^{-(N-2)}
		=A^{-(N-2)}\left(1-\frac1n\right).
		\]
		
		Writing \(A_{ij}=A_{\lambda_n}(x_i,x_j)\) and \(B_{ij}=B_{\lambda_n}(x_i,x_j)\), Lemma~\ref{Lem5.8} gives
		\[
		\begin{aligned}
			\sigma_N\sum_{i,j=1}^r\beta_i\beta_jA_{ij}
			&=
			\sigma_N\int_\Omega
			V\sum_{i=1}^r\beta_i\delta_i^{p-1}\,dz\\
			&=\|V\|_a^2=n.
		\end{aligned}
		\]
		This exact normalization, rather than a pairwise summation bound, removes any factor \(r^2\) or \(n^2\).  Consequently,
		\[
		\begin{aligned}
			0\le\mathcal E
			&=
			\sigma_N\sum_{i,j}\beta_i\beta_j(A_{ij}-B_{ij})\\
			&\le
			C\varepsilon_N(\lambda_n)
			\sigma_N\sum_{i,j}\beta_i\beta_jA_{ij}
			=Cn\varepsilon_N(\lambda_n).
		\end{aligned}
		\]
		At the matched scale,
		\[
		n\varepsilon_4(\lambda_n)
		=O_A\left(n^{-1}(1+\log n)^2\right),
		\qquad
		n\varepsilon_5(\lambda_n)
		=O_A\left(n^{-1/3}(1+\log n)\right).
		\]
		Both expressions tend to zero.  Equations \eqref{eq5.energy-identity} and \eqref{eq5.total-gap} therefore give
		\[
		\sigma_N\int_\Omega V^p\,dz\ge n+\frac12c_A
		\]
		for \(n\ge n_A\).  Finally,
		\[
		\frac{\mathfrak q_a(S)}{b_n}
		=
		\frac{\mathfrak q_a(V)}{b_n}
		=
		\frac{n}{\sigma_N\int_\Omega V^p\,dz}
		\le
		1-\frac{\kappa_A}{n}
		\]
		after decreasing \(\kappa_A\).  This proves \eqref{eq5.matched-drop}.
	\end{proof}
	
	\begin{corollary}\label{Cor5.11}
		Assume that the \(L_a\)-problem has no positive solution. Let \(A>0\).  For every sufficiently large \(m\ge n_A\), put \(\lambda_m=A m^{2/(N-2)}\).  For every \(\Lambda\ge\lambda_m\), the scale homotopies
		\[
		H_j(s,\mu)
		=g_{a,(1-s)\lambda_m+s\Lambda,j}(\mu),
		\qquad
		0\le s\le1,\quad 1\le j\le m,
		\]
		have the uniform upper bound
		\[
		\sup_{\substack{1\le j\le m\\0\le s\le1\\\mu\in B_j(K)}}
		\frac{\mathcal J_a(H_j(s,\mu))}{b_j}
		\le1+C\varepsilon_N(\lambda_m).
		\]
		Consequently, one can choose regular levels \(b_j<c_j<b_{j+1}\) so that every \(H_j\) is a homotopy in \(W_{a,j}\), and its restriction to \(B_{j-1}(K)\) is a homotopy in \(W_{a,j-1}\).  At the top level, \(c_{m-1}\) may simultaneously be chosen so that
		\[
		\sup_{\mu\in B_m(K)}
		\mathcal J_a(g_{a,\lambda_m,m}(\mu))
		<c_{m-1}<b_m.
		\]
	\end{corollary}
	
	\begin{proof}
		Repeat the normalization used in the proof of Proposition~\ref{Prop5.10}, now with \(j\) in place of \(n\). Dropping the nonnegative defect from \eqref{eq5.energy-identity} and using \eqref{eq5.source-transfer} gives
		\[
		\mathcal J_a(g_{a,\lambda,j}(\mu))
		\le b_j\bigl(1+C\varepsilon_N(\lambda)\bigr)
		\]
		uniformly in \(j\), the weights, and the centers.  The function \(\varepsilon_N\) is decreasing for large \(\lambda\).  Moreover,
		\[
		\varepsilon_N(\lambda_m)=o(m^{-1}),
		\qquad
		\frac{b_{j+1}}{b_j}-1
		=\left(1+\frac1j\right)^{\frac{2}{N-2}}-1
		\ge\frac{c_N}{j}
		\]
		uniformly for \(1\le j\le m\).  Thus the displayed upper bound lies below \(b_{j+1}\), and the corresponding \((j-1)\)-particle bound lies below \(b_j\), for all sufficiently large \(m\).  Proposition \ref{Prop5.10} leaves a nonempty interval below \(b_m\) for the additional top-level choice.  Avoiding critical values completes the selection.
	\end{proof}
	
	\subsection{The common fixed-stratum interface}
	
	For fixed \(j\), we use the parameter chart and balanced tubes of Definition~\ref{Def3.21}.  The global modulation construction below refines them to independently chosen inner and outer retained sets inside the same chart \(\Phi_{a,j}\).  These descriptions consist of exactly \(j\) positive, balanced, mutually orthogonal bubbles modulo permutations.
	
	\begin{lemma}\label{Lem5.12}
		Fix \(J_{\max}\in\mathbb N\).  For every choice of regular levels
		\[
		c_0<b_1,
		\qquad
		b_j<c_j<b_{j+1},
		\qquad 1\le j\le J_{\max},
		\]
		with \(W_{a,0}=\varnothing\), there is
		\[
		\Lambda=\Lambda(J_{\max},K,c_0,\ldots,c_{J_{\max}})>0
		\]
		such that, for every \(\lambda\ge\Lambda\), the maps
		\[
		g_{a,\lambda,j}:
		(B_j(K),B_{j-1}(K))
		\to
		(W_{a,j},W_{a,j-1}),
		\qquad 1\le j\le J_{\max},
		\]
		are continuous maps of pairs.
	\end{lemma}
	
	\begin{proof}
		For each fixed \(j\), the Bahri-Coron convexity estimate gives
		\[
		\sup_{\mu\in B_j(K)}
		\mathfrak q_a(S_{a,\lambda,j}(\mu))
		\le b_j+o_j(1).
		\]
		Indeed, this is \cite[Corollary~B.3]{BahriCoron1988}, applied at the fixed stratum \(j\), together with the fixed-\(j\) \(P_a\)-projection estimates of Section~\ref{sec:test}.  On \(B_{j-1}(K)\), the same sum has at most \(j-1\) active atoms, and its quotient is at most \(b_{j-1}+o_j(1)\).  Because only finitely many strata \(1\le j\le J_{\max}\) occur, one common large \(\lambda\) gives the two required strict inequalities for every \(j\).  Continuity follows from the weak topology of finite atomic measures and the continuity of \(x\mapsto P_a\delta_{\lambda,x}\).
	\end{proof}
	
	\begin{lemma}
		\label{Lem5.13}
		Fix \(j\ge1\) and \(K\Subset\Omega\).  There are \(C_{j,K}>0\) and functions
		\[
		\varrho_{j,K}(R)\to0\quad(R\to\infty),
		\qquad
		r_{j,K}(\lambda)\to0\quad(\lambda\to\infty)
		\]
		such that the following holds.  If
		\[
		\mu=\sum_{i=1}^j t_i\delta_{z_i}\in B_j(K),\qquad
		\max_i|jt_i-1|\le\tau\le\frac12,
		\qquad
		\lambda|z_i-z_k|\ge R\quad(i\ne k),
		\]
		and \(v=g_{a,\lambda,j}(\mu)\), then
		\begin{equation}\label{eq5.full-gradient}
			\left|\mathcal J_a(v)-b_j\right|
			+
			\left\|\nabla_{\Sigma_a}\mathcal J_a(v)\right\|_a
			\le
			C_{j,K}\tau+
			\varrho_{j,K}(R)+r_{j,K}(\lambda).
		\end{equation}
	\end{lemma}
	
	\begin{proof}
		Put
		\[
		Q_i=P_a\delta_{\lambda,z_i},\qquad
		S=\sum_i t_iQ_i,\qquad
		A=\|S\|_a,\qquad
		D=\int_\Omega S^p\,dx,\qquad
		q=\frac{A^2}{D}.
		\]
		The fixed-\(j\) overlap estimates, \(0<Q_i\le\delta_{\lambda,z_i}\), and the diagonal estimates in Lemma~\ref{Lem4.1} give, uniformly in the displayed region,
		\begin{equation}\label{eq5.full-gradient-overlaps}
			\begin{aligned}
				\left|A^2-\sum_i t_i^2\right|
				+\left|\sigma_ND-\sum_i t_i^p\right|
				&\le
				\varrho_{j,K}(R)+r_{j,K}(\lambda),\\
				\left\|S^{p-1}-\sum_i t_i^{p-1}Q_i^{p-1}
				\right\|_{L^{p/(p-1)}}
				&\le
				\varrho_{j,K}(R),\\
				\max_i
				\left\|\delta_{\lambda,z_i}^{p-1}-Q_i^{p-1}
				\right\|_{L^{p/(p-1)}}
				&\le r_{j,K}(\lambda).
			\end{aligned}
		\end{equation}
		Indeed, for fixed \(j\), the elementary two-profile overlap estimate for translated Talenti functions tends uniformly to zero when
		\[
		\lambda|z_i-z_k|\ge R\to\infty.
		\]
		After the common-scale change of variables, this is the translation continuity statement
		\[
		\left\|
		\left(\delta_{1,0}+\delta_{1,y}\right)^{p-1}
		-\delta_{1,0}^{p-1}-\delta_{1,y}^{p-1}
		\right\|_{L^{p/(p-1)}(\mathbb R^N)}
		\to0
		\qquad(|y|\to\infty).
		\]
		It follows from a two-core split and the Talenti tail.  The inequality
		\[
		\left|
		\left(\sum_i s_i\right)^{p-1}-\sum_i s_i^{p-1}
		\right|
		\le C_j\sum_{i\ne k}
		\left(s_i^{p-2}s_k+s_i s_k^{p-2}\right)
		\]
		reduces the fixed finite sum to the two-profile statement.  The last line of \eqref{eq5.full-gradient-overlaps} follows from the following explicit \(L^p\)-projection defect.  Since \(0\le Q_i\le\delta_{\lambda,z_i}\),
		\[
		\begin{aligned}
			\|\delta_{\lambda,z_i}-Q_i\|_p^p
			&\le
			\int_\Omega
			\left(\delta_{\lambda,z_i}^p-Q_i^p\right)\,dx\\
			&=O_{K}(\varpi_N(\lambda)),
		\end{aligned}
		\]
		where Lemma~\ref{Lem4.1} is used for \(Q_i\), while the omitted Talenti tail outside \(\Omega\) is \(O_K(\lambda^{-N})\). Now apply
		\[
		\|s^{p-1}-t^{p-1}\|_{L^{p/(p-1)}}
		\le
		C\left(\|s\|_p^{p-2}+\|t\|_p^{p-2}\right)\|s-t\|_p .
		\]
		
		Because the weights are balanced, \eqref{eq5.full-gradient-overlaps} bounds \(A,D,q\), and their reciprocals where needed, by fixed positive constants.  Since
		\[
		\frac{A^2}{D}
		=
		\sigma_N\frac{\sum_i t_i^2}{\sum_i t_i^p}
		+O_{j,K}\!\left(
		\varrho_{j,K}(R)+r_{j,K}(\lambda)\right)
		\]
		and the right-hand rational function equals \(\sigma_Nj^{p-2}\) at \(t_i=1/j\), finite-dimensional smoothness gives
		\begin{equation}\label{eq5.full-gradient-coefficients}
			\max_i|\sigma_Nt_i-qt_i^{p-1}|
			\le
			C_{j,K}\tau+
			\varrho_{j,K}(R)+r_{j,K}(\lambda).
		\end{equation}
		The same two scalar expansions, now applied to \(\mathcal J_a(S/A)=A^p/D\), prove the first term in \eqref{eq5.full-gradient}.
		
		For the full Hilbert-sphere gradient, use
		\[
		L_aS=\sigma_N\sum_i t_i\delta_{\lambda,z_i}^{p-1}
		\]
		and decompose
		\[
		\begin{aligned}
			L_aS-qS^{p-1}
			&=
			\sum_i(\sigma_Nt_i-qt_i^{p-1})
			\delta_{\lambda,z_i}^{p-1}\\
			&\quad
			+q\sum_i t_i^{p-1}
			\left(\delta_{\lambda,z_i}^{p-1}-Q_i^{p-1}\right)\\
			&\quad
			-q\left(S^{p-1}-\sum_i t_i^{p-1}Q_i^{p-1}\right).
		\end{aligned}
		\]
		Equations \eqref{eq5.full-gradient-overlaps}-%
		\eqref{eq5.full-gradient-coefficients}, the uniform \(L^{p/(p-1)}\)-norm of \(\delta_{\lambda,z_i}^{p-1}\), and the Sobolev dual embedding yield
		\begin{equation}\label{eq5.full-residual}
			\|L_aS-qS^{p-1}\|_{H^{-1}}
			\le
			C_{j,K}\tau+
			\varrho_{j,K}(R)+r_{j,K}(\lambda).
		\end{equation}
		Finally, for \(v=S/A\),
		\[
		L_av-\mathcal J_a(v)v^{p-1}
		=A^{-1}(L_aS-qS^{p-1}),
		\]
		and direct differentiation of the quotient on the Hilbert sphere gives
		\[
		\nabla_{\Sigma_a}\mathcal J_a(v)
		=
		p\mathcal J_a(v)L_a^{-1}
		\bigl(L_av-\mathcal J_a(v)v^{p-1}\bigr).
		\]
		The bounded inverse \(L_a^{-1}:H^{-1}\to H^1_0\), together with \eqref{eq5.full-residual}, proves the gradient part of \eqref{eq5.full-gradient}.
	\end{proof}
	
	\begin{lemma}\label{Lem5.14}
		Assume \(N\in\{4,5\}\). Fix \(j\ge2\), \(K\Subset\Omega\), \(\tau>0\), and \(R>0\). There are \(\eta_{j,K,\tau,R}>0\) and \(\Lambda_{j,K,\tau,R}>0\) such that, for every \(\lambda\ge\Lambda_{j,K,\tau,R}\) and every labelled tuple
		\[
		(\mathbf t,\mathbf x)\in\Delta_{j-1}\times K^j
		\]
		satisfying
		\[
		\max_i|jt_i-1|\ge\tau
		\quad\hbox{or}\quad
		\lambda|x_r-x_s|\le R
		\quad\hbox{for some }r\ne s,
		\]
		one has
		\[
		\mathfrak q_a\!\left(
		\sum_{i=1}^j t_iQ_{\lambda,x_i}
		\right)
		\le b_j-\eta_{j,K,\tau,R}.
		\]
		The conclusion includes zero weights and coincident centers.
	\end{lemma}
	
	\begin{proof}
		Normalize
		\[
		V=\frac{\sqrt j\sum_i t_iQ_{\lambda,x_i}}
		{\|\sum_i t_iQ_{\lambda,x_i}\|_a}
		=\sum_i\beta_iQ_i.
		\]
		The identity \eqref{eq5.energy-identity}, with \(j\) in place of \(n\), applies after zero weights are deleted.  Its source error satisfies
		\[
		0\le\mathcal E\le C_j\varepsilon_N(\lambda)
		\]
		by Lemma~\ref{Lem5.8}.  If fewer than \(j\) weights are positive, then
		\[
		\mathcal Q_j
		=
		(\gamma-1)\left(
		j-\sigma_N\sum_{t_i>0}\int_\Omega Q_i^p\,dz
		\right)
		\ge\gamma-1.
		\]
		If all weights are positive but some \(\beta_i\notin[1/2,2]\), Lemma~\ref{Lem5.7}(ii), together with the single-bubble mass lower bound in Lemma~\ref{Lem4.1}, gives \(\mathcal R\ge c_{j,K}>0\). If all \(\beta_i\in[1/2,2]\) and \(\lambda|x_r-x_s|\le R\) for some pair, then Lemma~\ref{Lem5.7}(i) and \eqref{eq5.pair-lower} give
		\[
		\mathcal R
		\ge c_{j,K}(1+R)^{-(N-2)}.
		\]
		In each of these cases, the exact identity gives a fixed decrease below \(b_j\) once \(\lambda\) is large.
		
		It remains to consider the weight collar when all coefficients stay in \([1/2,2]\).  If the asserted uniform decrease failed, the fixed-\(j\) upper estimate used in the proof of Lemma~\ref{Lem5.12} would give a sequence \(\lambda_k\to\infty\) of such tuples for which the quotient tends to \(b_j\).  The exact energy identity and \(\mathcal E=O(\varepsilon_N(\lambda_k))\) then imply \(\mathcal R_k+\mathcal Q_{j,k}\to0\). Lemma~\ref{Lem5.7}(i) and \eqref{eq5.pair-lower} force
		\[
		\lambda_k|x_{i,k}-x_{\ell,k}|\to\infty
		\qquad(i\ne\ell).
		\]
		After passing to a subsequence, \(\mathbf t_k\to\mathbf t\in\Delta_{j-1}\).  The fixed-\(j\) projection and overlap estimates then give
		\[
		\mathfrak q_a\!\left(\sum_i t_{i,k}Q_{\lambda_k,x_{i,k}}\right)
		\to
		\sigma_N
		\frac{\bigl(\sum_i t_i^2\bigr)^{p/2}}
		{\sum_i t_i^p}.
		\]
		The power-mean inequality shows that the expression on the right is at most \(b_j\), with equality only at \(\mathbf t=(1/j,\ldots,1/j)\).  The closed condition \(\max_i|jt_i-1|\ge\tau\) excludes that point, so compactness of the simplex gives a fixed strict gap.  This contradiction proves the lemma.
	\end{proof}
	
	\begin{lemma}\label{Lem5.15}
		Fix \(J_{\max}\in\mathbb N\).  There are
		\[
		0<2\rho_*<1,\qquad L_*>2,\qquad
		0<2\zeta_*<2^{-\frac{N-2}{2}},
		\]
		such that the following assertions hold simultaneously for \(1\le j\le J_{\max}\).  Let
		\[
		\mathscr Q
		=
		\overline{\mathscr Q}_j(\rho_*,L_*,\zeta_*).
		\]
		For a labelled tangent vector write
		\[
		h=(\dot\alpha_i,s_i,y_i)_{i=1}^j,\qquad
		\sum_i\dot\alpha_i=0,\qquad
		s_i=\frac{\dot\lambda_i}{\lambda_i},\qquad
		y_i=\lambda_i\dot x_i,
		\]
		and put
		\[
		|h|_*^2=\sum_i\bigl(\dot\alpha_i^2+s_i^2+|y_i|^2\bigr).
		\]
		For Taylor estimates at a base tuple \(q^0\), use the integrable coordinates
		\[
		\beta_i=\alpha_i-\alpha_i^0,\qquad
		\sigma_i=\log\frac{\lambda_i}{\lambda_i^0},\qquad
		\xi_i=\lambda_i^0(x_i-x_i^0),\qquad
		\sum_i\beta_i=0.
		\]
		On a fixed small chart their coordinate frame and its first derivative are uniformly equivalent to the moving normalized frame \((\dot\alpha_i,\lambda_i\partial_{\lambda_i}, \lambda_i^{-1}\partial_{x_i})\). There are \(c_*,C_*,r_*>0\), independent of \(q\in\mathscr Q\), such that
		\begin{equation}\label{eq5.modulation-Gram}
			c_*|h|_*^2
			\le
			\|D\Phi_{a,j}(q)h\|_a^2
			\le
			C_*|h|_*^2,
			\qquad
			\|D^2\Phi_{a,j}(q)[h,k]\|_a
			\le C_*|h|_*|k|_* .
		\end{equation}
		The second derivative is computed in these integrable base coordinates.
		
		For two labelled tuples define the parameter discrepancy
		\[
		d_*(q,q')
		=
		\max_i\left\{
		j|\alpha_i-\alpha_i'|,
		\left|\log\frac{\lambda_i}{\lambda_i'}\right|,
		\sqrt{\lambda_i\lambda_i'}\,|x_i-x_i'|
		\right\},
		\]
		and set
		\[
		d_{\mathfrak S}(q,q')
		=
		\min_{\pi\in\mathfrak S_j}d_*(q,\pi q').
		\]
		After decreasing \(r_*\), one has the local inverse estimate
		\begin{equation}\label{eq5.local-inverse}
			d_{\mathfrak S}(q,q')
			\le C_*\|\Phi_{a,j}(q)-\Phi_{a,j}(q')\|_a
			\quad\hbox{if}\quad
			d_{\mathfrak S}(q,q')<r_*.
		\end{equation}
		Moreover, if \(q_k,q_k'\in\mathscr Q\) and
		\[
		\|\Phi_{a,j}(q_k)-\Phi_{a,j}(q_k')\|_a\to0,
		\]
		then, after one subsequence and one permutation \(\pi\in\mathfrak S_j\),
		\begin{equation}\label{eq5.profile-matching}
			\max_i\left\{
			|\alpha_{i,k}-\alpha'_{\pi(i),k}|,
			\left|\log\frac{\lambda_{i,k}}
			{\lambda'_{\pi(i),k}}\right|,
			\sqrt{\lambda_{i,k}\lambda'_{\pi(i),k}}\,
			|x_{i,k}-x'_{\pi(i),k}|
			\right\}\to0.
		\end{equation}
		Consequently, for every \(0<r\le r_*\),
		\begin{equation}\label{eq5.off-orbit-gap}
			\kappa_j(r)
			:=
			\inf_{\substack{q,q'\in\mathscr Q\\
					d_{\mathfrak S}(q,q')\ge r}}
			\|\Phi_{a,j}(q)-\Phi_{a,j}(q')\|_a
			>0.
		\end{equation}
	\end{lemma}
	
	\begin{proof}
		We give the tangent and matching arguments because their uniformity at the noncompact scale ends is what is needed below.  For a whole space normalized bubble \(B_{\lambda,x}\), put
		\[
		Z_0=\lambda\partial_\lambda B_{\lambda,x},
		\qquad
		Z_\ell=\lambda^{-1}\partial_{x_\ell}B_{\lambda,x}.
		\]
		Scale and translation invariance and parity give constants \(\kappa_0,\kappa_1>0\) such that
		\[
		\begin{gathered}
			\langle B,Z_\nu\rangle_{\dot H^1}=0,\qquad
			\langle Z_0,Z_\ell\rangle_{\dot H^1}=0,\\
			\langle Z_\ell,Z_m\rangle_{\dot H^1}
			=\kappa_1\delta_{\ell m},\qquad
			\|Z_0\|_{\dot H^1}^2=\kappa_0.
		\end{gathered}
		\]
		For mutually orthogonal bubbles this is the diagonal scale-centre block.  The coefficient block is the differential of
		\[
		(\alpha_i)\longmapsto
		\frac{\sum_i\alpha_iB_i}{\|\sum_i\alpha_iB_i\|_{\dot H^1}}.
		\]
		Its only kernel in \(\mathbb R^j\) is the common scaling direction \(\mathbb R(\alpha_i)\).  Since \(\sum_i\dot\alpha_i=0\), this kernel is removed, and since the retained weights lie in a compact subset of the positive simplex, the angle between \(\{\sum_i\dot\alpha_i=0\}\) and \(\mathbb R(\alpha_i)\) has a uniform positive lower bound.
		
		The differentiated two-bubble overlap estimates tend to zero with the interaction parameter.  The cutoff argument in the first paragraph of the proof of Lemma~\ref{Lem4.3}, applied also to \(Z_\nu\) and their normalized first derivatives, transfers the preceding block matrix from whole space bubbles to the ordinary Dirichlet projections when \(L_*\) is large. Lemma~\ref{Lem4.3} writes its coefficient chart on \(\sum_i\alpha_i^2=1\), whereas the present chart uses \(\sum_i\alpha_i=1\).  The two charts are compatible because \(\Phi_{a,j}\) is homogeneous of degree zero in the coefficients, and
		\[
		(\alpha_i)\longmapsto
		\frac{(\alpha_i)}{(\sum_i\alpha_i^2)^{1/2}},
		\qquad
		(\beta_i)\longmapsto
		\frac{(\beta_i)}{\sum_i\beta_i},
		\]
		are inverse \(C^2\) coordinate changes on the retained positive balanced bands, with both derivatives and inverse derivatives uniformly bounded. Thus the coefficient derivatives in the two slices are uniformly equivalent through order two. Equations \eqref{eq4.7} and \eqref{eq4.7-dirichlet} then transfer it, through order two and with the present \(a\)-norm and normalization, to the \(P_a\)-family.  Choose first \(\rho_*\) small enough to keep the coefficient angle, next \(\zeta_*\) small enough that the sum of all off-diagonal blocks is less than one quarter of the least diagonal eigenvalue, and finally \(L_*\) large enough that the two projection errors have the same property. Gershgorin's estimate gives the first inequality in \eqref{eq5.modulation-Gram}; the upper and second-derivative estimates follow from the same normalized derivative bounds.  These choices are simultaneous for the finite range \(j\le J_{\max}\).
		
		The normalized fields do not commute, which is why Taylor's formula is applied in the integrable base chart above.  There
		\[
		\partial_{\sigma_i}=\lambda_i\partial_{\lambda_i},
		\qquad
		\partial_{\xi_{i,\ell}}
		=
		\frac{\lambda_i}{\lambda_i^0}
		\lambda_i^{-1}\partial_{x_{i,\ell}}.
		\]
		On \(|\sigma_i|<r_*\), the factors \(\lambda_i/\lambda_i^0=e^{\sigma_i}\) and their first derivatives are uniformly bounded above and below.  Thus \eqref{eq5.modulation-Gram}, the bound for the compositions of normalized fields in Lemma~\ref{Lem4.3}, and the first-derivative bound imply the same Gram and second-derivative bounds in the base chart.  Taylor's formula there gives
		\[
		\|\Phi_{a,j}(q)-\Phi_{a,j}(q')\|_a
		\ge \sqrt{c_*}\,|h|_*-C_*|h|_*^2.
		\]
		If \(r_*\) is sufficiently small, the straight coordinate segment stays in the parameter set with the relaxed bounds \((2\rho_*,L_*/2,2\zeta_*)\), on which the same Gram estimates hold after one harmless change of constants.  This proves \eqref{eq5.local-inverse}. At this point shrink the retained coefficient width once more so that
		\[
		2\rho_*<\frac{r_*}{4}.
		\]
		The smaller retained set inherits all preceding constants.  Since \(|j\alpha_i-1|\le\rho_*\), any two retained coefficient vectors then have \(j|\alpha_i-\alpha_i'|<r_*/4\); thus the coarse profile matching below controls the full discrepancy \(d_*\), including the coefficient part.
		
		It remains to prove that two close images enter one of these local charts. For a single pair of normalized whole space bubbles, its scalar product is a continuous function of
		\[
		\left|\log\frac\lambda\mu\right|
		\quad\hbox{and}\quad
		\sqrt{\lambda\mu}\,|x-y|,
		\]
		equals one only at \((\lambda,x)=(\mu,y)\), and is bounded by \(1-\nu(r)\), with \(\nu(r)>0\), when either displayed parameter is at least \(r\).  This follows directly after the change of variables \(z=\lambda(\xi-x)\): equality in Cauchy-Schwarz forces the two explicit Talenti functions to agree, and compactness on bounded relative-parameter annuli, together with weak convergence at their ends, gives the uniform gap.  The same statement, with an arbitrarily small error, holds for normalized \(P_a\)-profiles by the two projection comparisons just used.
		
		Now expand the square of the difference of two normalized balanced sums. Write the sums in normalized profiles as
		\[
		\Phi(q)=\sum_i c_iB_i,\qquad
		\Phi(q')=\sum_\ell d_\ell C_\ell .
		\]
		The retained coefficient band and the profile Gram lower bound give
		\[
		c_i,d_\ell\ge a_j>0,
		\qquad
		\sum_i c_i^2,\ \sum_\ell d_\ell^2
		=1+O(\omega_*),
		\]
		where \(\omega_*\to0\) as \(\rho_*,\zeta_*\to0\) and \(L_*\to\infty\) in the stated order.  Fix a small overlap threshold \(\eta>0\).  If \(\langle B_i,C_\ell\rangle_a\ge\eta\), the one-bubble relative parameters remain in a bounded set depending only on \(\eta\).  Two such \(C_\ell\)'s for the same \(B_i\) would therefore have interaction bounded below by a positive number depending only on \(\eta\).  Choosing \(\zeta_*\) below that number, and \(L_*\) so that the projection error is smaller still, shows that each row and each column of the cross-Gram matrix has at most one entry larger than \(\eta\).
		
		Since \(\|\Phi(q_k)-\Phi(q_k')\|_a\to0\), all rows and columns must have such an entry: otherwise testing the difference against the unmatched profile leaves at least \(a_j-C_j(\eta+\omega_*)>0\).  They consequently define a permutation \(\pi_k\), and
		\begin{equation}\label{eq5.matching-matrix}
			1-o(1)
			=
			\langle\Phi(q_k),\Phi(q_k')\rangle_a
			\le
			\sum_i c_{i,k}d_{\pi_k(i),k}
			\langle B_{i,k},C_{\pi_k(i),k}\rangle_a
			+C_j(\eta+\omega_*).
		\end{equation}
		Cauchy-Schwarz bounds the coefficient sum by \(1+C_j\omega_*\). If one matched profile had parameter discrepancy at least \(r_*/4\), the one-bubble gap would lower the right-hand side of \eqref{eq5.matching-matrix} by at least \(a_j^2\nu(r_*/4)\).  Choose \(\eta,\rho_*,\zeta_*\), and then \(L_*\), so that \(C_j(\eta+\omega_*)<a_j^2\nu(r_*/4)/2\).  This is a contradiction. Thus the matched scale-centre discrepancies are below \(r_*/4\); the coefficient shrinkage above puts the full paired tuples in a common \(r_*\)-chart.  Since the permutation group is finite, pass to one subsequence on which \(\pi_k=\pi\) is constant.  Applying \eqref{eq5.local-inverse} now proves the logarithmic-scale and scale-normalized-centre convergence in \eqref{eq5.profile-matching}. For the coefficient conclusion, write
		\[
		B_{i,k}
		=
		\frac{P_a\delta_{\lambda_{i,k},x_{i,k}}}
		{\|P_a\delta_{\lambda_{i,k},x_{i,k}}\|_a},
		\qquad
		\Phi_{a,j}(q_k)=\sum_i c_{i,k}B_{i,k}.
		\]
		The matched profile convergence replaces \(B'_{\pi(i),k}\) by \(B_{i,k}\) with an \(o(1)\) error.  The profile Gram matrix has a uniform positive lower eigenvalue, and hence convergence of the two normalized sums gives
		\[
		c_{i,k}-c'_{\pi(i),k}\to0.
		\]
		If \(A_k,A'_k\) are the norms of the two unnormalized sums, the matched single-profile norms have ratio tending to one, so
		\[
		\frac{\alpha_{i,k}}{A_k}
		-
		\frac{\alpha'_{\pi(i),k}}{A'_k}
		\to0.
		\]
		Multiplication by \(A_k\), summation in \(i\), and \(\sum_i\alpha_{i,k}=\sum_i\alpha'_{i,k}=1\) give \(A_k/A'_k\to1\), and then \(\alpha_{i,k}-\alpha'_{\pi(i),k}\to0\).  This completes \eqref{eq5.profile-matching}.
		
		The argument also covers repeated physical centres.  If \(x_{i,k}=x_{\ell,k}\), the interaction cutoff separates their scales, and the matching pairs the two logarithmic scales individually.  Finally, if \eqref{eq5.off-orbit-gap} failed, its minimizing sequences would contradict \eqref{eq5.profile-matching}.  This proves the lemma.
	\end{proof}
	
	For each \(j\le J_{\max}\), choose the retained constants independently and in the order
	\begin{equation}\label{eq5.nested-modulation-widths}
		\begin{gathered}
			0<\tau_{j,\mathrm{in}}<\tau_{j,\mathrm{out}}
			<\min\{\tfrac12,\rho_*\},\\
			0<\zeta_{j,\mathrm{in}}<\zeta_{j,\mathrm{out}}<\zeta_*,
			\qquad
			L_*<L_{j,\mathrm{out}}<L_{j,\mathrm{in}} .
		\end{gathered}
	\end{equation}
	Put
	\[
	\mathcal Q_j^{\mathrm{in}}
	=
	\mathscr Q_j(\tau_{j,\mathrm{in}},
	L_{j,\mathrm{in}},\zeta_{j,\mathrm{in}}),
	\qquad
	\overline{\mathcal Q}_j^{\mathrm{out}}
	=
	\overline{\mathscr Q}_j(\tau_{j,\mathrm{out}},
	L_{j,\mathrm{out}},\zeta_{j,\mathrm{out}}).
	\]
	Thus the first is uniformly separated from every finite face of the second. For a residual radius \(\rho>0\), define the open modulation tube
	\begin{equation}\label{eq5.modulation-domain}
		\mathcal U_{a,j}(\rho)
		=
		\left\{
		v\in\Sigma_{a,+}:
		\inf_{q\in\mathcal Q_j^{\mathrm{in}}}
		\|v-\Phi_{a,j}(q)\|_a<\rho
		\right\}.
	\end{equation}
	For \(v\in\Sigma_{a,+}\) and \(q\in\overline{\mathcal Q}_j^{\mathrm{out}}\), set
	\[
	E_v(q)=\frac12\|v-\Phi_{a,j}(q)\|_a^2.
	\]
	A minimizing \(\mathfrak S_j\)-orbit on this closed outer parameter set is called the set of modulation parameters of \(v\).
	
	\begin{lemma}\label{Lem5.16}
		Fix \(J_{\max}\in\mathbb N\).  The constants in \eqref{eq5.nested-modulation-widths} admit radii \(\rho_{j,\mathrm{mod}}>0\) such that every \(v\in\mathcal U_{a,j}(\rho_{j,\mathrm{mod}})\) has exactly one minimizing modulation orbit.  The minimum is attained in the smooth interior \(\mathcal Q_j^{\mathrm{out}}\) of the closed outer set.  In the residual window used in its proof this is also the only critical orbit of \(E_v\). The labelled parameters depend \(C^1\)-smoothly on \(v\) in every quotient chart, and forgetting amplitudes and scales defines a continuous map
		\[
		\chi_{a,j}:
		\mathcal U_{a,j}(\rho_{j,\mathrm{mod}})
		\to
		\operatorname{SP}^j(\Omega)=\Omega^j/\mathfrak S_j .
		\]
	\end{lemma}
	
	\begin{proof}
		Fix \(j\), and fix the constants \(c_*,C_*,r_*\) and the gap function \(\kappa_j\) supplied by Lemma~\ref{Lem5.15}.  Abbreviate the retained sets by \(\mathcal Q^{\rm in}\) and \(\overline{\mathcal Q}^{\rm out}\).  The fixed numerical gaps between their weight, scale, scale-boundary, and interaction inequalities give \(r_{\rm mar}>0\) such that
		\begin{equation}\label{eq5.parameter-margin}
			d_{\mathfrak S}(q,q')<8r_{\rm mar},\quad
			q\in\mathcal Q^{\rm in}
			\quad\to\quad
			q'\in\mathcal Q^{\rm out}.
		\end{equation}
		Indeed, \(d_*(q,q')<r\) gives
		\[
		e^{-r}<\frac{\lambda_i'}{\lambda_i}<e^r,
		\qquad
		\lambda_i'|x_i-x_i'|<e^{r/2}r.
		\]
		Consequently
		\[
		\begin{aligned}
			\lambda_i'&>e^{-r}L_{j,\mathrm{in}},\\
			\lambda_i'\operatorname{dist}(x_i',\partial\Omega)
			&>
			e^{-r}L_{j,\mathrm{in}}-e^{r/2}r .
		\end{aligned}
		\]
		The interaction expression has, in these same dimensionless coordinates, a uniform modulus \(\omega_{\rm int}(r)\to0\).  Thus \eqref{eq5.parameter-margin} follows by choosing \(8r_{\rm mar}\) so that
		\[
		\begin{gathered}
			\tau_{j,\mathrm{in}}+8r_{\rm mar}
			<\tau_{j,\mathrm{out}},\\
			e^{-8r_{\rm mar}}L_{j,\mathrm{in}}
			-8e^{4r_{\rm mar}}r_{\rm mar}
			>L_{j,\mathrm{out}},\\
			\zeta_{j,\mathrm{in}}+\omega_{\rm int}(8r_{\rm mar})
			<\zeta_{j,\mathrm{out}}.
		\end{gathered}
		\]
		Let \(C_{\rm loc}\ge1\) be the uniform local Lipschitz constant obtained from the upper Gram bound and the equivalence with the integrable base coordinates.  Take in addition
		\[
		r_{\rm mar}<
		\min\left\{\frac{r_*}{16},
		\frac{c_*}{32C_*C_{\rm loc}}\right\},
		\qquad
		\kappa_{\rm mar}=\kappa_j(r_{\rm mar})>0,
		\]
		decreasing \(r_{\rm mar}\) if necessary.  Finally choose \(\rho_{j,\mathrm{mod}}\) so small that
		\begin{equation}\label{eq5.modulation-smallness}
			4\rho_{j,\mathrm{mod}}<\kappa_{\rm mar},
			\qquad
			C_*\rho_{j,\mathrm{mod}}<\frac{c_*}{8}.
		\end{equation}
		The definition of \(C_{\rm loc}\) gives
		\begin{equation}\label{eq5.local-image-Lipschitz}
			\|\Phi_{a,j}(q)-\Phi_{a,j}(q^0)\|_a
			\le C_{\rm loc}|\theta(q)|_0,
		\end{equation}
		where
		\[
		|\theta(q)|_0
		=
		\max_i\{j|\beta_i|,|\sigma_i|,|\xi_i|\}.
		\]
		The formulas relating \(d_*\) and \(|\theta|_0\), followed by one further harmless decrease of \(r_{\rm mar}\), show that the full coordinate ball
		\[
		\mathbb B_0(4r_{\rm mar})
		=\{q:|\theta(q)|_0<4r_{\rm mar}\}
		\]
		has \(d_*(q,q^0)<8r_{\rm mar}\) and hence lies in the relaxed retained set \(\mathcal Q^{\rm out}\).  It is convex in the integrable coordinates, so the straight segment between any two of its points remains in the same retained ball. In particular the finite outer faces have the explicit image gap
		\begin{equation}\label{eq5.inner-face-gap}
			g_{j,\mathrm{face}}
			:=
			\inf_{\substack{q\in\mathcal Q^{\rm in}\\
					q'\in\partial\overline{\mathcal Q}^{\rm out}}}
			\|\Phi_{a,j}(q)-\Phi_{a,j}(q')\|_a
			\ge\kappa_j(8r_{\rm mar})>0.
		\end{equation}
		
		Let \(v\) belong to the inner tube and choose a witness \(q^0\in\mathcal Q^{\rm in}\) with
		\[
		\|v-\Phi_{a,j}(q^0)\|_a
		<\rho_{j,\mathrm{mod}}.
		\]
		If \((q^k)\subset\overline{\mathcal Q}^{\rm out}\) is a minimizing sequence, the witness gives, for large \(k\),
		\[
		\|v-\Phi_{a,j}(q^k)\|_a<2\rho_{j,\mathrm{mod}},
		\qquad
		\|\Phi_{a,j}(q^k)-\Phi_{a,j}(q^0)\|_a
		<3\rho_{j,\mathrm{mod}}.
		\]
		The off-orbit gap and \eqref{eq5.modulation-smallness} imply
		\[
		d_{\mathfrak S}(q^k,q^0)<r_{\rm mar}.
		\]
		Choose the unique matching labels.  For the fixed tuple \(q^0\), bounded coefficient, logarithmic-scale, and scale-normalized-centre coordinates form a compact finite-dimensional set.  Hence a subsequence of \(q^k\) converges, and \eqref{eq5.parameter-margin} puts its limit in the interior of \(\mathcal Q^{\rm out}\).  Thus \(E_v\) attains its global minimum.
		
		This argument treats every parameter end directly.  The weak outer inequalities control coefficient zero and \(\lambda\downarrow0\), while \eqref{eq5.parameter-margin} controls the scale-boundary and interaction-loss faces.  Matching to the fixed witness bounds \(\log(\lambda_i/\lambda_i^0)\), and the matching labels resolve the finite permutation ambiguity.  Repeated physical centres are handled by separate matching of their orthogonal scales through \eqref{eq5.profile-matching}.
		
		Every minimizer lies in \(\mathbb B_0(2r_{\rm mar})\), because
		\[
		|\xi_i|
		\le e^{r_{\rm mar}/2}
		\sqrt{\lambda_i\lambda_i^0}\,|x_i-x_i^0|
		<2r_{\rm mar}
		\]
		after the fixed decrease of \(r_{\rm mar}\); the coefficient and log-scale coordinates satisfy the same bound directly.  On the whole larger ball \(\mathbb B_0(4r_{\rm mar})\), \eqref{eq5.local-image-Lipschitz} gives the explicit residual estimate
		\begin{equation}\label{eq5.whole-chart-residual}
			\|v-\Phi_{a,j}(q)\|_a
			\le
			\rho_{j,\mathrm{mod}}+4C_{\rm loc}r_{\rm mar}.
		\end{equation}
		The choices
		\[
		C_*\rho_{j,\mathrm{mod}}<\frac{c_*}{8},
		\qquad
		4C_*C_{\rm loc}r_{\rm mar}<\frac{c_*}{8}
		\]
		therefore show, at every point of every straight segment joining two candidate minimizers, that
		\[
		D^2E_v(q)[h,h]
		=
		\|D\Phi_{a,j}(q)h\|_a^2
		-
		\langle v-\Phi_{a,j}(q),
		D^2\Phi_{a,j}(q)[h,h]\rangle_a
		\ge\frac{c_*}{2}|h|_*^2
		\]
		by \eqref{eq5.modulation-Gram} and \eqref{eq5.whole-chart-residual}.  Hence \(E_v\) is strictly convex on the whole segment.  Every global minimizer lies in that chart, so there is exactly one minimizing \(\mathfrak S_j\)-orbit.  More generally, choose \(r_{\rm crit}>0\) so small that
		\[
		\rho_{j,\mathrm{mod}}+r_{\rm crit}
		<\kappa_{\rm mar}.
		\]
		If \(\|v-\Phi_{a,j}(q)\|_a<r_{\rm crit}\), the inner witness and \eqref{eq5.off-orbit-gap} force \(d_{\mathfrak S}(q,q^0)<r_{\rm mar}\), hence \(q\in\mathbb B_0(2r_{\rm mar})\).  Thus every low-residual critical point and every minimizer lie in the same convex base chart, and strict convexity makes the minimizer the unique critical orbit in the low-residual window used below.
		
		At the critical point the derivative, in the \(q\)-variables, of the orthogonality system
		\[
		D_q\Phi_{a,j}(q)^*
		\bigl(v-\Phi_{a,j}(q)\bigr)=0
		\]
		is the negative of the uniformly positive Hessian just displayed, hence is uniformly negative definite and invertible.  The implicit function theorem gives \(C^1\) dependence of the labelled tuple on \(v\).  The construction is equivariant under \(\mathfrak S_j\).  On an overlap, the unique matching with the two base tuples differs by one constant permutation; uniqueness of the critical orbit makes the two \(C^1\) solutions identical after that permutation.  Thus the modulation orbit is independent of the auxiliary outer margins whenever the inner tube is fixed. Therefore
		\[
		\chi_{a,j}(v)=[x_1(v),\ldots,x_j(v)]
		\]
		is well defined and continuous into the full symmetric product.  In particular it accommodates coincident centres carried by orthogonal scales.
	\end{proof}
	
	Assume that the \(L_a\)-problem has no positive solution and fix all modulation constants from Lemma~\ref{Lem5.16}.  Choose
	\begin{equation}\label{eq5.excision-width}
		0<\varepsilon_{j,\mathrm{exc}}
		<
		\frac14\min\left\{
		\tau_{j,\mathrm{in}},\zeta_{j,\mathrm{in}},
		L_{j,\mathrm{in}}^{-1},\rho_{j,\mathrm{mod}}
		\right\}.
	\end{equation}
	Then
	\[
	\mathcal T_{a,j}(\varepsilon_{j,\mathrm{exc}})
	\subset
	\mathcal U_{a,j}(\rho_{j,\mathrm{mod}}).
	\]
	Indeed, the weak \(\varepsilon_{j,\mathrm{exc}}\)-bounds lie strictly inside the retained inner chart by the numerical margin in \eqref{eq5.excision-width}, while the residual is strictly smaller than \(\rho_{j,\mathrm{mod}}\). Apply Lemma~\ref{Lem3.28} once with this excision width.  Retain the resulting level-uniform inner width, cutoffs, and modified functional \(\mathscr F_j\), and denote the energy-window half-width, lower gradient threshold, and usable depth by
	\[
	\eta_{j,\mathrm{win}},\qquad s_{0,j},\qquad\vartheta_{j,*}.
	\]
	For an admissible regular pair \(c_{j-1}<b_j<c_j\), define
	\[
	\mathcal F_{a,j}
	=
	\{v\in W_{a,j}:\mathscr F_j(v)\le c_{j-1}\}.
	\]
	The next two lemmas use this same deformation data.  Set
	\[
	X_j=\mathcal F_{a,j}\cap
	\mathcal U_{a,j}(\rho_{j,\mathrm{mod}}),
	\qquad
	A_j=W_{a,j-1}\cap
	\mathcal U_{a,j}(\rho_{j,\mathrm{mod}}),
	\]
	and denote the inclusion of pairs by
	\begin{equation}\label{eq5.core-pair-new}
		\iota_j:(X_j,A_j)\to(W_{a,j},W_{a,j-1}).
	\end{equation}
	
	\begin{lemma}\label{Lem5.17}
		If \(c_{j-1}\) is sufficiently close to \(b_j\), then \(\iota_j\) induces an isomorphism in homology with \(\mathbb F_2\)-coefficients.
	\end{lemma}
	
	\begin{proof}
		Use the common deformation data fixed above. Let \(\delta_j\) be the inner width supplied by Lemma~\ref{Lem3.28}.  Then
		\[
		\overline{\mathcal T_{a,j}(\delta_j)}
		\subset
		\mathcal T_{a,j}(\varepsilon_{j,\mathrm{exc}})
		\subset
		\mathcal U_{a,j}(\rho_{j,\mathrm{mod}}).
		\]
		Proposition~\ref{Prop3.29}, with \(\mathcal O_{a,j}=\mathcal U_{a,j}(\rho_{j,\mathrm{mod}})\), therefore gives \eqref{eq5.core-pair-new}.  The stopped deformation fixes \(W_{a,j-1}\), hence fixes \(A_j\), pointwise; this is the invariance property required here.
	\end{proof}
	
	To formulate the remaining analytic interface, let \(h:V\to K\) be a smooth map from a closed manifold.  For positive \(\tau_j\) and \(D_j\), call a labelled weighted tuple \((z,t)\in V^j\times\Delta_{j-1}\) a \emph{core tuple} if
	\[
	|jt_i-1|<\tau_j\ \text{for every }i,
	\qquad
	\lambda|h(z_r)-h(z_s)|>D_j\quad(r\ne s),
	\]
	and a \emph{collar tuple} if
	\[
	\max_i|jt_i-1|\ge\tau_j/2,
	\quad\text{or}\quad
	\lambda|h(z_r)-h(z_s)|\le2D_j\ \text{for some }r\ne s,
	\quad\text{or}\quad t_i=0.
	\]
	The same terminology is used on the labelled spherical Fulton-MacPherson resolution by replacing each \(z_i\) with the extended root evaluation \(\operatorname{ev}_i\).  These conditions are \(\mathfrak S_j\)-invariant and descend to the unlabelled quotient.  In particular, every zero-weight or intrinsic-collision face is a collar face. We write
	\[
	B_j(h)\left(\sum_i t_i\delta_{z_i}\right)
	=\sum_i t_i\delta_{h(z_i)},
	\qquad
	g_{a,\lambda,j,h}=g_{a,\lambda,j}\circ B_j(h).
	\]
	The map \(h\) may be noninjective; an equality \(h(z_r)=h(z_s)\) with \(z_r\ne z_s\) is then an analytic collision and belongs to the collar.
	
	\begin{lemma}\label{Lem5.18}
		Fix \(j\ge1\), assume that the \(L_a\)-problem has no positive solution, and fix the level-independent deformation data used in Lemma~\ref{Lem5.17}.  Let \(V\) be a closed manifold and let \(h:V\to K\Subset\Omega\) be smooth.  There is \(0<\tau_{j,*}<1/2\) such that, for every \(0<\tau_j\le\tau_{j,*}\), one can choose \(D_j>1\) and \(\Lambda_j>0\) so that, for every \(\lambda\ge\Lambda_j\), there is a number
		\[
		\max\{b_{j-1},b_j-\vartheta_{j,*}\}
		<\ell_j^{\mathrm{low}}(\lambda)<b_j
		\]
		with the following property.  For every pair of regular levels satisfying
		\[
		\ell_j^{\mathrm{low}}(\lambda)<c_{j-1}<b_j<c_j<b_{j+1},
		\]
		use the canonical pair
		\[
		X_j=\mathcal F_{a,j}\cap
		\mathcal U_{a,j}(\rho_{j,\mathrm{mod}}),
		\qquad
		A_j=W_{a,j-1}\cap
		\mathcal U_{a,j}(\rho_{j,\mathrm{mod}}).
		\]
		If
		\[
		g_{a,\lambda,j,h}(B_j(V))\subset W_{a,j},
		\]
		then the core part of the test family maps into \(X_j\), the collar part maps into \(W_{a,j-1}\), and the same statements hold for the lifted family on the labelled spherical Fulton-MacPherson resolution.  This includes all zero-weight and intrinsic-collision boundary faces; on a zero-weight face the restriction is exactly the corresponding \((j-1)\)-particle test family.  On the core
		\begin{equation}\label{eq5.exact-center-new}
			\chi_{a,j}\!\left(
			g_{a,\lambda,j,h}\!\left(\sum_i t_i\delta_{z_i}\right)
			\right)
			=[h(z_1),\ldots,h(z_j)]
			\in\operatorname{SP}^j(\Omega).
		\end{equation}
		The deformation, collar, and lower-level thresholds are independent of the particular regular levels.  For every fixed \(J_{\max}\), the choices can be made simultaneously for \(1\le j\le J_{\max}\).
	\end{lemma}
	
	\begin{proof}
		Use the deformation data and modified functional \(\mathscr F_j\) fixed above.  Let \(\ell_j^{\mathrm{exc}}<b_j\) be a level-independent lower threshold for the isomorphism in Lemma~\ref{Lem5.17}; it exists by the level-uniform window in Lemma~\ref{Lem3.28}.
		
		Consider first \(j=1\).  Take \(\tau_{1,*}=1/4\) and \(D_1=2\); the one-particle family has only the core stratum.  Put
		\[
		e_{1,\mathrm{core}}(\lambda)=r_{1,K}(\lambda).
		\]
		Increase \(\Lambda_1\) so that
		\[
		\lambda>L_{1,\mathrm{in}},\qquad
		\lambda\operatorname{dist}(K,\partial\Omega)>L_{1,\mathrm{in}},
		\]
		and so that Lemma~\ref{Lem5.13} gives
		\[
		e_{1,\mathrm{core}}(\lambda)
		<\min\{\eta_{1,\mathrm{win}},\sqrt{s_{0,1}},\vartheta_{1,*}\}.
		\]
		The one-bubble tuple then lies in \(\mathcal Q_1^{\mathrm{in}}\).  The cutoff in Lemma~\ref{Lem3.28} is fully active and
		\[
		\mathscr F_1(g_{a,\lambda,1,h}(\delta_z))
		\le b_1+e_{1,\mathrm{core}}(\lambda)-\vartheta_{1,*}.
		\]
		Choose \(\ell_1^{\mathrm{low}}(\lambda)<b_1\) above
		\[
		\max\{b_0,b_1-\vartheta_{1,*},\ell_1^{\mathrm{exc}},
		b_1+e_{1,\mathrm{core}}(\lambda)-\vartheta_{1,*}\}.
		\]
		The assumed upper inclusion then places the image in \(X_1\), and the uniqueness in Lemma~\ref{Lem5.16} gives the asserted centre identity.
		
		Suppose now that \(j\ge2\).  First choose
		\[
		0<\tau_{j,*}<\min\{1/2,\tau_{j,\mathrm{in}}\}
		\]
		so small that the \(C_{j,K}\tau_{j,*}\)-term in \eqref{eq5.full-gradient} is smaller than one fourth of each of \(\eta_{j,\mathrm{win}}\), \(\sqrt{s_{0,j}}\), and \(\vartheta_{j,*}\).  Fix \(0<\tau_j\le\tau_{j,*}\), and then choose \(D_j\) so large that the interaction term in \eqref{eq5.full-gradient} satisfies the same three bounds and
		\begin{equation}\label{eq5.core-retained-interaction}
			(2+D_j^2)^{-\frac{N-2}{2}}<\zeta_{j,\mathrm{in}}.
		\end{equation}
		Increase \(\Lambda_j\) so that
		\[
		\lambda>L_{j,\mathrm{in}},\qquad
		\lambda\operatorname{dist}(K,\partial\Omega)>L_{j,\mathrm{in}},
		\]
		and so that the scale remainder has the same three bounds.  Put
		\[
		e_{j,\mathrm{core}}(\lambda)
		=C_{j,K}\tau_j+\varrho_{j,K}(D_j)+r_{j,K}(\lambda).
		\]
		Then
		\[
		e_{j,\mathrm{core}}(\lambda)
		<\min\{\eta_{j,\mathrm{win}},\sqrt{s_{0,j}},\vartheta_{j,*}\},
		\]
		and the complete common-scale core lies in the energy window with
		\begin{equation}\label{eq5.core-full-gradient}
			\|\nabla_{\Sigma_a}\mathcal J_a\|_a^2<s_{0,j}.
		\end{equation}
		The weight, scale, boundary, and interaction inequalities, including \eqref{eq5.core-retained-interaction}, put every core tuple in \(\mathcal Q_j^{\mathrm{in}}\), hence its image lies in \(\mathcal U_{a,j}(\rho_{j,\mathrm{mod}})\) with zero residual.
		
		We now use the only dimension-dependent input in this common proof.  If \(N=3\), apply the final assertion of Theorem~\ref{Thm5.5} with weight threshold \(\tau_j/2\) and interaction threshold \((1+4D_j^2)^{-1/2}\).  If \(N\in\{4,5\}\), apply Lemma~\ref{Lem5.14} with \(\tau=\tau_j/2\) and \(R=2D_j\).  After increasing only \(\Lambda_j\), in either case there is \(\eta_{j,\mathrm{col}}>0\) such that
		\[
		\mathcal J_a(g_{a,\lambda,j,h}(\mu))
		\le b_j-2\eta_{j,\mathrm{col}}
		\]
		on every weight or physical-collision collar.  This includes off-diagonal self-collisions of \(h\).  On \(t_i=0\), the formula is exactly the \((j-1)\)-particle formula.  Set
		\[
		B_{j-1}^{\mathrm{bd}}(\lambda)
		=\sup_{\nu\in B_{j-1}(K)}
		\mathcal J_a(g_{a,\lambda,j-1}(\nu));
		\]
		the fixed-\((j-1)\) estimate gives \(B_{j-1}^{\mathrm{bd}}(\lambda)<b_j\) after a further increase of \(\Lambda_j\).
		
		Choose an equivariant subdivision subordinate to the overlapping constants \(\tau_j,\tau_j/2\) and \(D_j,2D_j\).  On each core simplex, global modulation uniqueness gives \eqref{eq5.exact-center-new}.  The full-gradient bound \eqref{eq5.core-full-gradient} makes the gradient cutoff attain depth \(\vartheta_{j,*}\).  Choose \(\ell_j^{\mathrm{low}}(\lambda)<b_j\) above
		\[
		\max\{b_{j-1},b_j-\vartheta_{j,*},\ell_j^{\mathrm{exc}},
		B_{j-1}^{\mathrm{bd}}(\lambda),b_j-\eta_{j,\mathrm{col}},
		b_j+e_{j,\mathrm{core}}(\lambda)-\vartheta_{j,*}\}.
		\]
		Every entry in this maximum is strictly below \(b_j\).  Hence every regular \(c_{j-1}\) in the stated interval has all required properties.  The central block lies in \(\mathcal F_{a,j}\) and hence in \(X_j\); its overlap with the collar lies in \(W_{a,j-1}\), and every collar simplex is zero in \(C_*(W_{a,j},W_{a,j-1})\).
	\end{proof}
	
	\subsection{Vanishing of the top test map at the topology scale}
	
	For \(j\ge1\), define
	\[
	T_j(\lambda)=
	\sup_{\mu\in B_j(K)}
	\mathcal J_a(g_{a,\lambda,j}(\mu)),
	\qquad T_0(\lambda)=-\infty,
	\]
	and, whenever \(\lambda_*\le\Lambda\), set
	\[
	\widehat T_j(\lambda_*,\Lambda)
	=\sup_{\lambda_*\le\lambda\le\Lambda}T_j(\lambda),
	\qquad \widehat T_0(\lambda_*,\Lambda)=-\infty.
	\]
	The scale \(\lambda_*\) is the branch-dependent drop scale, whereas \(\Lambda\) is the single scale at which the topological module will be applied.
	
	\begin{proposition}\label{Prop5.19}
		Let \(m\ge2\), \(\lambda_*\le\Lambda\), and assume that regular levels \(c_0,\ldots,c_m\) have been chosen so that, for \(1\le j\le m\),
		\[
		\widehat T_j(\lambda_*,\Lambda)<c_j,
		\qquad
		\widehat T_{j-1}(\lambda_*,\Lambda)<c_{j-1},
		\]
		and
		\[
		\max\{b_{m-1},T_m(\lambda_*)\}<c_{m-1}<b_m.
		\]
		Then
		\[
		H_j(s,\mu)
		=g_{a,(1-s)\lambda_*+s\Lambda,j}(\mu)
		\]
		is a homotopy of pairs
		\[
		(B_j(K),B_{j-1}(K))\to
		(W_{a,j},W_{a,j-1}).
		\]
		For every closed manifold \(V\) and every smooth map \(h:V\to K\),
		\[
		(g_{a,\Lambda,m,h})_*=0:
		H_*(B_m(V),B_{m-1}(V);\mathbb F_2)
		\to
		H_*(W_{a,m},W_{a,m-1};\mathbb F_2).
		\]
	\end{proposition}
	
	\begin{proof}
		The two cylinder bounds make every \(H_j\) a homotopy of pairs.  At the drop scale, the last strict inequality gives
		\[
		g_{a,\lambda_*,m}(B_m(K))\subset W_{a,m-1}.
		\]
		Hence the induced relative map at \(\lambda_*\) is zero.  The homotopy of pairs transports this zero homomorphism to the scale \(\Lambda\). Composition with \(B_m(h)\) is legitimate because its image lies in \(B_m(K)\), and the asserted vanishing follows.
	\end{proof}
	
	The two branches supply the drop-scale and scale-cylinder estimates needed in the proposition.  If \(N=3\), choose \(m\) by \eqref{eq5.mchoice} and take \(\lambda_*=\Lambda\) sufficiently large for Proposition~\ref{Prop5.6}; the scale homotopy is constant.  If \(N\in\{4,5\}\), fix \(A>0\), choose \(m\) sufficiently large, take
	\[
	\lambda_*=\lambda_m=A m^{\frac{2}{N-2}},
	\]
	use Proposition~\ref{Prop5.10} at \(\lambda_m\), and use Corollary~\ref{Cor5.11} to reach the common topology scale \(\Lambda\).  In the latter case the strict energy drop is used at \(\lambda_m\), and the induced relative map is transported to \(\Lambda\). Section~\ref{sec:topology-module} makes the single regular-level choice only after the common topology scale \(\Lambda\) and all thresholds \(\ell_j^{\mathrm{low}}(\Lambda)\), \(1\le j\le m\), have been fixed.
	\section{Topological module argument and proof of the main theorem}
	\label{sec:topology-module}

	Throughout this section all homology and cohomology groups have coefficients in \(\mathbb F_2\).  The use of a field removes orientation and torsion issues.  We use the convention \(B_\ell(X)=\varnothing\) for every integer \(\ell\le0\).  The construction starts on the regular collision-free model and uses a resolved compactification that retains all labels and limiting evaluation points on its boundary strata. These data define the transfer and center maps throughout the filtration. After each stated triangulation, every displayed cellular chain is represented by the corresponding singular chain.  For every continuous map \(f\), the symbol \(f_\#\) denotes its induced singular-chain map; cellular representatives and cap cocycles are always placed on the common subdivisions specified below.
	
	\subsection{Thom realization and the barycenter class}
	
	\begin{lemma}\label{Lem6.1}
		Let \(\Omega\) be a connected open subset of a smooth \(N\)-dimensional manifold and let
		\[
		0\ne\alpha\in H_d(\Omega;\mathbb F_2),\qquad 1\le d\le N-1.
		\]
		There are a closed connected smooth \(d\)-manifold \(V\), a smooth map
		\[
		h:V\to\Omega,
		\]
		and a compact smooth subdomain \(K\Subset\Omega\) containing \(h(V)\), such that \(h_*[V]\ne0\).  Moreover, there is \(\omega\in H^d(\Omega;\mathbb F_2)\) such that
		\begin{equation}\label{eq6.pairing}
			\langle h^*\omega,[V]\rangle
			=\langle\omega,h_*[V]\rangle=1.
		\end{equation}
	\end{lemma}
	
	\begin{proof}
		Since a singular homology class is represented by a finite cycle, there are a finite polyhedron \(P\), a continuous map \(i:P\to\Omega\), and a class \(\bar\alpha\in H_d(P;\mathbb F_2)\) such that \(i_*\bar\alpha=\alpha\).  By Thom's realization theorem over \(\mathbb F_2\), there are a closed smooth \(d\)-manifold \(V\) and a continuous map \(f:V\to P\) such that \(f_*[V]=\bar\alpha\) \cite[Th\'eor\`eme~III.2]{Thom1954}.  Thus, for \(h_0=i\circ f\),
		\[
		(h_0)_*[V]=\alpha.
		\]
		By the universal coefficient theorem over \(\mathbb F_2\), choose \(\omega\in H^d(\Omega;\mathbb F_2)\) such that \(\langle\omega,\alpha\rangle=1\).  If \(V=\coprod_{\nu=1}^r V_\nu\), then
		\[
		1=\langle\omega,\alpha\rangle
		=\sum_{\nu=1}^r
		\langle(h_0|_{V_\nu})^*\omega,[V_\nu]\rangle.
		\]
		Hence some component \(V_{\nu_0}\) has pairing one.  Replace \(V\) by \(V_{\nu_0}\) and \(h_0\) by its restriction.  The selected component satisfies \(\langle h_0^*\omega,[V]\rangle=1\), which is the property used below.  Since \(\Omega\) is an open smooth manifold, the Whitney approximation theorem gives a smooth map \(h:V\to\Omega\), homotopic to \(h_0\) \cite[Theorem~6.26]{Lee2013}.  Therefore
		\[
		\langle h^*\omega,[V]\rangle
		=\langle h_0^*\omega,[V]\rangle=1,
		\]
		so \(h_*[V]\ne0\) and \eqref{eq6.pairing} holds.  Finally, the compact set \(h(V)\) is contained in the interior of a compact smooth subdomain \(K\Subset\Omega\).
	\end{proof}
	
	For a space \(X\), write
	\[
	\operatorname{Conf}_j(X)
	=\{(x_1,\ldots,x_j)\in X^j:x_r\ne x_s\text{ for }r\ne s\}.
	\]
	Throughout the remainder of this subsection, let \(d\ge1\), let \(V\) be a closed connected smooth \(d\)-manifold, let \(h:V\to\Omega\) be smooth, and let \(\omega\in H^d(\Omega;\mathbb F_2)\) satisfy
	\[
	\langle h^*\omega,[V]\rangle=1.
	\]
	These are precisely the properties of the data from Lemma~\ref{Lem6.1} used in the barycenter argument, which therefore applies to arbitrary smooth \(h\).
	
	\begin{lemma}\label{Lem6.2}
		Let \(d\ge1\).  For every \(\omega\in H^d(\Omega;\mathbb F_2)\) and every \(j\ge1\), there is a canonical class
		\[
		\Theta_j(\omega)\in H^d(\operatorname{SP}^j(\Omega);\mathbb F_2)
		\]
		such that, for the quotient \(\operatorname{sym}_j:\Omega^j\to\operatorname{SP}^j(\Omega)\),
		\begin{equation}\label{eq6.theta-pullback}
			\operatorname{sym}_j^*\Theta_j(\omega)
			=\sum_{i=1}^j\operatorname{pr}_i^*\omega.
		\end{equation}
		On \(\operatorname{Conf}_j(\Omega)/\mathfrak S_j\), this is the cohomological transfer of the pullback of \(\omega\) from the marked particle along the genuine \(j\)-sheeted cover
		\[
		\operatorname{Conf}_j(\Omega)/\mathfrak S_{j-1}
		\to
		\operatorname{Conf}_j(\Omega)/\mathfrak S_j.
		\]
		The same transfer description holds on every spherical Fulton-MacPherson resolution used below.
	\end{lemma}
	
	\begin{proof}
		Use the Dold-Kan simplicial abelian group \(\mathcal K_d=K(\mathbb F_2,d)\) associated with the chain complex having \(\mathbb F_2\) in degree \(d\) and zero elsewhere \cite[Sections~22-24]{May1967}.  Write
		\[
		\upsilon_d\in Z^dN^*(\mathcal K_d;\mathbb F_2)
		\]
		for its universal normalized cocycle, and use the same symbol for the represented universal class in \(H^d(|\mathcal K_d|;\mathbb F_2)\).  Because addition in \(\mathcal K_d\) is simplicial and linear,
		\[
		+^*\upsilon_d
		=\operatorname{pr}_1^*\upsilon_d
		+\operatorname{pr}_2^*\upsilon_d
		\]
		as normalized cochains.  Its geometric realization \(|\mathcal K_d|\) is a topological abelian group.  Choose a continuous classifying map
		\[
		f:\Omega\to|\mathcal K_d|,
		\qquad f^*\upsilon_d=\omega,
		\]
		and define the continuous map
		\[
		S_j:\operatorname{SP}^j(\Omega)\to|\mathcal K_d|,
		\qquad
		S_j([x_1,\ldots,x_j])=f(x_1)+\cdots+f(x_j).
		\]
		Strict commutativity makes \(S_j\) well defined on the symmetric product. Set \(\Theta_j(\omega)=S_j^*\upsilon_d\).  A homotopy between two classifying maps adds to a homotopy after passage to the symmetric product, so the resulting cohomology class is independent of \(f\). Strict additivity gives \eqref{eq6.theta-pullback}.
		
		Separately choose a simplicial classifying representative \(F:\operatorname{Sing}\Omega\to\mathcal K_d\) corresponding to the homotopy class of \(f\), and put
		\[
		\mathbf a_\omega=F^*\upsilon_d
		\in Z^dN^*(\operatorname{Sing}\Omega;\mathbb F_2).
		\]
		On the labelled spherical resolution, the pulled-back normalized cocycle representing the sum class is literally
		\[
		\sum_{i=1}^j(h\circ\operatorname{ev}_i)^*\mathbf a_\omega.
		\]
		For an evenly covered singular simplex in the unmarked configuration quotient, the transfer cochain is the sum of the evaluations on its \(j\) marked lifts.  By the strict normalized-cochain identity above, this is exactly the displayed labelled-resolution sum.  The same statement remains literal on the spherical resolution because the symmetric-group action there is free and the root evaluations extend. The transfer identity applies directly to even-degree covers as well.
	\end{proof}
	
	Put
	\[
	q_j=j(d+1)-1,
	\qquad
	\mathcal U_j
	=\operatorname{Conf}_j(V)
	\mathop{\times}_{\mathfrak S_j}\operatorname{int}\Delta_{j-1}.
	\]
	The sum map
	\[
	\rho_j:\mathcal U_j\to B_j(V)\setminus B_{j-1}(V),
	\qquad
	[z,t]\longmapsto\sum_{i=1}^jt_i\delta_{z_i},
	\]
	is a homeomorphism.  We record the precise homology convention used below.
	
	Choose a finite triangulation of \(V\) and realize it as a compact semialgebraic polyhedron \(P\subset\mathbb R^M\).  For \(|\gamma|\le2j-1\), the moments
	\[
	\mu\longmapsto\left(\int_Px^\gamma\,d\mu(x)\right)_{|\gamma|\le2j-1}
	\]
	separate the elements of \(B_j(P)\): the difference of two such measures has at most \(2j\) support points, and a product of at most \(2j-1\) affine functions isolates any prescribed support point.  Tarski-Seidenberg and compatible semialgebraic triangulation \cite[Sections~2.2 and~9.2]{BochnakCosteRoy1998} therefore make \((B_j(V),B_{j-1}(V))\) a finite polyhedral pair.  In particular, \(B_{j-1}(V)\) is closed in the compact space \(B_j(V)\).  For \(j\ge2\), the quotient \(B_j(V)/B_{j-1}(V)\) is the one-point compactification of \(\mathcal U_j\); for \(j=1\), one uses directly \(B_1(V)=\mathcal U_1=V\) and \(B_0(V)=\varnothing\).  Thus in every case there is a canonical isomorphism
	\begin{equation}\label{eq6.regular-model}
		\kappa_j:
		H_*^{\mathrm{BM}}(\mathcal U_j;\mathbb F_2)
		\xrightarrow{\ \cong\ }
		H_*(B_j(V),B_{j-1}(V);\mathbb F_2).
	\end{equation}
	
	The space \(\mathcal U_j\) is a smooth manifold without boundary of dimension \(q_j\).  It is connected.  For \(d\ge2\), one moves points one at a time along paths avoiding the remaining finite set.  For \(d=1\), the connected closed manifold \(V\) is a circle; the components of the ordered configuration space are the cyclic orders and \(\mathfrak S_j\) acts transitively on them.  Let \([\mathcal U_j]_{\mathrm{BM}}\) be its mod-\(2\) Borel-Moore fundamental class \cite{BorelMoore1960} and define
	\[
	\beta_j
	=\kappa_j[\mathcal U_j]_{\mathrm{BM}}
	\in H_{q_j}(B_j(V),B_{j-1}(V);\mathbb F_2).
	\]
	
	Let
	\[
	\zeta_j:\mathcal U_j\to\operatorname{SP}^j(V),
	\qquad [z,t]\longmapsto[z_1,\ldots,z_j],
	\]
	and define
	\[
	\theta_{V,j}
	=\zeta_j^*\operatorname{SP}^j(h)^*\Theta_j(\omega)
	\in H^d(\mathcal U_j;\mathbb F_2).
	\]
	For \(z\in H_*(B_j(V),B_{j-1}(V))\), use the Borel-Moore cap product and the convention
	\begin{equation}\label{eq6.topological-cap}
		\theta_{V,j}\frown z
		=\kappa_j\bigl(\theta_{V,j}\frown\kappa_j^{-1}z\bigr).
	\end{equation}
	For \(j\ge2\), let \(\partial_{V,j}\) be the connecting homomorphism of the triple
	\[
	B_{j-2}(V)\subset B_{j-1}(V)\subset B_j(V),
	\]
	and let \(e_V\) be the generator of \(H_0(V;\mathbb F_2)\).
	
	\begin{lemma}\label{Lem6.3}
		For every finite \(J_{\max}\), and simultaneously for \(1\le j\le J_{\max}\), there are compact triangulated manifolds with corners
		\[
		\widehat{\mathcal M}^{\mathrm{lab}}_j(V)
		=\operatorname{FM}^{\mathrm{sph}}_j(V)\times\Delta_{j-1},
		\qquad
		\widehat{\mathcal M}_j(V)
		=\widehat{\mathcal M}^{\mathrm{lab}}_j(V)/\mathfrak S_j,
		\]
		maps
		\[
		\pi_j:\widehat{\mathcal M}_j(V)\to B_j(V),
		\qquad
		\widehat\zeta_j:\widehat{\mathcal M}_j(V)
		\to\operatorname{SP}^j(V),
		\]
		and mod-\(2\) quotient fundamental chains
		\[
		\mathfrak B_j\in
		C_{q_j}(\widehat{\mathcal M}_j(V);\mathbb F_2).
		\]
		The labelled resolution carries the simplex weights \(t_i\) and extended root evaluations \(\operatorname{ev}_i\), permuted by \(\mathfrak S_j\), and
		\[
		\pi_j[z,t]=\sum_i t_i\delta_{\operatorname{ev}_i(z)},
		\qquad
		\widehat\zeta_j[z,t]
		=[\operatorname{ev}_1(z),\ldots,\operatorname{ev}_j(z)].
		\]
		For \(j\ge2\), the closed face unions
		\[
		\widehat{\mathcal M}_{j,r}
		=\pi_j^{-1}(B_{j-r}(V)),\qquad r=1,2,
		\]
		satisfy
		\[
		\partial\mathfrak B_j
		\in C_{q_j-1}(\widehat{\mathcal M}_{j,1};\mathbb F_2),
		\]
		the resolution is the identity over \(\mathcal U_j\), and
		\[
		(\pi_j)_*[\mathfrak B_j]=\beta_j
		\quad\hbox{in }H_{q_j}(B_j(V),B_{j-1}(V)).
		\]
	\end{lemma}
	
	\begin{proof}
		Use the spherical, or oriented, Fulton-MacPherson compactification \cite{FultonMacPherson1994}: a collision screen is taken modulo translation and positive dilation, with antipodal directions kept distinct.  The extended root projections, the manifold-with-corners charts, and the free symmetric-group action are given by \cite[Theorems~4.2, 4.4 and~4.10]{Sinha2004}.  In particular, the quotient is a compact manifold with corners of dimension \(jd+j-1=q_j\).  The use of the spherical compactification is essential when \(d=1\): the transposition on a binary \(S^0\)-screen exchanges its two points and remains free.
		
		On a boundary stratum, let \(\mathcal P\) be the partition of the labels into maximal blocks with common root evaluation and let \(Z=\{i:t_i=0\}\).  The support cardinality and support defect are
		\begin{equation}\label{eq6.support-defect}
			\begin{split}
				s(\mathcal P,Z)
				&=\#\{C\in\mathcal P:C\not\subset Z\},\\
				r(\mathcal P,Z)
				&=j-s(\mathcal P,Z)\\
				&=|Z|+
				\sum_{\substack{C\in\mathcal P\\C\not\subset Z}}
				\bigl(|C\setminus Z|-1\bigr).
			\end{split}
		\end{equation}
		This proves that \(\widehat{\mathcal M}_{j,r}\) is a closed union of faces.  It also shows that the filtration is indexed by support loss, independently of corner codimension.  In particular, a corner consisting of one zero-weight ghost colliding with one active particle has support defect one, whereas a three-particle collision divisor already has support defect two.
		
		Choose face-compatible triangulations of the finitely many quotients and let \(\mathfrak B_j\) be the sum of their top-dimensional simplices.  It must be defined downstairs: pushing the ordered fundamental chain to the quotient would multiply it by \(j!\), which vanishes in \(\mathbb F_2\) for \(j\ge2\).  Interior facets cancel in pairs and the remaining facets lie in \(\widehat{\mathcal M}_{j,1}\).  Finally, both
		\[
		\widehat{\mathcal M}_j/\widehat{\mathcal M}_{j,1}
		\quad\hbox{and}\quad
		B_j(V)/B_{j-1}(V)
		\]
		are the one-point compactification of \(\mathcal U_j\), and the map induced by \(\pi_j\) is the identity there.  It therefore sends the relative fundamental class to \(\beta_j\).
	\end{proof}
	
	Set
	\[
	u_j
	=\widehat\zeta_j^*\operatorname{SP}^j(h)^*\Theta_j(\omega)
	\in H^d(\widehat{\mathcal M}_j(V);\mathbb F_2).
	\]
	Using the strictly additive cocycle in the proof of Lemma~\ref{Lem6.2}, choose a normalized cocycle
	\[
	\mathbf u_j\in
	Z^dN^*(\operatorname{Sing}\widehat{\mathcal M}_j(V);\mathbb F_2)
	\]
	representing \(u_j\).  On the labelled resolution its pullback is literally
	\begin{equation}\label{eq6.strict-sum-cocycle}
		\sum_{i=1}^j(h\circ\operatorname{ev}_i)^*\mathbf a_\omega,
	\end{equation}
	where \(\mathbf a_\omega=F^*\upsilon_d\) is the normalized cocycle chosen in the proof of Lemma~\ref{Lem6.2}.
	
	For the chain calculations below, work over \(\mathbb F_2\) and use the Alexander-Whitney cap product on a common compatible subdivision.  If \(r:(\widetilde X,\widetilde A)\to(X,A)\) is a finite covering of triangulated pairs, let \(\operatorname{tr}_r c\) denote the sum of all lifts of \(c\).  For \(\widetilde a\in C^d(\widetilde X)\), define \(r_!\widetilde a\) on a simplex as the sum of the values of \(\widetilde a\) on all its lifts.
	
	\begin{lemma}\label{Lem6.4}
		If \(a\in Z^d(X)\), \(b\in C^{d-1}(X)\), and \(c\in C_q(X)\), then
		\begin{equation}\label{eq6.cap-boundary-bookkeeping}
			\partial(a\frown c)=a\frown\partial c,
			\qquad
			(a+\delta b)\frown c-a\frown c
			=\partial(b\frown c)+b\frown\partial c.
		\end{equation}
		Consequently, if \(c\) is a relative cycle modulo \(A\subset X\), then cohomologous cocycles give the same relative cap class modulo \(A\).
		
		For the covering \(r\), the projection formula holds literally:
		\begin{equation}\label{eq6.cap-transfer-bookkeeping}
			r_\#\bigl(\widetilde a\frown\operatorname{tr}_r c\bigr)
			=(r_!\widetilde a)\frown c.
		\end{equation}
		These statements remain valid for face-compatible subdivisions and relative pairs.  Finally, if \(M\) is a compact manifold with corners, \(e:M\to V\) is a submersion on every face, and \(x\in V\), then the cap of the relative fundamental chain of \(M\) by the pullback of the mod-\(2\) point-dual of \(x\) is represented by the relative fundamental chain of the neat submanifold \(e^{-1}(x)\).
	\end{lemma}
	
	\begin{proof}
		The two identities in \eqref{eq6.cap-boundary-bookkeeping} are the standard cap-boundary formula \cite[Section~3.3]{Hatcher02}; all signs disappear over \(\mathbb F_2\).  The second identity shows directly that the change of a capped relative cycle is an ordinary boundary plus a chain supported in \(A\).  Formula \eqref{eq6.cap-transfer-bookkeeping} is checked on one simplex: its lifted front faces are exactly the front faces on which \(r_!\widetilde a\) is evaluated.  Summing the simplex identities proves the formula.  Subdivision and simplicial approximation preserve these relative classes by the same cap-boundary identity and the standard subdivision chain homotopy.  For the last assertion, triangulate compatibly with all faces and with the transverse inverse image.  The simplicial intersection chain defining the pulled-back point-dual is the fundamental chain of \(e^{-1}(x)\); its boundary is its intersection with \(\partial M\).
	\end{proof}
	
	\begin{lemma}\label{Lem6.5}
		For every \(j\ge1\),
		\[
		\beta_j\ne0
		\quad\text{in}\quad
		H_{q_j}(B_j(V),B_{j-1}(V);\mathbb F_2).
		\]
		For \(j\ge2\),
		\begin{equation}\label{eq6.bc-identity}
			\partial_{V,j}\bigl(\theta_{V,j}\frown\beta_j\bigr)
			=\beta_{j-1},
		\end{equation}
		while
		\begin{equation}\label{eq6.base-cap}
			\theta_{V,1}\frown\beta_1=e_V.
		\end{equation}
	\end{lemma}
	
	\begin{proof}
		Mod-\(2\) Borel-Moore Poincar\'e duality on the connected \(q_j\)-manifold \(\mathcal U_j\) gives
		\[
		H_{q_j}^{\mathrm{BM}}(\mathcal U_j;\mathbb F_2)
		\cong H^0(\mathcal U_j;\mathbb F_2)
		\cong\mathbb F_2.
		\]
		Thus \([\mathcal U_j]_{\mathrm{BM}}\ne0\), and the isomorphism \eqref{eq6.regular-model} gives \(\beta_j\ne0\).
		
		We prove the connecting formula directly on the marked resolution. Let
		\[
		\widehat{\mathcal M}^{\bullet}_j(V)
		=\widehat{\mathcal M}^{\mathrm{lab}}_j(V)/\mathfrak S_{j-1},
		\]
		where the first label is marked, and let
		\[
		r_j:\widehat{\mathcal M}^{\bullet}_j(V)
		\to\widehat{\mathcal M}_j(V)
		\]
		be the quotient map.  Freeness of the spherical action makes \(r_j\) a genuine \(j\)-sheeted cover on the whole resolution, including collision faces.  Let \(\mathfrak B_j^\bullet\) be the sum of all lifts of the top-dimensional quotient simplices.  For an evenly covered simplex, the transfer of \((h\circ\operatorname{ev}_1)^*\mathbf a_\omega\) is the sum over the \(j\) possible marked lifts.  By \eqref{eq6.strict-sum-cocycle} and the projection formula \eqref{eq6.cap-transfer-bookkeeping}, this gives the literal chain identity
		\begin{equation}\label{eq6.marked-transfer-cap}
			r_{j\#}\bigl(
			(h\circ\operatorname{ev}_1)^*\mathbf a_\omega\frown\mathfrak B_j^\bullet
			\bigr)
			=\mathbf u_j\frown\mathfrak B_j.
		\end{equation}
		The marked quotient therefore contributes coefficient one, while the uncapped pushforward of the marked fundamental chain would have coefficient \(j\).
		
		Since \(V\) is closed and connected and \(\langle h^*\omega,[V]\rangle=1\), the class \(h^*\omega\) is the nonzero element of \(H^d(V;\mathbb F_2)\).  Fix \(x\in V\) and represent this class by the mod-\(2\) Poincar\'e dual of \(x\).  Replacing \(h^*\mathbf a_\omega\) by that point-dual cocycle changes \eqref{eq6.marked-transfer-cap} by the boundary and resolved-boundary terms in \eqref{eq6.cap-boundary-bookkeeping}; hence the relative cap class and its connecting image remain invariant.
		
		The marked root evaluation
		\[
		\operatorname{ev}_1:\widehat{\mathcal M}^{\bullet}_j(V)\to V
		\]
		is a submersion on every face: in the Fulton-MacPherson local charts the common root position is an independent \(V\)-coordinate.  Consequently
		\[
		Z_j=\operatorname{ev}_1^{-1}(x)
		\]
		is a neat codimension-\(d\) submanifold with corners.  By Lemma~\ref{Lem6.4}, the marked cap class is represented by its fundamental chain.  Its dimension is
		\[
		q_j-d=q_{j-1}+1.
		\]
		
		Define the fixed-point star
		\begin{equation}\label{eq6.fixed-point-star}
			A_{x,j-1}
			=B_{j-2}(V)\cup
			\{\mu\in B_{j-1}(V):x\in\operatorname{supp}\mu\}.
		\end{equation}
		For \(j=2\), this is the single point \(\delta_x\).  For \(j\ge3\), it is the image of
		\[
		[0,1]\times B_{j-2}(V)\to B_{j-1}(V),
		\qquad
		(s,\nu)\longmapsto s\delta_x+(1-s)\nu.
		\]
		It is a closed subpolyhedron after a compatible subdivision, and
		\begin{equation}\label{eq6.fixed-star-dimension}
			\dim A_{x,j-1}
			\le q_{j-2}+1=q_{j-1}-d<q_{j-1}.
		\end{equation}
		Therefore
		\begin{equation}\label{eq6.fixed-star-vanishing}
			H_{q_{j-1}}(A_{x,j-1},B_{j-2}(V);\mathbb F_2)=0.
		\end{equation}
		
		We now inspect \(\partial Z_j\).  On the face \(t_1=0\), the marked point is forgotten.  Away from \(A_{x,j-1}\), the remaining \(j-1\) particles are distinct and avoid \(x\), and the face map has exactly one preimage on the marked quotient.  Thus it has local mod-\(2\) degree one onto the open \((j-1)\)-particle stratum away from the fixed-point star.
		
		Every other codimension-one face maps into \(A_{x,j-1}\).  Indeed, on an unmarked zero-weight face the marked point \(x\) remains active.  On an intrinsic collision face, either \(x\) remains in the support with at most \(j-2\) other support points, or the marked weight is zero and the image has at most \(j-2\) support points.  The same support alternative applies to all intersections and nested collision corners.  In particular, at a ghost-active corner with \(t_1=0\), the active particle colliding with the ghost is located at \(x\); the image can still have \(j-1\) support points, but the fixed point \(x\) places it in \(A_{x,j-1}\).
		
		The degree-one statement and the support description of all other faces show that the pushed boundary of the marked capped chain and \(\beta_{j-1}\) have the same image in
		\[
		H_{q_{j-1}}(B_{j-1}(V),A_{x,j-1};\mathbb F_2).
		\]
		Equivalently, choose the fundamental chain \([Z_j]\) on a subdivision compatible with Lemma~\ref{Lem6.3}, and put \(\varpi_j=\pi_j\circ r_j\).  By the definition of relative homology there are chains
		\[
		Q_j\in C_{q_{j-1}+1}(B_{j-1}(V);\mathbb F_2),
		\qquad
		R_j\in C_{q_{j-1}}(A_{x,j-1};\mathbb F_2)
		\]
		such that
		\begin{equation}\label{eq6.relative-degree-chain}
			(\varpi_j)_\#\partial[Z_j]
			= (\pi_{j-1})_\#\mathfrak B_{j-1}+\partial Q_j+R_j.
		\end{equation}
		The exact sequence of the triple
		\[
		B_{j-2}(V)\subset A_{x,j-1}\subset B_{j-1}(V)
		\]
		and \eqref{eq6.fixed-star-vanishing} show that the natural map
		\[
		H_{q_{j-1}}(B_{j-1}(V),B_{j-2}(V))
		\to
		H_{q_{j-1}}(B_{j-1}(V),A_{x,j-1})
		\]
		is injective.  Hence the two classes already agree modulo \(B_{j-2}(V)\), proving \eqref{eq6.bc-identity}.  This marked-point calculation is the dimension-independent mechanism underlying \cite[Appendix~C, Proposition~C.1]{BahriCoron1988}.
		
		The argument also covers codimension-one intrinsic collisions.  When \(d=1\), a binary collision has an \(S^0\)-screen and is a genuine boundary face; the same fixed-point-star inclusion and the strict inequality in \eqref{eq6.fixed-star-dimension} apply.
		
		For \(j=1\), \(\mathcal U_1=V\), \(q_1=d\), and
		\[
		\theta_{V,1}\frown\beta_1
		=(h^*\omega)\frown[V]
		=\langle h^*\omega,[V]\rangle e_V=e_V.
		\]
	\end{proof}
	
	\begin{lemma}\label{Lem6.6}
		The cocycle \(\mathbf u_j\) and the fundamental chain \(\mathfrak B_j\) may be placed on one common face-compatible subdivision.  For \(j\ge2\), there are singular chains
		\[
		\mathfrak Q_j\in
		C_{q_{j-1}+1}(B_{j-1}(V);\mathbb F_2),
		\qquad
		\mathfrak R_j\in
		C_{q_{j-1}}(B_{j-2}(V);\mathbb F_2)
		\]
		such that
		\begin{equation}\label{eq6.filtered-chain-identity}
			\pi_{j\#}\partial(\mathbf u_j\frown\mathfrak B_j)
			-\pi_{j-1,\#}\mathfrak B_{j-1}
			=\partial\mathfrak Q_j+\mathfrak R_j,
			\qquad
			\operatorname{supp}\mathfrak R_j\subset B_{j-2}(V).
		\end{equation}
		For \(j=1\), \(\mathbf u_1\frown\mathfrak B_1\) represents \(e_V\).
	\end{lemma}
	
	\begin{proof}
		By \eqref{eq6.cap-boundary-bookkeeping}, \(\partial(\mathbf u_j\frown\mathfrak B_j) =\mathbf u_j\frown\partial\mathfrak B_j\), because \(\delta\mathbf u_j=0\).  The complete degree audit is
		\[
		\begin{array}{c|ccccc}
			&\mathfrak B_j&\mathbf u_j\frown\mathfrak B_j
			&\partial(\mathbf u_j\frown\mathfrak B_j)
			&\mathfrak Q_j&\mathfrak R_j\\ \hline
			\deg&q_j&q_j-d=q_{j-1}+1&q_{j-1}&q_{j-1}+1&q_{j-1}.
		\end{array}
		\]
		On the common open stratum, \(\mathbf u_j\) represents \(\theta_{V,j}\), while \(\mathfrak B_j\) represents \(\beta_j\). Cap naturality and the one-point-compactification identifications therefore give
		\[
		\left[\pi_{j\#}(\mathbf u_j\frown\mathfrak B_j)\right]
		=\theta_{V,j}\frown\beta_j
		\quad\text{in }H_{q_{j-1}+1}(B_j(V),B_{j-1}(V)).
		\]
		Its connecting boundary is represented by \(\pi_{j\#}\partial(\mathbf u_j\frown\mathfrak B_j)\).  By Lemma~\ref{Lem6.5}, this has the same relative class as \(\pi_{j-1,\#}\mathfrak B_{j-1}\) in
		\[
		H_{q_{j-1}}(B_{j-1}(V),B_{j-2}(V)).
		\]
		The definition of equality in relative homology now supplies chains \(\mathfrak Q_j\) and \(\mathfrak R_j\) of the stated degrees and gives \eqref{eq6.filtered-chain-identity} literally.
		
		Only a chain whose actual \(\pi_j\)-image lies in \(B_{j-2}(V)\) is put in \(\mathfrak R_j\).  A redundant zero-weight/collision corner may have geometric codimension two but support defect one; such a term belongs to the \(\partial\mathfrak Q_j\) bookkeeping.  This distinction is essential for the ghost-active corner and remains essential when \(d=1\).  For \(j=2\), the convention \(B_0(V)=\varnothing\) gives \(\mathfrak R_2=0\).
		
		For \(j=1\), the conclusion is exactly \eqref{eq6.base-cap} on the resolved model \(\widehat{\mathcal M}_1(V)=V\).
	\end{proof}
	
	\medskip
	\noindent We next account for off-diagonal self-collisions.  The preceding filtration is intrinsic to the barycenter space of \(V\). If
	\[
	z_r\ne z_s,
	\qquad
	h(z_r)=h(z_s),
	\]
	then the source barycenter still has \(j\) distinct support points and the point lies in the interior of the resolved source, outside \(\widehat{\mathcal M}_{j,1}\).  Consequently \eqref{eq6.filtered-chain-identity} remains unchanged.  After composition with \(h\), the corresponding analytic centers coincide.  In Lemma~\ref{Lem6.7} all such points are therefore assigned to the analytic collision collar \(\mathcal L_j\).  A PL cut may create an artificial interface, but that interface occurs in both cut boundaries and is absorbed by the ordinary lower-sublevel boundary \(\partial(\mathbf u_j\frown\mathfrak L_j)\).  Thus off-diagonal self-collisions affect the analytic collar, while the topological connecting boundary remains governed by intrinsic support loss.  This treatment applies to self-collision loci of arbitrary regularity and codimension.
	
	For the next result, fix \(j\ge2\) and assume that the \(L_a\)-problem has no positive solution. Let \(V\) be a closed manifold, let \(h:V\to K\Subset\Omega\) be smooth, and choose the modified core data \(0<\tau_j\le\tau_{j,*}\), \(D_j\), and \(\Lambda_j\) as in Lemma~\ref{Lem5.18}.  Fix \(\lambda\ge\Lambda_j\), choose regular levels satisfying
	\[
	\ell_j^{\mathrm{low}}(\lambda)<c_{j-1}<b_j<c_j<b_{j+1},
	\]
	and assume
	\[
	g_{a,\lambda,j}\bigl(B_j(h)(B_j(V))\bigr)\subset W_{a,j}.
	\]
	Let \(\widehat{\mathcal M}_j(V)\), \(\mathfrak B_j\), and \(\pi_j\) be as in Lemma~\ref{Lem6.3}, and set
	\[
	\widehat g_j=g_{a,\lambda,j,h}\circ\pi_j.
	\]
	On \(\widehat{\mathcal M}^{\mathrm{lab}}_j(V)\), write \(z_i=\operatorname{ev}_i\) for the evaluation points.  The invariant subsets
	\[
	\begin{aligned}
		\mathcal C_j
		&=\left\{
		\max_i|jt_i-1|<\tau_j,\quad
		\lambda|h(z_r)-h(z_s)|>D_j\ (r\ne s)
		\right\},\\
		\mathcal L_j
		&=\left\{
		\max_i|jt_i-1|>\tau_j/2
		\ \hbox{or}\
		\lambda|h(z_r)-h(z_s)|<2D_j\ \hbox{for some }r\ne s
		\right\}
	\end{aligned}
	\]
	descend to \(\widehat{\mathcal M}_j(V)\).
	
	\begin{lemma}\label{Lem6.7}
		After a finite cellular subdivision, there are chains \(\mathfrak P_j,\mathfrak L_j\in C_{q_j}(\widehat{\mathcal M}_j(V);\mathbb F_2)\) such that
		\begin{equation}\label{eq6.chain-cut}
			\mathfrak B_j=\mathfrak P_j+\mathfrak L_j,
		\end{equation}
		\[
		\operatorname{supp}\mathfrak P_j\subset\mathcal C_j,
		\qquad
		\operatorname{supp}\mathfrak L_j\subset\mathcal L_j,
		\]
		and a chain
		\[
		\Gamma_j\in
		C_{q_j-1}(\widehat{\mathcal M}_j(V);\mathbb F_2)
		\]
		such that
		\[
		\operatorname{supp}\Gamma_j
		\subset\mathcal C_j\cap\mathcal L_j,
		\qquad
		\partial\mathfrak P_j=\Gamma_j,
		\qquad
		\partial\mathfrak L_j=\partial\mathfrak B_j+\Gamma_j.
		\]
		These are identities of absolute mod-\(2\) chains.  Moreover,
		\[
		\widehat g_j(\operatorname{supp}\mathfrak P_j)\subset X_j,
		\qquad
		\widehat g_j(\operatorname{supp}\mathfrak L_j)\subset W_{a,j-1},
		\qquad
		\widehat g_j(\operatorname{supp}\Gamma_j)\subset A_j.
		\]
		On \(\operatorname{supp}\mathfrak P_j\),
		\begin{equation}\label{eq6.center-square}
			\chi_{a,j}\circ\widehat g_j
			=
			\operatorname{SP}^j(h)\circ\widehat\zeta_j.
		\end{equation}
	\end{lemma}
	
	\begin{proof}
		The resolution retains all labels, weights, and limiting evaluation points.  Because \(\tau_{j,*}<1/2\), every zero-weight face satisfies \(\max_i|jt_i-1|>\tau_j/2\).  On every Fulton-MacPherson collision face, there are \(r\ne s\) with \(\operatorname{ev}_r=\operatorname{ev}_s\), and hence
		\[
		\lambda
		\left|h(\operatorname{ev}_r)-h(\operatorname{ev}_s)\right|
		=0<2D_j.
		\]
		Thus the entire natural boundary lies in the interior of \(\mathcal L_j\).  The same holds for every off-diagonal \(h\)-self-collision.  Here the two notions play different roles.  An intrinsic collision means \(z_r=z_s\); it is a Fulton-MacPherson boundary point and its barycenter has lower support, so it belongs to the face filtration of Lemma~\ref{Lem6.3}.  An off-diagonal \(h\)-self-collision means
		\[
		z_r\ne z_s,\qquad h(z_r)=h(z_s).
		\]
		It is generally an interior point of the resolved configuration space, and \(\pi_j\) retains \(j\)-point support there.  Its coincident analytic centers place it in \(\mathcal L_j\), while the support-loss filtration remains indexed by intrinsic collisions and zero weights.  The preceding boundary description includes nested collision faces and corners at which a zero weight is redundant with a collision.  Every point of \(\mathcal C_j\) has positive weights and pairwise distinct intrinsic evaluations.  The two descended sets form an open cover of the compact resolved model.
		
		Choose a partition of unity \(\rho_{\mathrm C}+\rho_{\mathrm L}=1\) subordinate to this cover such that
		\[
		\operatorname{supp}\rho_{\mathrm C}\subset\mathcal C_j,\qquad
		\operatorname{supp}\rho_{\mathrm L}\subset\mathcal L_j,
		\]
		and \(\rho_{\mathrm C}=0\) on a neighborhood of the natural boundary. After a preliminary subdivision, choose disjoint closed subpolyhedral neighborhoods \(P_0,P_1\) such that \(P_0\) contains both the natural boundary and the complement of \(\mathcal C_j\), while \(P_1\) contains the complement of \(\mathcal L_j\), and
		\[
		\rho_{\mathrm C}|_{P_0}=0,
		\qquad
		\rho_{\mathrm C}|_{P_1}=1.
		\]
		The relative PL approximation theorem gives a PL function \(\widetilde\rho_{\mathrm C}\), fixed at \(0\) on \(P_0\) and at \(1\) on \(P_1\), arbitrarily close to \(\rho_{\mathrm C}\).  Perturb its remaining vertex values, still relative to \(P_0\cup P_1\), choosing every vertex value different from \(1/2\).  Then
		\[
		\{\widetilde\rho_{\mathrm C}\ge1/2\}\subset\mathcal C_j,
		\qquad
		\{\widetilde\rho_{\mathrm C}\le1/2\}\subset\mathcal L_j,
		\qquad
		\{\widetilde\rho_{\mathrm C}=1/2\}
		\subset\mathcal C_j\cap\mathcal L_j,
		\]
		and the natural boundary lies strictly on the second side.  Replace \(\rho_{\mathrm C}\) by this PL function, reset \(\rho_{\mathrm L}=1-\rho_{\mathrm C}\), and subdivide along
		\[
		H_j=\{\rho_{\mathrm C}=1/2\}.
		\]
		Take the resulting triangulation to be a common refinement, for all the finitely many indices under consideration, of the face filtrations and the triangulations previously used.  On this refinement, continue to denote by \(\mathfrak B_j\) the sum of all \(q_j\)-simplices representing the relative fundamental class. Let \(\mathfrak P_j\) be the sum of the \(q_j\)-cells on the \(\rho_{\mathrm C}\ge1/2\) side, let \(\mathfrak L_j\) be the sum on the \(\rho_{\mathrm C}\le1/2\) side, and let \(\Gamma_j\) be the sum of the \((q_j-1)\)-cells in \(H_j\).  Interior faces cancel in pairs, the artificial interface occurs once in each cut boundary, and the natural boundary occurs only on the collar side.  Therefore
		\[
		\mathfrak B_j=\mathfrak P_j+\mathfrak L_j,\qquad
		\partial\mathfrak P_j=\Gamma_j,\qquad
		\partial\mathfrak L_j=\partial\mathfrak B_j+\Gamma_j.
		\]
		The support properties of the partition give the three asserted support inclusions.
		
		Lemma~\ref{Lem5.18} maps \(\mathcal C_j\) into \(X_j\), maps \(\mathcal L_j\) into \(W_{a,j-1}\), and gives the exact center identity on \(\mathcal C_j\).  On the interface, the image lies in
		\[
		X_j\cap W_{a,j-1}
		=W_{a,j-1}\cap
		\mathcal U_{a,j}(\rho_{j,\mathrm{mod}})
		=A_j.
		\]
		This proves the analytic support assertions.  On the common subdivision, choose anew a cellular cocycle \(\mathbf u_j\) representing \(u_j\). The subdivision map and a simplicial approximation to the identity are chain homotopic through homotopies preserving the filtration subcomplexes, and the Alexander-Whitney cap product is compatible with these maps only up to the corresponding chain homotopy.  Thus the new capped chain represents the same relative homology class, which is the identification used below.  Reapply the relative-homology argument of Lemma~\ref{Lem6.6} on this common subdivision and rechoose \(\mathfrak Q_j,\mathfrak R_j\).  Then \eqref{eq6.filtered-chain-identity}, with \(\operatorname{supp}\mathfrak R_j\subset B_{j-2}(V)\), and the literal absolute cut-chain identities above hold simultaneously for the same \(\mathfrak B_j\) and \(\mathbf u_j\).
	\end{proof}
	
	\subsection{The global analytic cap module}
	
	Let \((X_j,A_j)\) and \(\iota_j\) be as in \eqref{eq5.core-pair-new}.  Since \(\chi_{a,j}\) takes values in the full symmetric product, the class
	\[
	\xi_{a,j}=(\chi_{a,j}|_{X_j})^*\Theta_j(\omega)
	\in H^d(X_j;\mathbb F_2)
	\]
	is defined on the entire balanced core, including repeated centers. Transporting its ordinary cap product through \(\iota_j\) defines
	\[
	\theta_{a,j}\frown z
	=
	(\iota_j)_*\left(
	\xi_{a,j}\frown (\iota_j)_*^{-1}z
	\right),
	\qquad z\in H_*(W_{a,j},W_{a,j-1}).
	\]
	Thus the analytic operation is the cap product with an absolute cohomology class on the balanced core.
	
	For \(j\ge2\), let
	\[
	\partial_{a,j}:
	H_k(W_{a,j},W_{a,j-1})
	\to
	H_{k-1}(W_{a,j-1},W_{a,j-2})
	\]
	be the connecting homomorphism of the sublevel triple.
	
	\begin{proposition}
		\label{Prop6.8}
		Let \(V,h,K,\omega\) be as in Lemma~\ref{Lem6.1}, let \(\beta_j\) be the classes of Lemma~\ref{Lem6.5}, and let \((X_j,A_j)\), \(\iota_j\), and \(\theta_{a,j}\) be the analytic core data defined above.  Fix \(m\ge1\), one common scale \(\lambda\), and one simultaneous family of regular levels.  Assume that, for every \(1\le j\le m\), the canonical core pair is the pair \eqref{eq5.core-pair-new}, Lemmas~\ref{Lem5.17} and \ref{Lem5.18} apply at that scale and those levels, and
		\[
		g_{a,\lambda,j,h}:
		(B_j(V),B_{j-1}(V))\to
		(W_{a,j},W_{a,j-1})
		\]
		is a map of pairs.  Then, for \(2\le j\le m\),
		\begin{equation}\label{eq6.module-diagram}
			\partial_{a,j}\!\left(
			\theta_{a,j}\frown(g_{a,\lambda,j,h})_*\beta_j
			\right)
			=
			(g_{a,\lambda,j-1,h})_*\beta_{j-1}.
		\end{equation}
		At the bottom stratum,
		\begin{equation}\label{eq6.module-base}
			\theta_{a,1}\frown(g_{a,\lambda,1,h})_*\beta_1
			=
			(g_{a,\lambda,1,h})_*e_V.
		\end{equation}
	\end{proposition}
	
	\begin{proof}
		For \(j\ge2\), use the chains \(\mathfrak P_j,\mathfrak L_j,\Gamma_j\) from Lemma~\ref{Lem6.7}.  For \(j=1\), put
		\[
		\mathfrak P_1=\mathfrak B_1,\qquad
		\mathfrak L_1=0.
		\]
		Write
		\[
		\widehat g_j=g_{a,\lambda,j,h}\circ\pi_j.
		\]
		The support identities imply that \(\widehat g_{j\#}\mathfrak P_j\) is a relative cycle in \((X_j,A_j)\), and
		\[
		(\iota_j)_*
		\left[\widehat g_{j\#}\mathfrak P_j\right]
		=
		(g_{a,\lambda,j,h})_*\beta_j
		\]
		in \(H_{q_j}(W_{a,j},W_{a,j-1})\).  Since \((\iota_j)_*\) is an isomorphism, this is exactly the inverse image used in the definition of \(\theta_{a,j}\).
		
		Let \(\boldsymbol\xi_{a,j}\) be a cocycle on \(X_j\) representing \(\xi_{a,j}\), and use the cocycle \(\mathbf u_j\) fixed in Lemma~\ref{Lem6.6}.  On the subcomplex supporting \(\mathfrak P_j\), the center identity \eqref{eq6.center-square} gives
		\[
		[\mathbf u_j]
		=
		[\widehat g_j^*\boldsymbol\xi_{a,j}].
		\]
		Cohomologous representatives differ by a coboundary.  Formula \eqref{eq6.cap-boundary-bookkeeping} shows that replacing one by the other changes the capped chain by an ordinary boundary plus a chain supported over \(\partial\mathfrak P_j=\Gamma_j\), whose image lies in \(A_j\). The relative-pair and subdivision assertions of Lemma~\ref{Lem6.4} therefore give
		\begin{equation}\label{eq6.core-cap-chain}
			\theta_{a,j}\frown(g_{a,\lambda,j,h})_*\beta_j
			=
			\left[
			\widehat g_{j\#}
			(\mathbf u_j\frown\mathfrak P_j)
			\right].
		\end{equation}
		This class has degree
		\[
		q_j-d=q_{j-1}+1.
		\]
		
		Its connecting boundary is represented by
		\[
		\partial\widehat g_{j\#}
		(\mathbf u_j\frown\mathfrak P_j)
		\in C_{q_{j-1}}(W_{a,j-1}),
		\]
		because \(\partial(\mathbf u_j\frown\mathfrak P_j) =\mathbf u_j\frown\Gamma_j\) and \(\widehat g_j(\Gamma_j)\subset A_j\). Since \(\mathfrak P_j=\mathfrak B_j+\mathfrak L_j\) over \(\mathbb F_2\),
		\begin{equation}\label{eq6.capped-collar-algebra}
			\begin{aligned}
				\partial\widehat g_{j\#}
				(\mathbf u_j\frown\mathfrak P_j)
				&=
				\partial\widehat g_{j\#}
				(\mathbf u_j\frown\mathfrak B_j)\\
				&\quad+
				\partial\widehat g_{j\#}
				(\mathbf u_j\frown\mathfrak L_j).
			\end{aligned}
		\end{equation}
		The second summand is the boundary of the explicitly defined chain \(\widehat g_{j\#}(\mathbf u_j\frown\mathfrak L_j)\), which is supported in \(W_{a,j-1}\).  It therefore represents zero in
		\[
		H_{q_{j-1}}(W_{a,j-1},W_{a,j-2}).
		\]
		This single boundary contains the artificial interface faces and all their nested intersections.
		
		Apply the test map to \eqref{eq6.filtered-chain-identity}.  Deleting a zero weight or merging equal centers gives the exact restrictions
		\[
		g_{a,\lambda,j,h}|_{B_{j-1}(V)}
		=g_{a,\lambda,j-1,h},
		\qquad
		g_{a,\lambda,j-1,h}(B_{j-2}(V))
		\subset W_{a,j-2}.
		\]
		Consequently,
		\[
		\begin{aligned}
			&\widehat g_{j\#}
			\partial(\mathbf u_j\frown\mathfrak B_j)
			+
			(g_{a,\lambda,j-1,h})_\#
			\pi_{j-1,\#}\mathfrak B_{j-1}\\
			&\qquad=
			\partial\left(
			(g_{a,\lambda,j-1,h})_\#\mathfrak Q_j
			\right)
			+
			(g_{a,\lambda,j-1,h})_\#\mathfrak R_j.
		\end{aligned}
		\]
		The first term on the right is an ordinary boundary in \(W_{a,j-1}\), while the second lies in \(W_{a,j-2}\).  Hence
		\[
		\left[
		\widehat g_{j\#}
		\partial(\mathbf u_j\frown\mathfrak B_j)
		\right]
		=
		(g_{a,\lambda,j-1,h})_*\beta_{j-1}
		\]
		in \(H_{q_{j-1}}(W_{a,j-1},W_{a,j-2})\). Together with \eqref{eq6.core-cap-chain} and \eqref{eq6.capped-collar-algebra}, this proves \eqref{eq6.module-diagram}.
		
		For \(j=1\), \(q_1=d\), the test family lies in \(X_1\), and Lemma~\ref{Lem6.6} gives
		\[
		[\mathbf u_1\frown\mathfrak B_1]=e_V.
		\]
		The same relative-homology cap naturality proves \eqref{eq6.module-base}.
	\end{proof}
	
	\begin{proposition}\label{Prop6.9}
		Fix \(m\ge1\) and \(\lambda>0\).  Suppose that, for every \(1\le j\le m\), the maps \(g_{a,\lambda,j,h}\) are maps of pairs
		\[
		(B_j(V),B_{j-1}(V))
		\to
		(W_{a,j},W_{a,j-1}),
		\]
		and that \eqref{eq6.module-diagram} and \eqref{eq6.module-base} hold.  Then
		\begin{equation}\label{eq6.nonzero}
			(g_{a,\lambda,j,h})_*(\beta_j)\ne0
			\quad\text{in}\quad
			H_{q_j}(W_{a,j},W_{a,j-1})
			\qquad(1\le j\le m).
		\end{equation}
	\end{proposition}
	
	\begin{proof}
		Since \(V\) is connected, \(e_V\) is the generator of \(H_0(V;\mathbb F_2)\).  The space \(W_{a,0}\) is empty, so the image of any point under \(g_{a,\lambda,1,h}\) is a nonzero class in
		\[
		H_0(W_{a,1},W_{a,0};\mathbb F_2)=H_0(W_{a,1};\mathbb F_2).
		\]
		Equation~\eqref{eq6.module-base} therefore implies \((g_{a,\lambda,1,h})_*\beta_1\ne0\).
		
		Assume inductively that \((g_{a,\lambda,j-1,h})_*\beta_{j-1}\ne0\).  If \((g_{a,\lambda,j,h})_*\beta_j=0\), then its cap product with \(\theta_{a,j}\) is zero and so is its connecting boundary.  This contradicts \eqref{eq6.module-diagram}, whose right-hand side is the nonzero inductive class.  Induction proves \eqref{eq6.nonzero} for every \(1\le j\le m\).
	\end{proof}
	
	We retain the notation from Section~\ref{sec:matched}
	\[
	T_j(\lambda)
	=
	\sup_{\mu\in B_j(K)}
	\mathcal J_a(g_{a,\lambda,j}(\mu)),
	\qquad T_0(\lambda)=-\infty,
	\]
	and, for \(\lambda_*\le\Lambda\),
	\[
	\widehat T_j(\lambda_*,\Lambda)
	=
	\sup_{\lambda_*\le\lambda\le\Lambda}T_j(\lambda),
	\qquad
	\widehat T_0(\lambda_*,\Lambda)=-\infty.
	\]
	The first scale is the branch-dependent energy-drop scale; the second is the common scale at which the topological cap module is run.
	
	\begin{lemma}\label{Lem6.10}
		Assume that the \(L_a\)-problem has no positive solution.  Fix \(m\ge2\), a compact smooth subdomain \(K\Subset\Omega\), and a smooth map \(h:V\to K\) as in Lemma~\ref{Lem6.1}.  Make, for \(1\le j\le m\), the level-independent modulation, excision, modified-core, weight-collar, and physical-collision-collar choices in Lemmas~\ref{Lem5.16}, \ref{Lem5.17}, and \ref{Lem5.18}.  There is a common threshold \(\Lambda_{\mathrm{top}}=\Lambda_{\mathrm{top}}(m,K,h)\) with the following property.
		
		Suppose that
		\[
		\lambda_*\le\Lambda,
		\qquad
		\Lambda\ge\Lambda_{\mathrm{top}},
		\]
		and
		\begin{equation}\label{eq6.scale-cylinder-bounds}
			\widehat T_j(\lambda_*,\Lambda)<b_{j+1}
			\quad(1\le j\le m),
			\qquad
			T_m(\lambda_*)<b_m.
		\end{equation}
		Then one can choose, once and simultaneously, regular levels \(c_0,\ldots,c_m\) such that all the hypotheses of Propositions~\ref{Prop6.8}, \ref{Prop6.9}, and \ref{Prop5.19} hold at the topology scale \(\Lambda\).  In particular,
		\begin{equation}\label{eq6.topology-scale-nonzero}
			(g_{a,\Lambda,j,h})_*(\beta_j)\ne0
			\quad\text{in}\quad
			H_{q_j}(W_{a,j},W_{a,j-1};\mathbb F_2)
			\qquad(1\le j\le m).
		\end{equation}
	\end{lemma}
	
	\begin{proof}
		The choices preceding the statement are finite and independent of the exact energy levels.  More explicitly, for each \(j\) choose the retained inner and outer parameter sets and \(\rho_{j,\mathrm{mod}}\), then the excision width, then the level-uniform modified-functional data, and finally the overlapping weight and physical-collision collars.  Let \(\Lambda_{\mathrm{top}}\) dominate the resulting finitely many core-collar thresholds.
		
		Fix scales satisfying \eqref{eq6.scale-cylinder-bounds}. After \(\lambda_*\) and \(\Lambda\) are fixed and before the final regular levels are chosen, take the finitely many common PL subdivisions required by Lemma~\ref{Lem6.7}.  For \(1\le j\le m\), Lemma~\ref{Lem5.18} at \(\Lambda\) supplies
		\[
		\max\{b_{j-1},b_j-\vartheta_{j,*}\}
		<
		\ell_j^{\mathrm{low}}(\Lambda)<b_j.
		\]
		Choose the levels in the following single operation:
		\begin{align}
			\max\{b_r,\widehat T_r(\lambda_*,\Lambda),
			\ell_{r+1}^{\mathrm{low}}(\Lambda)\}
			&<c_r<b_{r+1},
			&&0\le r\le m-2,\label{eq6.simultaneous-levels}\\
			\max\{b_{m-1},\widehat T_{m-1}(\lambda_*,\Lambda),
			\ell_m^{\mathrm{low}}(\Lambda),T_m(\lambda_*)\}
			&<c_{m-1}<b_m,\label{eq6.simultaneous-top-lower}\\
			\max\{b_m,\widehat T_m(\lambda_*,\Lambda)\}
			&<c_m<b_{m+1}.\label{eq6.simultaneous-top-upper}
		\end{align}
		Here \(\widehat T_0=-\infty\).  Every interval is nonempty by \eqref{eq6.scale-cylinder-bounds} and the strict core-collar bounds. The levels can be chosen regular inside these open intervals.
		
		Since
		\[
		T_j(\Lambda)\le\widehat T_j(\lambda_*,\Lambda)<c_j,
		\qquad
		T_{j-1}(\Lambda)
		\le\widehat T_{j-1}(\lambda_*,\Lambda)<c_{j-1},
		\]
		the scale-\(\Lambda\) test maps are maps of pairs, also after composition with \(B_j(h)\).  The inequalities \(c_{j-1}>\ell_j^{\mathrm{low}}(\Lambda)\) give the canonical \(\mathcal U_{a,j}(\rho_{j,\mathrm{mod}})\)-core pair and all collar support statements required by Lemma~\ref{Lem6.7}. Propositions~\ref{Prop6.8} and \ref{Prop6.9} therefore give \eqref{eq6.topology-scale-nonzero} for this one family of levels.
		
		Finally, \eqref{eq6.simultaneous-levels}-\eqref{eq6.simultaneous-top-upper} also give
		\[
		\widehat T_j(\lambda_*,\Lambda)<c_j,\qquad
		\widehat T_{j-1}(\lambda_*,\Lambda)<c_{j-1}
		\quad(1\le j\le m),
		\]
		and
		\[
		\max\{b_{m-1},T_m(\lambda_*)\}<c_{m-1}<b_m.
		\]
		These are exactly the hypotheses of Proposition~\ref{Prop5.19}.
	\end{proof}
	
	\begin{proposition}\label{Prop6.11}
		Assume \(N\in\{3,4,5\}\), \(H_d(\Omega;\mathbb F_2)\ne0\) for some \(1\le d\le N-1\), and that the \(L_a\)-problem has no positive solution.  Let \(V,h,K,\omega\) be supplied by Lemma~\ref{Lem6.1}.  There exist \(m\ge2\), scales
		\[
		\lambda_*\le\Lambda,
		\]
		and one simultaneous family of regular levels \(c_0,\ldots,c_m\) for which \eqref{eq6.topology-scale-nonzero} holds and all hypotheses of Proposition~\ref{Prop5.19} are satisfied.  The scales may be chosen as follows:
		\[
		\begin{array}{c|c|c}
			N&\lambda_*&\Lambda\\ \hline
			3&\Lambda&\text{one sufficiently large fixed-multiplicity scale},\\
			4,5&A\,m^{2/(N-2)}
			&\text{a possibly larger topology scale},
		\end{array}
		\]
		where \(A>0\) is arbitrary in the second line.  In dimensions four and five the strict top energy drop is used only at \(\lambda_*\).
	\end{proposition}
	
	\begin{proof}
		First suppose \(N=3\).  Choose \(m\ge2\) by \eqref{eq5.mchoice}.  Make the finite, level-independent choices in Lemma~\ref{Lem6.10}, obtaining \(\Lambda_{\mathrm{top}}\). Choose auxiliary numbers
		\[
		b_r<\bar c_r<b_{r+1}\quad(0\le r\le m),
		\]
		with \(W_{a,0}=\varnothing\).  Lemma~\ref{Lem5.12}, applied only to these auxiliary levels, gives a scale beyond which
		\[
		T_r(\lambda)<\bar c_r<b_{r+1}
		\qquad(1\le r\le m).
		\]
		Now choose
		\[
		\Lambda\ge\Lambda_{\mathrm{top}}
		\]
		large enough for this estimate and for Proposition~\ref{Prop5.6}, and put \(\lambda_*=\Lambda\).  Then
		\[
		\widehat T_r(\lambda_*,\Lambda)
		=T_r(\Lambda)<b_{r+1},
		\qquad
		T_m(\lambda_*)<b_m.
		\]
		Lemma~\ref{Lem6.10} supplies the asserted single family of final levels. The scale cylinder is constant in this branch.
		
		Now let \(N\in\{4,5\}\) and fix \(A>0\).  Choose \(m\) sufficiently large for Proposition~\ref{Prop5.10} and Corollary~\ref{Cor5.11}, and put
		\[
		\lambda_*=A\,m^{\frac{2}{N-2}}.
		\]
		After \(m\) is fixed, make the finite, level-independent choices in Lemma~\ref{Lem6.10}.  Choose
		\[
		\Lambda\ge\max\{\lambda_*,\Lambda_{\mathrm{top}}\}.
		\]
		Proposition~\ref{Prop5.10} gives
		\[
		T_m(\lambda_*)<b_m,
		\]
		whereas Corollary~\ref{Cor5.11} gives, on the entire cylinder,
		\[
		\widehat T_r(\lambda_*,\Lambda)<b_{r+1}
		\qquad(1\le r\le m).
		\]
		Thus Lemma~\ref{Lem6.10} again supplies one simultaneous final level family.  The induced relative map is transported from \(\lambda_*\) to \(\Lambda\), while the matched-scale inequality remains attached to \(\lambda_*\).
	\end{proof}

	\subsection{Completion of the proof}

	\begin{proof}[Proof of Theorem~\ref{Thm1.1}]
		Assume by contradiction that the hyperbolic problem has no positive solution.  By Lemma~\ref{Lem2.1}, the conformally equivalent Euclidean problem
		\[
		L_av=v^{2^*-1},\qquad v>0\ \text{in }\Omega,\qquad
		v=0\ \text{on }\partial\Omega
		\]
		has no positive solution either.
		
		Choose
		\[
		0\ne\alpha\in H_d(\Omega;\mathbb F_2),
		\qquad 1\le d\le N-1.
		\]
		Lemma~\ref{Lem6.1} gives a smooth map
		\[
		h:V\to K\Subset\Omega
		\]
		and a class \(\omega\) satisfying \(\langle h^*\omega,[V]\rangle=1\). Choose \(m,\lambda_*,\Lambda\), and the single family of sublevels given by Proposition~\ref{Prop6.11}.  Its topological conclusion gives
		\begin{equation}\label{eq6.final-nonzero}
			(g_{a,\Lambda,m,h})_*(\beta_m)\ne0
			\quad\text{in}\quad
			H_{m(d+1)-1}(W_{a,m},W_{a,m-1};\mathbb F_2).
		\end{equation}
		For the same levels and scales, Proposition~\ref{Prop5.19} gives
		\[
		(g_{a,\Lambda,m,h})_*=0
		\quad\text{on}\quad
		H_*(B_m(V),B_{m-1}(V);\mathbb F_2),
		\]
		contradicting \eqref{eq6.final-nonzero}.  In the \(N=3\) branch this is the direct drop at the single scale \(\lambda_*=\Lambda\).  In the \(N=4,5\) branch the matched-scale inequality is applied at \(\lambda_*=A m^{2/(N-2)}\), and the homotopy of pairs transports the zero map to \(\Lambda\).
		
		Thus the Euclidean \(L_a\)-problem has a positive solution \(v\). Lemma~\ref{Lem2.1} then yields the positive solution \(u=\phi^{-1}v\) of the original hyperbolic problem.
	\end{proof}
	\section*{Acknowledgments}
	
	\noindent\textbf{Funding.} This work was supported by the National Natural Science Foundation of China (Grant Nos. 12301145, 12561020, and 12261107) and the Yunnan Fundamental Research Projects (Grant Nos. 202401AU070123 and 202601AT070048). T.~Rana and M.~Ruzhansky were supported by the FWO Odysseus 1 Grant G.0H94.18N (Analysis and Partial Differential Equations) and the Methusalem programme of the Ghent University Special Research Fund (BOF) (Grant Nos. BOFMET2021000601 and 01M01021). T.~Rana was also supported by a BOF postdoctoral fellowship at Ghent University (Grant No. BOF24/PDO/025). M.~Ruzhansky was also supported by EPSRC Grant UKRI3645 and FWO Senior Research Grant G022821N.
	
	\medskip
	\noindent\textbf{Author contributions.} All authors contributed equally to the preparation of the manuscript.
	
	\medskip
	\noindent\textbf{Data availability.} No data were generated or analysed in this study.
	
	\medskip
	\noindent\textbf{Conflict of interest.} The authors declare that they have no conflict of interest.

\end{document}